\documentclass[12pt,a4paper,reqno]{amsart}
\usepackage[margin=1in]{geometry}
\usepackage[T1]{fontenc}

\usepackage{enumerate, mathabx, mathtools}
\usepackage{amsthm, amsmath, amsfonts, amssymb}
\usepackage{cite}
\usepackage{esint} %
\usepackage[normalem]{ulem}
\usepackage{hyperref}

\newcommand{\C}{\mathbb{C}}
\newcommand{\R}{\mathbb{R}}

\newcommand{\T}[0]{\mathbb{T}}

\newcommand{\ld}[0]{\lambda}
\newcommand{\vep}[0]{\epsilon}
\newcommand{\supp}[0]{\operatorname{supp}}
\newcommand{\lsm}[0]{\lesssim}

\newcommand{\wh}[1]{\widehat{#1}}
\newcommand{\mc}[1]{\mathcal{#1}}
\newcommand{\mf}[1]{\mathfrak{#1}}

\newcommand{\wt}[1]{\widetilde{#1}}
\newcommand{\nms}[1]{\| #1 \|}
\newcommand{\norm}[1]{\left\|#1\right\|}

\newtheorem{theorem}{Theorem}[section]

\newtheorem{cor}[theorem]{Corollary}
\newtheorem{defn}[theorem]{Definition}

\newtheorem{clm}[theorem]{Claim}
\newtheorem{conjthm}[theorem]{Plausible Theorem}
\newtheorem{conjcor}[theorem]{Plausible Corollary}

\theoremstyle{remark}
\newtheorem{rem}{Remark}

\begin{document}

\title[Weighted refined decoupling and Mizohata--Takeuchi for moment curve]{A weighted formulation of refined decoupling and inequalities of Mizohata--Takeuchi-type for the moment curve}
\author{Anthony Carbery, Zane Kun Li, Yixuan Pang and Po-Lam Yung}
\address{School of Mathematics and Maxwell Institute for Mathematical Sciences\\
University of Edinburgh\\
Edinburgh EH9 3FD\\
Scotland}
\email{A.Carbery@ed.ac.uk}

\address{Department of Mathematics, North Carolina State University, Raleigh, NC 27695
}
\email{zkli@ncsu.edu}

\address{Department of Mathematics, University of Pennsylvania, Philadelphia, PA 19104}
\email{pyixuan@sas.upenn.edu}

\address{Mathematical Sciences Institute, Australian National University, Canberra, ACT 2601 \& Department of Mathematics, The Chinese University of Hong Kong, Shatin, Hong Kong \& CNRS-ANU International Research Laboratory FAMSI (France Australia Mathematical Sciences and Interactions)}
\email{PoLam.Yung@anu.edu.au, plyung@math.cuhk.edu.hk}

\dedicatory{To Jill Pipher, with admiration and affection.}

\begin{abstract} 
Let $\Gamma$ be a compact patch of a well-curved $C^{n+1}$ curve in $\mathbb{R}^n$ with induced Lebesgue measure ${\rm d} \lambda$, and let $g \mapsto \widehat{g \,{\rm d}\lambda}$ be the Fourier extension operator for $\Gamma$. Then we have, for arbitrary non-negative weights $w$,
\begin{equation*}
        \int_{B_R} |\widehat{g \,{\rm d}\lambda}|^2w \leq C_{n,a} R^{a} \sup_S \left(\int_S w\right)\int_\Gamma |g|^2 \, {\rm d} \lambda
    \end{equation*}
for any $a> \frac{n-3}{2} + \frac{2}{n} - \frac{2}{n^2(n+1)}$, where the $\sup$ is over all $1$-neighbourhoods $S$ of hyperplanes whose normals are parallel to the tangent at some point of $\Gamma$. This represents partial progress on the Mizohata--Takeuchi conjecture for curves in dimensions $n \geq 3$, improving upon the exponent $a=n-1$ which can be obtained as a consequence of the Agmon--H\"ormander trace inequality. Our main tool in establishing this inequality will be a weighted formulation of refined decoupling for well-curved curves. We also discuss the sharpness of the exponents we obtain in this and in auxiliary results, and further explore this in the context of axiomatic decoupling for curves.
\end{abstract}

\maketitle

\section{Introduction}%
Let $\Sigma$ be a compact patch of a $C^2$ hypersurface in $\mathbb{R}^n$, with surface measure ${\rm d}\sigma$ induced by Lebesgue measure in $\mathbb{R}^n$. The Mizohata--Takeuchi conjecture for hypersurfaces addresses the local quantitative behaviour of the Fourier extension operator $g \mapsto \widehat{g \,{\rm d}\sigma}$, in contrast to the global and more classical $L^p(\Sigma)$ -- $L^q(\R^n)$ Fourier extension estimates. One of several possible formulations of the Mizohata--Takeuchi conjecture for hypersurfaces states that
    \begin{equation*}
        \int_{\mathbb{R}^n} |\widehat{g \, {\rm d}\sigma}|^2w \lesssim \|Xw\|_\infty \int_\Sigma |g|^2 {\rm d} \sigma
    \end{equation*}
where $X$ denotes the $X$-ray transform, so that for a line $l$ we have $Xw(l)= \int_l w$. 
Choosing $w$ appropriately in such an inequality potentially gives us information not only on the size of $|\widehat{g \,{\rm d}\sigma}|$ for $g \in L^2(\Sigma)$, but also on the locations where $|\widehat{g \,{\rm d}\sigma}|$ is large. In particular, it has been expected that this formulation of the conjecture should hold in the curved case in which $\Sigma$ is a strictly convex compact patch of a $C^2$ hypersurface with non-vanishing Gaussian curvature. Moreover, 
the full strength of boundedness of the $X$-ray transform of $w$ should not be needed: it should be sufficient to consider the maximal amount of mass that $w$ puts on lines which are normal to the hypersurface $\Sigma$ at some point. 

One of several ways to measure progress on the Mizohata--Takeuchi conjecture is to ask for ``local'' but uniform estimates of the form 
 \begin{equation}\label{eq:MT_hyp_local}
        \int_{B_R} |\widehat{g\,{\rm d}\sigma}|^2w \lesssim R^a \|Xw\|_\infty \int_\Sigma |g|^2 {\rm d} \sigma
    \end{equation}
(where $B_R$ is a ball of radius $R$) where one seeks to take the exponent $a\geq 0$ as small as possible. Such estimates for each $a > (n-1)/(n+1)$ were established in \cite{CIW24} for $\Sigma$ a strictly convex compact patch of a $C^2$ hypersurface with non-vanishing Gaussian curvature. On the other hand, an argument of Guth, \cite{Gu22} (which is discussed in detail in \cite{CIW24}) shows that, roughly speaking, if one is only permitted to use the ``standard tools'' of wave packet analysis and decoupling theory, then one cannot establish inequality \eqref{eq:MT_hyp_local} with {\em any} $a < (n-1)/(n+1)$. Note that as $n \to \infty$, the exponent $(n-1)/(n+1) \to 1$, and inequality \eqref{eq:MT_hyp_local} {is easily seen to hold with $a=1$ in all dimensions as a consequence of the Agmon--H\"ormander trace inequality for hypersurfaces
\[
  \int_{B_R} |\widehat{g\,{\rm d}\sigma}|^2 \lesssim R  \int_\Sigma |g|^2 {\rm d} \sigma
\]
(which is essentially an equivalence for $g$ constant on scale $R^{-1}$)  -- see \cite{CIW24} for details. Thus there is no ``asymptotic gain'' in the exponent $a$ over the Agmon--H\"ormander result in high dimensions.
Furthermore, a recent result of Cairo \cite{Ca25} shows that when $a = 0$, in the presence of curvature \eqref{eq:MT_hyp_local} fails, and in order for \eqref{eq:MT_hyp_local} to hold, one needs to insert a factor of at least $\log R$ on the right-hand side. Cairo's example does not preclude \eqref{eq:MT_hyp_local} holding for all $a >0$, and so the original Mizohata--Takeuchi conjecture needs to be modified accordingly to say that \eqref{eq:MT_hyp_local} holds for all $a >0$ with an implicit constant depending on $a$. Moreover, Mulherkar \cite{M25} has shown that a randomly generated weight will, with high probability, satisfy \eqref{eq:MT_hyp_local} for each $a >0$. The discussion of the present paper takes place in the context $a>0$. 

For other approaches to partial progress on the Mizohata--Takeuchi conjecture for hypersurfaces see \cite{CIW24} and especially the references therein.

It has long been recognised that there is nothing special about the role of hypersurfaces in the formulation of the Mizohata--Takeuchi conjecture. Indeed, let $\Sigma$ be a compact patch of a smooth $k$-dimensional surface in $\mathbb{R}^n$ with induced $k$-dimensional measure ${\rm d} \lambda_k$. The natural generalisation of the Mizohata--Takeuchi conjecture in this setting is
   \begin{equation*}
        \int_{\mathbb{R}^n} |\widehat{g \,{\rm d}\lambda_k}|^2w \lesssim \sup_H \left(\int_H w \, {\rm d} \lambda_{n-k} \right)\int_\Sigma |g|^2 \, {\rm d} \lambda_k
    \end{equation*}
for all $g\in L^2(\Sigma)$ and all non-negative weights $w$, where the $\sup$ is over all $(n-k)$-planes $H$ which are normal to the $k$-dimensional tangent plane at some point of $\Sigma$. An easy calculation using Plancherel's theorem verifies this inequality for all $1 \leq k \leq n-1$ in the completely flat case where $\Sigma$ is a (piece of) $k$-plane.

In this work we focus on the case $k=1$, so that we are dealing with curves, and, in analogy with the curved case treated in \cite{CIW24} for hypersurfaces, we shall consider {\em well-curved $C^{n+1}$ curves} $\Gamma:[0,1] \to \mathbb{R}^n$ for which we have $\nms{\Gamma}_{C^{n+1}} \lsm 1$ and $|\partial^{1}\Gamma(\xi) \wedge \cdots \wedge \partial^{n}\Gamma(\xi)| \gtrsim 1$ for all $\xi \in [0,1]$. Typical amongst this class of curves is the moment curve which is given by $\xi \mapsto  
(\xi,\xi^{2},\ldots,\xi^{n})$. Fixing notation, the generically-dimensioned $\Sigma$ is replaced by the curve $\Gamma$, and we suppress the superscript $1$ in $\lambda_1$ so that we are now considering 
the inequality 
\begin{equation*}
        \int_{\mathbb{R}^n} |\widehat{g \,{\rm d}\lambda}|^2w \lesssim \sup_H \left(\int_H w \, {\rm d}\lambda_{n-1}\right)\int_\Gamma |g|^2 \, {\rm d} \lambda
    \end{equation*}
for all $g\in L^2(\Gamma)$ and all non-negative weights $w$, where the $\sup$ is over all hyperplanes $H$ whose normals are parallel to the tangent at some point of $\Gamma$, and where $\int_H w$ is formed with respect to $(n-1)$-dimensional Lebesgue measure ${\rm d}\lambda_{n-1}$ on the hyperplane $H$. (We stress that we are not assuming that the hyperplane $H$ passes through the relevant point of tangency on $\Gamma$, just that it is parallel to the one passing through such a point of tangency.)

Since we shall not be able to prove such an inequality (even when $n=2$) we shall instead consider its local uniform variant  
\begin{equation*}
        \int_{B_R} |\widehat{g \,{\rm d}\lambda}|^2w \lesssim R^{a} \sup_H \left(\int_H w \, {\rm d}\lambda_{n-1} \right)\int_\Gamma |g|^2 \, {\rm d} \lambda.
    \end{equation*}
In the light of Cairo's work in \cite{Ca25}, one cannot expect this to hold for $a=0$.
The Agmon--H\"ormander trace inequality for curves 
\begin{equation}\label{eq:AH_curves_intro}
  \int_{B_R} |\widehat{g\,{\rm d}\lambda}|^2 \lesssim R^{n-1}  \int_\Gamma |g|^2\, {\rm d} \lambda
\end{equation}
(which is once again essentially an equivalence for $g$ constant on scale $R^{-1}$) shows that this local uniform inequality holds with $a=n-1$. Thus one seeks to take the exponent $0 < a< n-1$ as small as possible.  We will establish the following result in Section~\ref{MT_results}:

\begin{theorem}\label{thm:headline}
Fix $n \ge 2$. Let $\Gamma$ be a well-curved $C^{n+1}$ curve in $\mathbb{R}^n$. Let $r = \frac{n(n+1)}{n(n+1)-2}$. Then for each non-negative weight $w$ and each $\vep>0$, we have
\begin{equation*}
    \int_{B_R} |\widehat{g \,{\rm d}\lambda}|^2w \lesssim R^{a} \left(\sup_H \int_H w^r \, {\rm d}\lambda_{n-1} \right)^{1/r}\int_\Gamma |g|^2 \, {\rm d} \lambda
\end{equation*}
where $a = {\frac{n-3}{2} + \frac{2}{n} - \frac{2}{n^2(n+1)}+\vep}$, and where the $\sup$ is over all hyperplanes $H$ whose normals are parallel to the tangent at some point of $\Gamma$.
\end{theorem}

The case $n=2$ of this result was established and the corresponding power $a=1/3 + \vep$ was shown to be sharp in \cite{CIW24}. See also Theorem~\ref{thm:headline_companion} below.

\begin{rem}\label{rem:pla}
There is an equivalent version of Theorem~\ref{thm:headline}, 
in which the expression $\left(\sup_H \int_H w^r \, {\rm d}\lambda_{n-1} \right)^{1/r}$ on the right hand side is replaced by 
$$\left(\sup_S \int_S w(x)^r {\rm d}x\right)^{1/r},$$ 
where the $\sup$ is now over all $1$-neighbourhoods $S$ of hyperplanes which are normal to the tangent at some point of $\Gamma$, and where the integral is formed with respect to standard $n$-dimensional Lebesgue measure.

Indeed, for each fixed $S$ we have
$$ \int_Sw(x)^r {\rm d}x= \int_{H \in \mathcal{H}} \left(\int_H w(x)^r {\rm d}\lambda_{n-1}(x)\right) {\rm d} H \lesssim \sup_{H \in \mathcal{H}} \left(\int_H w(x)^r {\rm d}\lambda_{n-1}(x)\right) $$
where the family $\mathcal{H}$ parametrises the hyperplanes $H$ which are parallel to and contained in $S$, so that the version with thickened hyperplanes implies Theorem~\ref{thm:headline} as stated. 

Conversely, assume we know Theorem~\ref{thm:headline} as stated and we wish to establish the version with $1$-thickened hyperplanes. Since we are assuming that $\Gamma$ is compact, standard considerations related to the uncertainty principle show that we may replace $w$ on the left hand side by $\Phi \ast w$, where $\Phi$ is a non-negative $L^1$-normalised bump function which is bounded below on $B(0,1/2)$ (we refer to such weights as constant at unit scale and to such $\Phi$ as normalised bump functions).
Applying Theorem~\ref{thm:headline} as stated gives an upper bound in terms of 
$$\left(\sup_H \int_H (\Phi \ast w)^r \, {\rm d}\lambda_{n-1} \right)^{1/r}$$ 
on the right hand side. Since $r\geq 1$, we have 
$\Phi \ast w \leq \left( \Phi \ast w^r\right)^{1/r}$, so that 
$$\left(\int_H (\Phi \ast w)^r \, {\rm d}\lambda_{n-1} \right)^{1/r} \leq \left(\int_H \Phi \ast w^r \, {\rm d}\lambda_{n-1} \right)^{1/r} \lesssim \sup_S \left(\int_S w(x)^r {\rm d} x\right)^{1/r}$$
where $S$ is a $1$-neighbourhood of some hyperplane parallel to $H$. This shows that the version with $1$-thickened hyperplanes holds. 

For similar reasons, Theorem~\ref{thm:headline} is equivalent to the version of itself (and to the version with $\sup_S \left(\int_S w(x)^r {\rm d}x\right)^{1/r} $ on the right hand side) in which $w$ is assumed to be constant at unit scale. 
For weights which are constant at unit scale we also have, by the embedding of $\ell^1$ into $\ell^r$, that
$$\left(\int_S w^r\right)^{1/r} \lesssim  \int_S w \;\; \mbox{ and } \,\, \left(\int_H w^r\right)^{1/r} \lesssim  \int_H w,$$
so that there is a weaker version of any of these equivalent versions of Theorem~\ref{thm:headline} in which the $L^r$-norm of $w$ on the right hand side is replaced by its $L^1$-norm. 

Remarks similar to these apply at many points in what follows.
\end{rem}

Thus, as a corollary to Theorem~\ref{thm:headline}, we have:
\begin{cor}\label{cor:thm_headline}
Fix $n \ge 2$. Let $\Gamma$ be a well-curved $C^{n+1}$ curve in $\mathbb{R}^n$. Then for each non-negative weight $w$ and each $\vep>0$, we have
\begin{equation*}
        \int_{B_R} |\widehat{g \,{\rm d}\lambda}|^2w \lesssim R^{a} \sup_S w(S)\int_\Gamma |g|^2 \, {\rm d} \lambda
    \end{equation*}
where $a = {\frac{n-3}{2} + \frac{2}{n} - \frac{2}{n^2(n+1)}+\vep}$, and where the $\sup$ is over all $1$-neighbourhoods $S$ of hyperplanes whose normals are parallel to the tangent at some point of $\Gamma$.
\end{cor}

\begin{rem}
Throughout this work, if $v$ is a non-negative locally integrable function, we shall use the notations $v(E)$ and $\int_E v$ interchangeably.
\end{rem}

\begin{rem} Notice that when $n=2$, we have $a = 1/3 + \vep$, which is in agreement with the result from \cite{CIW24}. However, as $n \to \infty$, $\frac{n-3}{2} + \frac{2}{n} - \frac{2}{n^2(n+1)} =\frac{n-3}{2} + o(1)$, demonstrating a certain asymptotic gain over the Agmon--H\"ormander power $(n-1)$. We shall further address the value of the exponent $a$ in Theorem~\ref{thm:headline} (with $r$ as prescribed) and in Corollary~\ref{cor:thm_headline} in Sections~\ref{MT_results} and \ref{sec:axdec} below. See Remark~\ref{rem:diff_various_sharpness} after Theorem~\ref{eq:primordial_prior_X} and Theorem~\ref{thm:axiomatic_sharpness} (iii).
\end{rem}

\begin{rem}
For some partial results for the Mizohata--Takeuchi conjecture for curves in the spirit of the tomographic and phase space approaches of \cite{BN21}, \cite{BNS22} and \cite{BGNO24}, see the PhD thesis \cite{Fer24}. The latter reference also explores limitations of the Mizohata--Takeuchi conjecture outwith the setting of {\em smooth} submanifolds of $\mathbb{R}^n$.
\end{rem}

There is a companion result to Theorem~\ref{thm:headline}, in which we place line or tube conditions instead of hyperplane conditions on the weight. We state the tube version, but as in Remark~\ref{rem:pla} above, there is an equivalent version for lines.
\begin{theorem}\label{thm:headline_companion}
Let $\Gamma$ be a well-curved $C^{n+1}$ curve in $\mathbb{R}^n$. Let $r= \frac{n(n+1)}{n(n+1)-2}$. Then for each $\vep >0$, we have
\[
\int_{B_R} |\widehat{g \,{\rm d}\lambda}|^2  w(x) {\rm d} x \lesssim R^{(n-2)+\frac{2}{n(n+1)}+\vep} \Big(\sup_P w^r(P)\Big)^{1/r} \int_{\Gamma} |g|^2 {\rm d} \lambda.
\]
where the $\sup$ is taken over the family of $1$-tubes $P$ all $1$-tubes which are parallel to the final Serret--Frenet direction at some point of $\Gamma$. Moreover, this inequality is sharp in the sense that with $r$ as prescribed, the power ${(n-2)+\frac{2}{n(n+1)}}$ on $R$ on the right hand side cannot be replaced by any strictly smaller number.  
\end{theorem}
(Recall that for a well-curved curve $\Gamma$, there is at each point of $\Gamma$ a Serret--Frenet frame whose directions are the tangent, normal, binormal etc. directions of $\Gamma$ at each point. See the introduction to Section~\ref{sec:refineddecoupling_curves} for further discussion.)

We also prove this result in Section~\ref{MT_results}. As above, we have: 
\begin{cor}\label{cor:headline_companion}
Let $\Gamma$ be a well-curved $C^{n+1}$ curve in $\mathbb{R}^n$. Then for each $\vep >0$, we have
\[
\int_{B_R} |\widehat{g \,{\rm d}\lambda}|^2  w(x) {\rm d} x \lesssim R^{(n-2)+\frac{2}{n(n+1)}+\vep} \sup_P w(P) \int_{\Gamma} |g|^2 {\rm d} \lambda.
\]
where the $\sup$ is taken over the family of $1$-tubes $P$ which are parallel to the final Serret--Frenet direction at some point of $\Gamma$. 
\end{cor}

\begin{rem}\label{rem:confusion}
To help place Theorem~\ref{thm:headline_companion} in context, and indeed to establish its sharpness assertion, suppose that for some $a> 0$ and $r\geq 1$ we have the inequality 
\begin{align*}
\int_{B_R}|\wh{g\,{\rm d}\ld}|^{2}w\lsm R^a\,\left(\sup_{P}w^r(P)\right)^{1/r}\int_{\Gamma} |g|^{2}\,{\rm d}\ld,
\end{align*}
where $P$ ranges over all $1$-tubes $P$. Note that this latter inequality does hold for $a=n-1$ and $1 \leq r \leq \infty$ by the Agmon--H\"ormander trace inequality. We claim that necessarily $a \geq  (n-1) - 1/r = (n-2) + 1/r'$.

This follows from unpacking the argument for the Agmon--H\"ormander inequality (and simultaneously its essential sharpness for inputs constant on scale $R^{-1}$) in the case of curves. Indeed, taking
$w = 1$, then $\sup_{P} w(P) = R$, and one then has
\begin{align}\label{mttube}
\int_{B_R}|\wh{g\,{\rm d}\ld}|^{2} \lsm R^{a+\frac{1}{r}}\int_{\Gamma} |g|^{2}\, {\rm d}\ld.
\end{align}
Take $g = \sum_{v}a_{v}1_{S_v}$ where $\{S_v\}$
is a disjoint collection of $R^{-1}$-arcs on $\Gamma$. Then
$\nms{g}_{L^{2}}^{2} = \sum_{v}|a_{v}|^{2}R^{-1}$, and the right hand side of \eqref{mttube} is $\sum_{v}|a_{v}|^{2}R^{a - \frac{1}{r'}}$.
Let $\Phi_{1/R}$ be an inverse Fourier transform of a smoothed version of $1_{B_R}$ (and
so is essentially $R^{n}1_{B(0,1/R)}$). Then
$g\,{\rm d}\ld \ast \Phi_{1/R}\sim \sum_{v} a_{v}1_{\wt{S}_v}R^{n-1}$
where $\wt{S}_{v}$ is an $R^{-1}$-neighbourhood of $S_v$ (essentially an $R^{-1}$-ball).
Therefore 
$$\int_{B_R}|\wh{g\,{\rm d}\ld}|^{2}  \sim \nms{g\,{\rm d}\ld \ast \Phi_{1/R}}_{2}^{2}\sim \sum_{v}|a_v|^{2}R^{2(n-1)}R^{-n} = \sum_{v}|a_v|^{2}R^{n-2},$$
and hence in order for \eqref{mttube} to hold we need $a \geq n-2 + \frac{1}{r'}$. This demonstrates the sharpness assertion of Theorem~\ref{thm:headline_companion}. The power $(n-2) + \frac{2}{n(n+1)}$ on $R$ in Theorem~\ref{thm:headline_companion} therefore again represents a significant asymptotic gain over the straightforward Agmon--H\"ormander power $(n-1)$. 

On the other hand, we do not expect the power 
$(n-2) + \frac{2}{n(n+1)}$ in the weaker inequality of Corollary~\ref{cor:headline_companion} to be sharp. Indeed, when $n=2$, the Mizohata--Takeuchi conjecture indicates that the power should be arbitrarily small  rather than $1/3$. Nevertheless, it is sharp ``modulo axiomatic decoupling'' when $n=2$, as established in \cite{Gu22}, see also \cite{CIW24}; for further discussion on this latter point see Section~\ref{sec:axdec}, especially Remark~\ref{rem:axiomatiacally_sharp}.

\end{rem}

In common with the approach in \cite{CIW24}, Theorems~\ref{thm:headline} and \ref{thm:headline_companion} are established using techniques from refined decoupling theory. In the setting of hypersurfaces, the natural input is refined decoupling for hypersurfaces. In our setting we shall require refined decoupling for curves, and we turn to this in the next section. However, the analogue of the naive wave packet decomposition for a function defined on a hypersurface which was employed in \cite{CIW24} is not directly available to us in the setting of curves in $\mathbb{R}^n$ when $n \geq 3$, and so we shall need to apply decoupling in a different way from the one set out in \cite{CIW24}\footnote{Somewhat anticipating the details, the issue is that when we do the wave packet decomposition in higher dimensions for the function $1_{B_R} Eg$ (where $g$ is defined on the curve), say $1_{B_R} Eg = \sum_T f_T$, the constituent pieces $f_T$ are frequency localised to boxes of size $R^{-1/n} \times R^{-2/n} \times \dots \times R^{-1}$, and these boxes have substantial mass outwith the $R^{-1}$ neighbourhood of the curve $\Gamma$. Hence it will not in general be possible to write these wave packets $f_T$ as $1_{B_R} Eg_T$ for any function $g_T$ (for then the Fourier support of $1_{B_R} Eg_T$ would lie only in a $R^{-1}$ neighbourhood of the curve $\Gamma$). We only see this issue in dimensions $n \geq 3$. \label{footnote1}}. Moreover the scope of the decoupling theory we shall use is significantly more powerful and general than is needed to obtain the results we have described. (In other words, when $n \geq 3$, our results use rather special cases of the general decoupling results.) It turns out that the Agmon--H\"ormander inequality for curves will play a role in making the connection between the form of decoupling we use and our desired estimates for the Fourier extension operator. 

In Section~\ref{sec:refineddecoupling_curves} we shall discuss the background to decoupling for curves and the refined decoupling estimates we shall need. (Some of the detailed arguments are postponed to Section~\ref{sec:appendix}.) We explore the direct consequences of these results in Section~\ref{sec:rdc_weighted_applns}, and we will then apply them to obtain Theorems~\ref{thm:headline} and \ref{thm:headline_companion} in Section~\ref{MT_results}. In Section~\ref{sec:axdec} we shall consider the situation in the more general setting of axiomatic decoupling, and in Section~\ref{sec:appendix} we shall formulate and gather together the details of the results we need from decoupling theory which we postponed earlier.

 \noindent{\textbf{Acknowledgments.}} 
We are happy to acknowledge the profound influence of the perspective that Larry Guth has brought to bear on decoupling in this work. We would also like to thank Marina Iliopoulou, Dominique Maldague and Hong Wang for sharing their insights during the preparation of the manuscript. The referees' careful reading of the paper and their very helpful comments were appreciated. This project started when the first two authors were visiting Yung during the Special Year on Harmonic Analysis at the Australian National University. Li is partially supported by DMS-2409803 and DMS-2453448. Yung is partially supported by a Future Fellowship FT200100399 from the Australian Research Council, and an AustMS WIMSIC Anne Penfold Street Award during his visit to Bonn where part of this work was carried out.

\section{Refined decoupling for curves and weighted formulations}\label{sec:refineddecoupling_curves}

Let $\Gamma \colon [0,1] \to \R^n$ be a well-curved $C^{n+1}$ curve in $\R^n$, so that $\nms{\Gamma}_{C^{n+1}[0,1]} \lsm 1$ and $|\partial^{1}\Gamma(\xi) \wedge \cdots \wedge \partial^{n}\Gamma(\xi)| \gtrsim 1$ for all $\xi \in [0,1]$.  Associated to each point of the curve is a Serret--Frenet orthonormal frame, whose directions are given by the tangential, normal, binormal etc. directions, up to the ``final'' Serret--Frenet direction.
Let ${\rm d} \lambda$ be the arclength measure for $\Gamma$, or any other measure comparable to it. Given $R \in 2^{n\mathbb{N}}$, we pick a maximal collection of points $\{\xi_i\} \subset [0,1]$ that are $R^{-1/n}$ separated, and let $\Theta_{\Gamma}(R^{-1/n})$ be the collection of anisotropic curvature boxes of the form $\theta(\xi_i):=\Gamma(\xi_i) + \sum_{j=1}^n R^{-j/n} [-1,1] \Gamma^{(j)}(\xi_i)/j! $. 
These boxes have size $\sim R^{-1/n}$ in the tangential direction, $\sim R^{-2/n}$ in the normal direction, $\sim R^{-3/n}$ in the binormal direction, etc., and size $ \sim R^{-1}$ in the final direction given by the Serret--Frenet frame for $\Gamma$.
The union of all boxes in $\Theta_{\Gamma}(R^{-1/n})$ forms a ``curvature sleeve'' of $\Gamma$. 
The curvature assumed means  that $\Gamma$ fits neatly in $\Theta_{\Gamma}(R^{-1/n})$.
It is crucial to note that, when $n \geq 3$, this curvature sleeve is significantly larger than the isotropic $R^{-1}$-neighbourhood of the curve. It is this difference which accounts for the divergence referred to above in the applications of decoupling between the cases $n=2$ (or more generally the hypersurface case) and the cases $n \geq 3$.  Simple orthogonality considerations give
\[
\|\sum_\theta f_\theta\|_{L^2(\R^n)} \sim \Big(\sum_\theta \|f_\theta\|_{L^2(\R^n)}^2\Big)^{1/2}.
\]
The Bourgain--Demeter--Guth decoupling inequality for curves, \cite{BDG16}, is a far-reaching generalisation of this, and says that if $\{f_{\theta}\}$ is a family of Schwartz functions such that $\widehat{f_{\theta}}$ is supported in $\theta$ for every $\theta \in \Theta_{\Gamma}(R^{-1/n})$, then for $2 \leq p \leq n(n+1)$,
\[
\Big\|\sum_{\theta} f_{\theta} \Big\|_{L^p(\R^n)} \lesssim_{\epsilon} R^{\epsilon} \Big( \sum_{\theta} \|f_{\theta}\|_{L^p(\R^n)}^2 \Big)^{1/2}.
\]

\subsection{Refined decoupling for curves}%

		In what follows we need a refined decoupling inequality. Such inequalities are usually formulated in terms of a wave packet decomposition, which we carry out as follows.
		
		First, every box $\theta = \theta(\xi_i)\in \Theta_{\Gamma}(R^{-1/n})$  arises as the affine image of $[-1,1]^n$, say $\theta = A_{\theta} [-1,1]^n$ with $A_{\theta}(\eta) = \Gamma(\xi_i) + \sum_{j=1}^n R^{-j/n} \Gamma^{(j)}(\xi_i) \eta_j / j!$, and we shall write $T_{\theta}$ for the transpose of the linear part of $A_{\theta}$. Write $\theta/4$ for $A_{\theta}[-1/4,1/4]^n$, and from now on let $f_{\theta}$ be a Schwartz function with Fourier support inside $\theta/4$. Then $\widehat{f_{\theta}} \circ A_{\theta}$ is supported in $[-1/2,1/2]^n$, and we may expand it in Fourier series. Now we fix, once and for all, a Schwartz function $\Phi$ on $\R^n$ so that $\widehat{\Phi}$ is supported on $[-1/2,1/2]^n$ and is identically 1 on $[-1/4,1/4]^n$. Then 
		\[
		\widehat{f_{\theta}} \circ A_{\theta}(\xi) = \sum_{m \in \mathbb{Z}^n} a_{m,\theta} e^{-2\pi i m \cdot \xi} \widehat{\Phi}(\xi) \quad \text{for all $\xi \in \R^n$}
		\]
		where $a_{m,\theta}$ are the Fourier coefficients of $\widehat{f_{\theta}} \circ A_{\theta}$.
		Undoing the Fourier transforms, we have
		\begin{equation} \label{eq:wavepacket_decomp_ftheta}
			f_{\theta} = \sum_{T \in \mathbb{T}_{\theta}} f_T
			\; \; \mbox{ and } \; \; \|f_{\theta}\|_2^2 = \sum_{T \in \mathbb{T}_{\theta}} \|f_T\|_2^2
		\end{equation}
		for every $\theta$. Here  $\T_{\theta}$ is the collection of tiles
		\begin{equation} \label{eq:T_theta_def}
		\T_{\theta} := \Big\{T_{\theta}^{-1} \Big( m + [-1/2,1/2]^n \Big) \colon m \in \mathbb{Z}^n \Big\}
		\end{equation}
		which tiles $\R^n$, and 
		\begin{equation} \label{eq:fTdef0}
		f_T(x) := |\det T_{\theta}| e^{2\pi i \Gamma(\theta) \cdot x} a_{m,\theta} \Phi(T_{\theta} x - m)
		\end{equation}
		if $T = T_{\theta}^{-1} \Big(m + [-1/2,1/2]^n \Big)$ for some $m \in \mathbb{Z}^n$, and $\Gamma(\theta) := A_{\theta}(0)$.  We also define 
		\[
		\T = \bigsqcup_{\theta \in \Theta_{\Gamma}(R^{-1/n})} \T_{\theta}
		\]
		as the collection of all admissible wave packet tiles.
		
		For each box $\theta$, the tiles in $\T_\theta$ are aligned with the Serret--Frenet frame for points in $\Gamma \cap \theta$ insofar as they are parallelepipeds whose faces have normal directions 
		given by the Serret--Frenet frame for such points, and with sidelengths approximately $R^{j/n}$ in the $j$'th Serret--Frenet direction for $1 \leq j \leq n$. It is this observation which ultimately gives rise to how our main results are formulated.
		
		Our definition guarantees that $\widehat{f_T}$ is supported in $\theta$ for every $T \in \T_{\theta}$. Unfortunately this means $f_T$ is not compactly supported in space, but $f_T$ is decaying rapidly outside $R^{\epsilon}T$ for any fixed $\epsilon$, and this explains the role of such dilates of $T$ in the following formulation of the refined decoupling theorem.

\begin{theorem} \label{thm:refined_dec_wp_prior}
Let $\Gamma$ be a well-curved $C^{n+1}$ curve in $\R^n$ and $2 \leq p \leq n(n+1)$. Let $R \in 2^{n\mathbb{N}}$ and for each $\theta \in \Theta_{\Gamma}(R^{-1/n})$, let $f_{\theta}$ be a Schwartz function with Fourier support in $\theta / 4$, with wave packet decomposition \eqref{eq:wavepacket_decomp_ftheta}. Let $B$ be a ball of radius $R$, and $\T(B)$ be a collection of $T \in \T$ for which $T \cap B \ne \emptyset$. Let $\epsilon > 0$, $M \geq 1$ and let $Y_M$ be a union of disjoint cubes $Q$ of side length $R^{1/n}$ in $\R^n$, such that every $Q \subset Y_M$ is contained in $R^{\epsilon}T$ for at most $M$ tiles $T \in \T(B)$. Then
\begin{equation*}%
\Big\| \sum_{T\in \T(B)} f_T \Big\|_{L^p(Y_M)} \lesssim_{\epsilon} R^{\epsilon} M^{\frac{1}{2}-\frac{1}{p}} \Big( \sum_{T \in \T(B)} \|f_T\|_{L^p(\R^n)}^p \Big)^{1/p}.
\end{equation*}

\end{theorem}

The case $n=2$ of Theorem~\ref{thm:refined_dec_wp_prior} is a result of \cite{GIOW20}, while the case $n=3$ appeared as Theorem 7.5 in \cite{DGW20}. The argument for the general case follows the same broad lines as in these references, but a fully rigorous treatment is rather lengthy. In order not to interrupt the flow of our discussion, we postpone this to Section~\ref{sec:appendix_subsec1}, see Theorem~\ref{thm:refined_dec_wp}. %

\subsection{A weighted formulation of refined decoupling for curves}%
One of the insights of \cite{CIW24} was that refined decoupling estimates are equivalent to certain $L^2$-weighted decoupling estimates for Schwartz functions which are Fourier supported in the curvature sleeve of a suitable submanifold of $\R^n$. In \cite{CIW24} this equivalence was established for strictly convex hypersurfaces, and here we shall need the analogous developments for well-curved curves. We next formulate our weighted refined decoupling estimates in this context. As in Theorem~\ref{thm:refined_dec_wp_prior}, for $\theta \in \Theta_{\Gamma}(R^{-1/n})$ we consider Schwartz functions $f_{\theta}$ with Fourier support in $\theta/4$, and let $f := \sum_{\theta} f_{\theta}$. Then defining $f_T$ as in \eqref{eq:wavepacket_decomp_ftheta}, 
\begin{equation} \label{eq:wp_decomp}
    f  = \sum_{T \in \T} f_T
\end{equation} 
is the wave packet decomposition of $f$.

\begin{theorem} \label{main_corrected_reprise}
Let $n \geq 2$, and fix a parameter $K\gg1$.  For any $\epsilon > 0$, there exists a constant $C_{\epsilon,K}$ depending only on $\epsilon, K$ and $n$, such that the following is true: If $p = n(n+1)$ and $\frac{1}{r}=1-\frac{2}{p}$, $f=\sum_{T \in \T} f_T$ is a Schwartz function which is Fourier supported in the curvature sleeve $\Theta_{\Gamma}(R^{-1/n})$ of a well-curved $C^{n+1}$ curve $\Gamma$ as in \eqref{eq:wp_decomp}, and $w$ is any non-negative weight, then
\begin{equation}\label{eq:weightedwp1_reprise}
\begin{split}\int_{\R^n} |f(x)|^2 & w(x) {\rm d} x \leq C_{\epsilon,K} R^{\epsilon} \Big( \sum_{T \in \T} \|f_T\|_{L^2(\R^n)}^2  \frac{w^r(2 R^{\epsilon} T)}{|T|} \Big)^{1/r} \|f\|_{L^2(\R^n)}^{4/p} \\
+ & C_{\epsilon,K} R^{-K}\sup_{m\geq 1} 2^{-mK} \Big(\sum_{T \in \T} \|f_T\|_{L^2(\R^n)}^2 \frac{w^r(2^m T)}{|2^mT|} \Big)^{1/r} \|f\|_{L^2(\R^n)}^{4/p}.
\end{split}
\end{equation}
\end{theorem}

Choosing $K$ sufficiently large, we see that in the second term we get exponential decay of all orders in both the $2^m$ and the $R$ terms.

Since the connection between refined decoupling estimates and $L^2$-weighted estimates is still relatively novel, we do wish to give an informal indication of how we might obtain Theorem~\ref{main_corrected_reprise} as a consequence of Theorem~\ref{thm:refined_dec_wp_prior}. We give only a sketch of the argument. For full details of the proof of Theorem~\ref{main_corrected_reprise} see Section~\ref{sec:proof_2.2}, and for the {\em equivalence} of Theorem~\ref{main_corrected_reprise} with Theorem~\ref{thm:refined_dec_wp_prior}, see Section~\ref{sec:duality} below. Note that 
both the direct and converse arguments in the equivalence follow \cite[Section~4]{CIW24} rather closely.

\begin{proof} [Sketch of proof that Theorem~\ref{thm:refined_dec_wp_prior} implies Theorem~\ref{main_corrected_reprise}]
We shall pretend that each $f_T$ is actually supported in $T$. This will allow us to derive a version of \eqref{eq:weightedwp1_reprise} without error terms, and where we can replace $w^r(2 R^{\epsilon}T)$ on the right-hand side by $w^r(2 T)$. We give a heuristic argument for a local version of \eqref{eq:weightedwp1_reprise}, namely
\[
\int_{B} |f(x)|^2 w(x) {\rm d}x \lesssim_{\epsilon} R^{\epsilon} \Big( \sum_{T \in \T(B)} \|f_T\|_{L^2(\R^n)}^2 \frac{w^r(2 T)}{|T|} \Big)^{1/r} \Big( \sum_{T \in \T(B)} \|f_T\|_{L^2(\R^n)}^2 \Big)^{2/p}
\]
for any ball $B$ of radius $R$. Here $\T(B)$ consists of those $T \in \T$ which intersect $B$. The global version of the inequality can then be established by summing over $B$ and using H\"{o}lder's inequality on the right-hand side. 

To establish the local inequality, we first note that 
\[
\int_{B} |f(x)|^2 w(x) {\rm d}x = \int_B \Big| \sum_{T \in \T(B)} f_T(x)\Big|^2 w(x) {\rm d}x
\]
by our (white lie) hypothesis on the support of $f_T$. Then we dyadically pigeonhole, at the cost of a factor of $\log R$, and obtain a collection $\T' \subset \T(B)$ such that all $\|f_T\|_{L^2(\R^n)}$ are comparable as $T$ varies over $\T'$, and such that up to an acceptable error,
\[
\int_B \Big| \sum_{T \in \T(B)} f_T(x)\Big|^2 w(x) {\rm d}x \lesssim (\log R)^2 \int_B \Big| \sum_{T \in \T'} f_T(x)\Big|^2 w(x) {\rm d}x.
\]
Further dyadic pigeonholing, also with a logarithmic cost, allows us to choose a popular value of $k$ so that
\[
\int_B \Big| \sum_{T \in \T'} f_T(x)\Big|^2 w(x) {\rm d}x \lesssim (\log R) \int_{Y_k} \Big| \sum_{T \in \T'} f_T(x)\Big|^2 w(x) {\rm d} x
\]
where $Y_k$ is the union of those $R^{1/n}$-cubes in $B$ that are incident to $\sim k$ planks from $\T'$. Then by H\"{o}lder's inequality,
\[
\int_{Y_k}  \Big| \sum_{T \in \T'} f_T(x) \Big|^2 w(x) {\rm d} x \leq w^r(Y_k)^{1/r} \Big( \int_{Y_k} \Big| \sum_{T \in \T'} f_T(x) \Big|^p {\rm d} x \Big)^{2/p}.
\]
We would now apply Theorem~\ref{thm:refined_dec_wp_prior}, %
without worrying too much about the distinction between incidence to a plank $T$ from $\T'$ and its $R^{\epsilon}$ thickened version $R^{\epsilon}T$,
and obtain
\[
w^r(Y_k)^{1/r} \Big( \int_{Y_k} \Big| \sum_{T \in \T'} f_T(x) \Big|^p {\rm d} x \Big)^{2/p} \lesssim_{\epsilon} R^\epsilon \, (k w^r(Y_k))^{1/r} \Big( \sum_{T \in \T'} \|f_T\|_{L^p(\R^n)}^p \Big)^{2/p}.
\]
Local constancy (or the uncertainty principle) gives
\[
\|f_T\|_{L^p(\R^n)}^2 \sim |T|^{-1/r} \|f_T\|_{L^2(\R^n)}^2.
\]
Hence, since all the $\|f_T\|_{L^2(\R^n)}$ are comparable, (recalling also that $|T| \sim R^{(n+1)/2}$ is independent of $T$),
\[
\begin{split}
(k w^r(Y_k))^{1/r} \Big( \sum_{T \in \T'} \|f_T\|_{L^p(\R^n)}^p \Big)^{2/p} 
& \sim \Big(\frac{k w^r(Y_k)}{|T|} \Big)^{1/r} (\#\T')^{2/p} \|f_T\|_{L^2(\R^n)}^2  \\
&\lesssim \Big(\|f_T\|_{L^2(\R^n)}^2 \frac{k w^r(Y_k)}{|T|} \Big)^{1/r} \Big( \sum_{T \in \T'} \|f_T\|_{L^2(\R^n)}^2 \Big)^{2/p},
\end{split}
\]
where we have used that $ 2 = 2/r + 4/p$ and $\#\T' \, \|f_T\|_{L^2(\R^n)}^2 \sim \sum_{T \in \T'} \|f_T\|_{L^2(\R^n)}^2$. 
Now we have the simple counting argument 
\begin{align*}
k \, w^{r}(Y_k) &\sim \sum_{Q \text{ cubes } \subset Y_k :\, \ell(Q)= R^{1/n}} k \, w^{r}(Q)\\
&\sim \sum_{Q  \subset Y_k}\sum_{T \in\T' \, :\, T \cap Q\neq \emptyset}w^{r}(Q)\\
&\leq \sum_{T \in\T'}\sum_{Q \cap T \neq \emptyset}w^{r}(Q) \lesssim \sum_{T\in\T'}w^{r}(2 T),
\end{align*}
and gathering everything together gives that we have
\[
\int_{B} |f(x)|^2 w(x) {\rm d} x \lesssim_{\epsilon} R^{2\epsilon} \, \Big( \sum_{T \in \T'} \|f_T\|_{L^2(\R^n)}^2 \frac{w^r(2 T)}{|T|} \Big)^{1/r} \Big( \sum_{T \in \T'} \|f_T\|_{L^2(\R^n)}^2 \Big)^{2/p}.
\]
The argument is now complete since $\T'$ is a subcollection of $\T(B)$.

We caution the reader that this argument is not rigorous on at least two counts: firstly, it assumed that $f_T$ is compactly supported in $T$, and secondly the application of Theorem~\ref{thm:refined_dec_wp_prior} is not justified as 
cubes in $Y_k$ intersect $\sim k$ planks $T$ from $\T'$, but may intersect many more $R^{\epsilon} T$ with $T \in \T'$.
Dealing with these issues 
requires consideration of the contribution from $w$ on larger tubes than the wave packet tubes $T$ occurring here, and this is what Theorem~\ref{thm:weightedwp} and the subsequent full proof of Theorem~\ref{main_corrected_reprise} in Section~\ref{sec:proof_2.2} address.
\end{proof}

\begin{rem}
As in \cite[Remark~4.3]{CIW24} we may apply Theorem~\ref{main_corrected_reprise} with $f$ as there and the weight $w$ taken to be $w(x) = |f(x)|^a$ for appropriate powers $a$ %
to obtain, under the same hypotheses,
\begin{equation*}
\int_{\R^n} |f(x)|^{2+a}{\rm d} x \leq C_{\epsilon,K} R^{\epsilon} \Big( \sum_{T \in \T} \|f_T\|_{L^2(\R^n)}^2 \frac{1}{|T|}\int_{T}|f|^{ra} \Big)^{1/r} \|f\|_{L^2(\R^n)}^{4/p}%
\end{equation*}
up to terms which decay rapidly in $R$ and $2^K$ which, together with $\epsilon > 0$,  we shall ignore. In particular, when $a= n(n+1)-2$, $p=n(n+1)$, $1/r + 2/p =1$, this formulation becomes 
\[
\int_{\R^n} |f|^{p} \lesssim  \Big( \sum_{T \in \T} \|f_T\|_{L^2(\R^n)}^2 \frac{1}{|T|}\int_{T}|f|^{p} \Big)^{1- \frac{2}{p}} \|f\|_{L^2(\R^n)}^{\frac{4}{p}}.%
\]
Now 
\begin{equation}\label{eq:CW}
\begin{split}
& \sum_{T \in \T} \|f_T\|_{L^2(\R^n)}^2 \frac{1}{|T|}\int_{T}|f|^{p}
= \sum_\theta \sum_{T \in \T_\theta} \|f_T\|_{L^2(\R^n)}^2 \frac{1}{|T|}\int_{T}|f|^{p}
\\
\leq & \sum_\theta  \sup_{T \in \T_\theta}\|f_T\|_{L^\infty(\R^n)}^2\sum_{T \in \T_\theta} \int_{T}|f|^{p}
\sim \sum_\theta  \|f_\theta\|_{L^\infty(\R^n)}^2 \int_{\mathbb{R}^n}|f|^{p}
\end{split}
\end{equation}
since by local constancy (or the uncertainty principle)
$\|f_T\|_{L^\infty(\R^n)}^2 \sim |T|^{-1} \|f_T\|_{L^2(\R^n)}^2$.
Therefore
\[
\int_{\R^n} |f|^{p} \lessapprox \Big( 
\sum_\theta  \|f_\theta\|_{L^\infty(\R^n)}^2 \int_{\mathbb{R}^n}|f|^{p}
\Big)^{1- \frac{2}{p}} \|f\|_{L^2(\R^n)}^{\frac{4}{p}},%
\]
and hence
\begin{equation}\label{eq:CIW}
\int_{\R^n} |f|^{p} \lessapprox \Big( 
\sum_\theta  \|f_\theta\|_{L^\infty(\R^n)}^2 
\Big)^{\frac{1}{2}(p-2)} \|f\|_{L^2(\R^n)}^2.
\end{equation}
As in \cite[Remark~4.3]{CIW24}, this inequality makes contact with the formulation of \cite[Theorem~1.2]{GMW20}, which says that in the case of the parabola $\Gamma_0$ in the plane, and $p=6$, we have
\[
 \|f\|_{L^6(\R^2)}^6 \lesssim  (\log R)^{C} \Big( \sum_{\theta \in \Theta_{\Gamma_0}(R^{-1/2})} \|f_{\theta}\|_{L^\infty(\mathbb{R}^2)}^2 \Big)^{2} \|f\|_{L^2(\R^2)}^{2}
 \]
for an explicit value of the constant $C$.

However, the main point of \cite{GMW20} is that, at least in the case $n=2$ and the specific case of the parabola, the $\epsilon$-losses implicit in these inequalities %
may be replaced by logarithmic ones.

A small tweak in the above argument 
gives a slightly better bound for the $L^{n(n+1)}$ norm of $f$ in terms of its $L^2$ norm and the $L^{\infty}$ norm of its square function $(\sum_\theta |f_\theta|^2)^{1/2}$ which is sharp up to $R^{\epsilon}$ losses. Indeed, with $f$ as in Theorem~\ref{main_corrected_reprise}, we have\footnote{As a substitute for the false inequality $\displaystyle{\|f\|_{L^p(\R^n)} \lesssim_{\epsilon} R^{\epsilon} \Big\| \Big( \sum_{\theta \in \Theta_{\Gamma}(R^{-1/n})} |f_{\theta}|^2 \Big)^{1/2} \Big\|_{L^{p}(\R^n)}}$.}
\begin{equation}\label{eq:HW}
\|f\|_{L^p(\R^n)}^p \lesssim_{\epsilon} R^{\epsilon} \Big\| \Big( \sum_{\theta \in \Theta_{\Gamma}(R^{-1/n})} |f_{\theta}|^2 \Big)^{1/2} \Big\|_{L^{\infty}(\R^n)}^{p-2} \|f\|_{L^2(\R^n)}^2.
\end{equation}
This inequality was known to the authors of \cite{GMW20} and \cite{FGM23}, and was kindly communicated to us by Hong Wang \cite{Wa24}, but as far as we know has not appeared explicitly previously.

To see this, for convenience we shall work locally on a ball $B_R$ of radius $R$ in $\R^n$, and once again pretend that all wave packets $f_T$ are actually supported on $T$. This leaves only $R^{O(1)}$ wave packets to consider on $B_R$, and what we need to prove is that 
\begin{equation} \label{eq:GMW}
\int_{B_R} |f|^p \lesssim_{\epsilon} R^{\epsilon} \Big\| \Big( \sum_{T \subset B_R} |f_T|^2 \Big)^{1/2} \Big\|_{L^{\infty}}^{p-2} \int_{B_R} |f|^2.
\end{equation}
Since the contribution to $\int_{B_R}|\sum_T f_T|^p$ coming from those $T$ which satisfy   
$\|f_{T}\|_{L^{\infty}} \ll 
R^{-O(1)} \sup_{T \subset B_R} \|f_T\|_{L^{\infty}}$ is controlled by 
\[\Big\| \Big( \sum_{T \subset B_R} |f_T|^2 \Big)^{1/2} \Big\|_{L^{\infty}}^{p-2} \int_{B_R} |f|^2,
\]
leaving only $\log R$  values of $\|f_T\|_{L^{\infty}}$ to consider,  we may pigeonhole a subcollection $\T'$  such that $\|f_{T}\|_{L^{\infty}}$ are comparable -- say $\|f_{T}\|_{L^{\infty}} \sim \lambda$ -- as $T$ varies over $\T'$. 

Let $\tilde{f} := \sum_{T \in \T'} f_T$ be the pruned version of $f$. We apply Theorem~\ref{main_corrected_reprise} as before, with $\tilde{f}$ replacing ${f}$, and with weight $|\tilde{f}|^{p-2} 1_{B_R}$. We claim that
\[
\sum_{T \in \T'} \|f_T\|_{L^2(\R^n)}^2 \frac{1}{|T|}\int_{T}|\tilde{f}|^{p}
\lesssim \Big\|\Big(\sum_{T\in \T'} |f_T|^2\Big)^{1/2}\Big\|_{L^\infty(B_R)}^2 \int_{B_R}|\tilde{f}|^{p}.
\]
This represents an improvement over \eqref{eq:CW} in the previous argument. It immediately gives the corresponding improvement \eqref{eq:HW} over \eqref{eq:CIW}.

Indeed, since $\|f_{T}\|_{L^{\infty}} \sim |T|^{-1/2} \|f_{T}\|_{L^{2}} \sim \lambda$, we have
\[\sum_{T \in \T'} \|f_T\|_{L^2(\R^n)}^2 \frac{1}{|T|}\int_{T}|\tilde{f}|^{p} \sim \lambda^2 \sum_{T \in \T'} \int_{T}|\tilde{f}|^{p} = \lambda^2  \int_{B_R}|\tilde{f}|^{p} \sum_{T \in \T'} 1_T
\]
\[ \sim \lambda^2  \int_{B_R}|\tilde{f}|^{p} \frac{1}{\|f_T\|_\infty^2}\sum_{T \in \T'} |f_T|^2
\sim \int_{B_R}|\tilde{f}|^{p} \sum_{T \in \T'} |f_T|^2,
\]
and the claim and then \eqref{eq:GMW} follow immediately.
\end{rem}

\section{Consequences of weighted formulation of refined decoupling for curves}\label{sec:rdc_weighted_applns}
Throughout this section we shall be making the standing assumptions that $f=\sum_{T \in \T} f_T$ is a Schwartz function which is Fourier supported in the curvature sleeve $\Theta_{\Gamma}(R^{-1/n})$ of a well-curved $C^{n+1}$ curve $\Gamma$ as described at the beginning of Section~\ref{sec:refineddecoupling_curves}, and that $w$ is any non-negative weight. Moreover we shall fix $p = n(n+1)$ and $\frac{1}{r}=1-\frac{2}{p}$. 

We first give a simple consequence of Theorem~\ref{main_corrected_reprise}, and all of our subsequent results will in turn be obtained from this consequence. Indeed, simply using the fact that 
$$\sum_{T \in \mathbb{T}} \|f_T\|_{L^2(\R^n)}^2 \sim \|\sum_{T \in \mathbb{T}}f_T\|_2^2,$$ 
we have as a direct corollary of Theorem~\ref{main_corrected_reprise}:
\begin{cor}\label{cor:fat_tubes_cor}
For every $\epsilon > 0$ and $K \gg 1$ there is a constant $C_{\epsilon,K}$ such that: 

(a) We have
\begin{equation} \label{eq:weightedwp2_prior}
\int_{\R^n} |f(x)|^2 w(x) {\rm d} x \leq C_{\epsilon,K} R^{\epsilon} R^{-\frac{n+1}{2r}}  \sup_{T \in \T} \Big(w^r(R^{\epsilon}T) + R^{-K} \sup_{m \geq 1} 2^{-mK} w^r(2^m T) \Big)^{1/r} \|f\|_{L^2(\R^n)}^2.
\end{equation}

(b) We have
\begin{equation} \label{eq:weightedwp2_r=1_prior}
\int_{\R^n} |f(x)|^2 w(x) {\rm d} x \leq C_{\epsilon,K} R^{\epsilon} R^{-\frac{n+1}{2r}} \sup_{T \in \T} \Big(w(R^{\epsilon}T) + R^{-K}\sup_{m \geq 1} 2^{-mK} w(2^m T)\Big)\|f\|_{L^2(\R^n)}^2.
\end{equation}
Moreover inequality~\eqref{eq:weightedwp2_r=1_prior} is sharp in the sense that the inequality
\[
\int_{B_R} |f(x)|^2 w(x) {\rm d} x 
\lesssim R^{b} \, \sup_{T \in \mathbb{T}} w(T) \, \|f\|_{L^2(\R^n)}^2
\]
fails for any $b < -(n+1)/2r$.

\end{cor}

To obtain inequality \eqref{eq:weightedwp2_r=1_prior} of assertion (b), we may assume by performing a suitable convolution (see also Remark~\ref{rem:pla}) that $w$ is essentially constant on unit scale, and we may then apply assertion (a). A similar remark applies in the subsequent corollaries below.

\begin{rem}\label{rem:sharp_fat_tubes_cor}
Since for any $T$ we have $w^r(T) \lesssim R^{(n+1)/2} \|w\|_\infty^r$, one sees that inequality \eqref{eq:weightedwp2_prior} improves on the trivial one with $\|w\|_\infty$ on the right hand side, (modulo the $R^\epsilon$-loss). 

One may also consider necessary conditions on the exponents $b$ and $r$ (temporarily removing the standing assumption $r = (n(n+1)/2)'$) in order that the inequality 
\begin{equation}\label{eq:test}
\int_{B_R} |f(x)|^2 w(x) {\rm d} x 
\lesssim R^{b} \, \Big(\sup_{T \in \mathbb{T}} w^r(T)\Big)^{1/r} \, \|f\|_{L^2(\R^n)}^2
\end{equation}
should hold. 

By testing on $w = \Phi(\cdot/R)$ for a normalised bump function $\Phi$ we see that necessarily any $b$ for which this holds 
must satisfy $b + (n+1)/2r \geq 0$.

Furthermore, we may also test \eqref{eq:test} on $w = 1_{B(0,1)}$ and $f$ a Schwartz function with $\widehat{f}$ non-negative and supported in the curvature sleeve $\Theta_{\Gamma}(R^{-1/n})$ around $\Gamma$ (so it is essentially a sum of normalised bump functions of height $1$ associated to $R^{1/n}$ disjointly supported tubes of sides $R^{-1/n} \times \dots \times R^{-n/n}$). In this case, $f(0) = \int \widehat{f} \sim \int |\widehat{f}|^2 = \int|f|^2 \sim R^{1/n} \times R^{-1/n} \times \dots \times R^{-n/n} = R^{-((n+1)/2 - 1/n)}$  and moreover $f$ is constant at scale $1$. Thus we also have the necessary condition $b +  (n+1)/2 - 1/n \geq 0$, irrespective of the value of $r$. 

These two necessary conditions meet when $r = (n(n+1)/2)'$, and thus Corollary~\ref{cor:fat_tubes_cor} (a) is rather sharp in the sense that it implies all possible inequalities of the form \eqref{eq:test} (up to $R^\vep$ losses).  

In particular, the powers on $R$ appearing in Corollary~\ref{cor:fat_tubes_cor} (a) and (b) are both sharp. 

\end{rem}

Two very simple geometric observations motivate and indeed give rise to the results which follow next. Recall that if $T \in \T$ is a wave packet tube with directions corresponding to the Serret--Frenet frame at some point of $\Gamma$, then it has lengths $R^{1/n}$ in the tangential direction, $R^{2/n}$ in the normal direction, etc., and $R^{n/n}=R$ in the final Serret--Frenet direction at the corresponding point of $\Gamma$.
These two observations are:

\begin{itemize}
\item Any such $T$ may be sliced into disjoint $R^{1/n}$ parallel planks $L \in \mathcal{L}(T)$ of dimensions $1 \times R^{2/n} \times \dots \times R^{n/n}$, with the side of length $1$ parallel to the corresponding tangent to $\Gamma$. Denote the collection of all planks $L$ of dimensions $1 \times R^{2/n} \times \dots \times R^{n/n}$, with the side of length $1$ parallel to the tangent to $\Gamma$ at some point by $\mathcal{L}$.

\item Each such $L$ may be decomposed into $R^{2/n} \times \dots \times R^{(n-1)/n}= R^{(n-1)/2 - 1/n}$ $1$-tubes $P \in \mathcal{P}(K)$ whose long direction (of dimension $R$) is parallel to the final Serret--Frenet direction at the corresponding point of $\Gamma$. Similarly, each $R^\epsilon L$ may be decomposed into $R^{n \epsilon} \times R^{2/n} \times \dots \times R^{(n-1)/n}= R^{(n-1)/2 - 1/n + n \epsilon}$ $1$-tubes $P \in \mathcal{P}(R^\epsilon L)$.  Denote the collection of all $1$-tubes $P$ with long direction parallel to the final Serret--Frenet direction to $\Gamma$ at some point by $\mathcal{P}$.

\end{itemize}

Combining Corollary~\ref{cor:fat_tubes_cor} (a) with the first of these observations immediately gives inequality \eqref{eq:weightedwp2_K} of:

\begin{cor}\label{cor:fat_tubes_cor_for_tubeslabs}
For every $\epsilon > 0$ and $K \gg 1$ there are constants $C_{\epsilon,K}$ and $C_\epsilon$ such that:

(a) We have
\begin{align}\label{eq:weightedwp2_K}
\int_{\R^n} |f(x)|^2 w(x) {\rm d} x 
&\leq C_{\epsilon,K} R^{\epsilon} R^{-\frac{n+1}{2r}} R^{\frac{1}{nr}} 
\sup_{T \in \mathbb{T}} \sup_{L \in \mathcal{L}(T)} \Big(w^r(R^{\epsilon}L) + {\rm RapDec}_K(R)\Big)^{1/r} \|f\|_{L^2(\R^n)}^2 \\
\label{eq:weightedwp2_K_again}
& \leq C_{\epsilon} R^{\epsilon} R^{-\frac{n+1}{2r}} R^{\frac{1}{nr}} \Big(\sup_{L \in \mathcal{L}} w^r(L)\Big)^{1/r} \|f\|_{L^2(\R^n)}^2
\end{align} 

 where the term ${\rm RapDec}_K(R)$ denotes $R^{-K} \sup_{m \geq 1} 2^{-mK} w^r(2^mL)$.

(b) We have
\begin{align}\label{eq:weightedwp2_K_r=1}
\int_{\R^n} |f(x)|^2 w(x) {\rm d} x & \leq C_{\epsilon,K} R^{\epsilon} R^{-\frac{n+1}{2r}} R^{\frac{1}{nr}} \sup_{T \in \mathbb{T}} \sup_{L \in \mathcal{L}(T)} \Big(w(R^{\epsilon}L) + {\rm RapDec}_\epsilon(R)\Big) \|f\|_{L^2(\R^n)}^2 \\
\label{eq:weightedwp2_K_r=1_again}
& \leq C_{\epsilon} R^{\epsilon} R^{-\frac{n+1}{2r}} R^{\frac{1}{nr}} \sup_{L \in \mathcal{L}} w(L) \|f\|_{L^2(\R^n)}^2
\end{align}
where the term ${\rm RapDec}_K(R)$ here denotes
$R^{-K} \sup_{m \geq 1} 2^{-mK} w(2^mL)$.
\end{cor}

The example $w = \Phi(\cdot/R)$ for a normalised bump function $\Phi$ (the first example in Remark~\ref{rem:sharp_fat_tubes_cor}) demonstrates the sharpness of part (a).

In subsequent corollaries we shall make similar use of the generic notation 
${\rm RapDec}_K(R)$ to indicate a natural rapidly decaying error term in the weight $w$ without explicit identification or further comment. 

Combining Corollary~\ref{cor:fat_tubes_cor_for_tubeslabs} (a) with the second of these observations immediately gives inequality \eqref{eq:weightedwp2_P} of:

\begin{cor}\label{cor:tubes_cor}
For every $\epsilon > 0$ and $K \gg 1$ there are constants $C_{\epsilon,K}$ and $C_\epsilon$ such that: 

(a) We have
\begin{align} \label{eq:weightedwp2_P}
\int_{\R^n} |f(x)|^2 w(x) {\rm d} x & \leq C_{\epsilon,K} R^{\epsilon} R^{-\frac{1}{r}} \sup_{T \in \mathbb{T}} \sup_{L \in \mathcal{L}(T)} \sup_{P \in \mathcal{P}(R^\epsilon L)} \Big(w^r(P) +  {\rm RapDec}_K(R) \Big)^{1/r} \|f\|_{L^2(\R^n)}^2\\
\label{eq:weightedwp2_P_again}
& \leq C_{\epsilon} R^{\epsilon} R^{-\frac{1}{r}}  \sup_{P \in \mathcal{P}} \Big(w^r(P) \Big)^{1/r} \|f\|_{L^2(\R^n)}^2.
\end{align}

(b) We have
\begin{align}\label{eq:weightedwp2_P_r=1}
\int_{\R^n} |f(x)|^2 w(x) {\rm d} x &\leq C_{\epsilon,K} R^{\epsilon} R^{-\frac{1}{r}} 
\sup_{T \in \mathbb{T}} \sup_{L \in \mathcal{L}(T)} \sup_{P \in \mathcal{P}(R^\epsilon L)} \Big(w(P) +  {\rm RapDec}_K(R) \Big)
\|f\|_{L^2(\R^n)}^2 \\
\label{eq:weightedwp2_P_r=1_again_prior}
&\leq C_{\epsilon} R^{\epsilon} R^{-\frac{1}{r}} 
 \sup_{P \in \mathcal{P}} w(P)\|f\|_{L^2(\R^n)}^2.
\end{align}
\end{cor}

\begin{rem}%
Since for any $P$ we have $w^r(P) \lesssim R \|w\|_\infty^r$, one sees that inequality \eqref{eq:weightedwp2_P} improves the trivial one with $\|w\|_\infty$ on the right hand side. Once again, the first example in Remark~\ref{rem:sharp_fat_tubes_cor}, that is $w = \Phi(\cdot/R)$ for a normalised bump function $\Phi$, demonstrates the sharpness of part (a).  
 
We do not know whether parts (b) in Corollary~\ref{cor:fat_tubes_cor_for_tubeslabs} and \ref{cor:tubes_cor}
have the sharp power of $R$ on the right hand side. However, if we formulate Corollary~\ref{cor:fat_tubes_cor_for_tubeslabs} and \ref{cor:tubes_cor} in the more general setting of axiomatic decoupling (see Section~\ref{sec:axdec} below), then the inequalities corresponding to \eqref{eq:weightedwp2_K_r=1_again} and \eqref{eq:weightedwp2_P_r=1_again_prior} do have the sharp powers $R^{-\frac{n+1}{2r} + \frac{1}{nr}}$ and $R^{-\frac{1}{r}}$ on the right hand side respectively (see Theorem~\ref{thm:axiomatic_sharpness} (i) (ii)). When $n=2$, they are both essentially the content of \cite{Gu22}, see also \cite{CIW24}.

\end{rem}

{
Combining Corollary~\ref{cor:fat_tubes_cor_for_tubeslabs} (a) with the further observation that each $L$ is contained in a $1$-neighbourhood of a hyperplane $S(L)$ with normal parallel to the corresponding tangent to $\Gamma$ immediately gives inequality \eqref{eq:weightedwp2_S_prior} of:

\begin{cor}\label{eq:primordial_prior}\label{thm:main_ingredient}
Let $\mathcal{S}_\epsilon(T)$ denote the collection of $1$-neighbourhoods $S$ of hyperplanes with normals in the direction of side $R^{1/n}$ of the wave packet tube $T$, and with $S$ meeting $R^\epsilon T$, and let $\mathcal{S}$ denote the set of all $1$-neighbourhoods of hyperplanes with normal parallel to the tangent to $\Gamma$ at some point. For every $\epsilon > 0$ and $K \gg 1$ there are constants $C_{\epsilon,K}$ and $C_\epsilon$ such that: 

(a) We have
\begin{align} \label{eq:weightedwp2_S_prior}
\int_{\R^n} |f(x)|^2 w(x) {\rm d} x &\leq C_{\epsilon,K} R^{\epsilon} R^{-\frac{n+1}{2r}}R^{\frac{1}{nr}}  \sup_{T \in \mathbb{T}} \sup_{S \in \mathcal{S}_\epsilon(T)} \Big( w^r(S) + {\rm RapDec}_K(R) \Big)^{1/r} \|f\|_{L^2(\R^n)}^2\\
\label{eq:weightedwp2_S_again_prior}
&\leq C_{\epsilon} R^{\epsilon} R^{-\frac{n+1}{2r}}R^{\frac{1}{nr}} \sup_{S \in \mathcal{S}} \Big( w^r(S) \Big)^{1/r} \|f\|_{L^2(\R^n)}^2.
\end{align}
Moreover inequality \eqref{eq:weightedwp2_S_again_prior} is sharp\footnote{Upon noting that with $p=n(n+1)$ and $\frac{1}{r} + \frac{2}{p} =1$ we have 
$-\frac{n+1}{2r} + \frac{1}{nr} = -\frac{(n+1)}{2} + \frac{2}{n} - \frac{2}{n^2(n+1)}.$
} in the sense that the inequality
\[
\int_{B_R} |f(x)|^2 w(x) {\rm d} x 
\lesssim R^{b} \, \Big(\sup_{S\in \mathcal{S}} w^r(S)\Big)^{1/r} \, \|f\|_{L^2(\R^n)}^2
\]
fails for any $b < -\frac{(n+1)}{2} + \frac{2}{n} - \frac{2}{n^2(n+1)}$.

(b) We have
\begin{align}\label{eq:weightedwp2_S_r=1}
\int_{\R^n} |f(x)|^2 w(x) {\rm d} x & \leq C_{\epsilon,K} R^{\epsilon} R^{-\frac{n+1}{2r}} R^{\frac{1}{nr}}  \sup_{T \in \mathbb{T}} \sup_{S \in \mathcal{S}_\epsilon(T)} \Big( w(S) + {\rm RapDec}_K(R) \Big) \|f\|_{L^2(\R^n)}^2 \\
\label{eq:weightedwp2_S_r=1_again}
&\leq C_{\epsilon} R^{\epsilon} R^{-\frac{n+1}{2r}} R^{\frac{1}{nr}} \sup_{S \in \mathcal{S}} w(S)  \|f\|_{L^2(\R^n)}^2.
\end{align}
\end{cor}

\begin{rem}\label{rem:sharpness_implicns}
When $n \geq 3$, the first example from Remark~\ref{rem:sharp_fat_tubes_cor} ($w = \Phi(\cdot/R)$) no longer suffices to demonstrate the sharpness assertion of part (a). Instead  
we simply take $f = f_{T_0}$ to be a single wave packet and let $w= 1_{T_0}$. This same example, or indeed the second example  from Remark~\ref{rem:sharp_fat_tubes_cor} (a single bush), shows that  the exponent on $R$ in part (b) cannot be replaced by anything smaller than $-(n+1)/2r$. The discrepancy here is merely $R^{\frac{1}{nr}}$, in which the exponent tends to $0$ as $n \rightarrow \infty$. Thus we may regard Corollary~\ref{thm:main_ingredient} (b) as ``asymptotically sharp''.

The nature of the arguments we have used demonstrates the following two observations: (i) any assertion of sharpness of part (b) in any of the previous Corollaries~\ref{cor:fat_tubes_cor}, ~\ref{cor:fat_tubes_cor_for_tubeslabs}, ~\ref{cor:tubes_cor} or \ref{eq:primordial_prior} also implies sharpness of the corresponding part (a) (but not vice versa); and (ii) any assertion of sharpness of part (a) in either Corollary~\ref{cor:tubes_cor} or Corollary~\ref{eq:primordial_prior} implies sharpness of part (a) in Corollary~\ref{cor:fat_tubes_cor_for_tubeslabs}, which in turn implies sharpness of part (a) of Corollary~\ref{cor:fat_tubes_cor}. 
\end{rem}

\section{Applications to Mizohata--Takeuchi-type results}\label{MT_results}
In this section we apply the results of the previous section in the context of the extension operator for the Fourier transform to obtain Theorems~\ref{thm:headline} and \ref{thm:headline_companion}. Indeed, each of the four results Corollary~\ref{cor:fat_tubes_cor}, ~\ref{cor:fat_tubes_cor_for_tubeslabs}, ~\ref{cor:tubes_cor} and \ref{eq:primordial_prior}
has a direct counterpart for the extension operator, and  which is a direct consequence of the corresponding result in the previous section, and we continue to use the notation $\mathcal{L}$, $\mathcal{P}$ and $\mathcal{S}$ as in the previous section. As before, $w$ will continue to denote an arbitrary non-negative weight.

The mechanism for passing from the setting of the previous section to the current one is as follows. 

Let $\Gamma$ be a well-curved $C^{n+1}$ curve, let $g \in L^2(\Gamma)$, and let $E$ denote the Fourier extension operator $E: g \mapsto \widehat{ g\,{\rm d}\ld}$ where $ {\rm d}\lambda$ is the induced $1$-dimensional measure on the curve $\Gamma$. First, we pick some Schwartz function $\Phi_1$ with compact Fourier support in a ball of radius $R^{-1}$, so that $|\Phi_1| \gtrsim 1$ on $B_R$. Let $f := \Phi_1 Eg$. Then
\[
\int_{B_R} |Eg|^2 w  \lesssim \int_{\R^n} |f|^2 w.
\]
Note that $f$ is a Schwartz function of compact Fourier support in the $R^{-1}$-neighbourhood of $\Gamma$, and thus in the curvature sleeve $\Theta_{\Gamma}(R^{-1/n})$ around $\Gamma$.

We are therefore in a position to estimate $\int_{\R^n} |f|^2 w$ using any of the results Corollary~\ref{cor:fat_tubes_cor}, ~\ref{cor:fat_tubes_cor_for_tubeslabs}, ~\ref{cor:tubes_cor} and \ref{eq:primordial_prior} from the previous section. Each of these has the form 
$$\int_{\R^n} |f|^2 w \lesssim Q_R(w) \int_{\mathbb{R}^n} |f|^2$$
where the geometric functional $Q_R(w)$ of the weight $w$ varies from case to case. We pass back to $g$ using 
the Agmon--H\"{o}rmander inequality \eqref{eq:AH_curves_intro} which implies 
that 
\begin{equation*}%
\int_{\R^n} |f|^2 \lesssim \int_{\R^n}|Eg|^2 |\Phi_1|^2\lesssim R^{n-1} \int_{\Gamma} |g|^2
\end{equation*}
as may be seen by dyadically decomposing $\R^n$ into $B_R$ and annuli $\{|x| \sim 2^kR$\} for $k \geq 0$ and using the rapid decay of $\Phi_1$ outside $B_R$ to sum in $k$. The upshot is that any inequality of the form
$$\int_{\R^n} |f|^2 w \lesssim Q_R(w) \int |f|^2$$
coming from the previous section immediately implies a corresponding weighted inequality
$$\int_{B_R} |Eg|^2 w \lesssim R^{n-1} Q_R(w) \int |g|^2\, {\rm d} \lambda$$
for the Fourier extension operator. Previously-explored implications of this nature in the case of strictly convex hypersurfaces may be found in Theorem 1 of \cite{CSV07}. (In fact, in that setting, the argument may also be reversed.)

For example, the counterpart of Corollary~\ref{eq:primordial_prior} is, (where 
$1/r + 2/n(n+1) =1$ as usual):} 
\begin{theorem}\label{eq:primordial_prior_X}\label{thm:main_ingredient_X}
Let $\Gamma$ be a well-curved $C^{n+1}$ curve, let $\mathcal{S}$ denote the set of all $1$-neighbourhoods of hyperplanes with normal parallel to the tangent to $\Gamma$ at some point, and let $\mathcal{S}_\epsilon(T)$ denote the collection of $1$-neighbourhoods $S$ of hyperplanes with normals in the direction of side $R^{1/n}$ of the wave packet tube $T$, and with $S$ meeting $R^\epsilon T$. For every $\epsilon > 0$ and $K \gg 1$ there are constants $C_{\epsilon,K}$ and $C_\epsilon$ such that: 

(a) For $w$ an arbitrary weight we have %
\begin{align}\label{eq:weightedwp2_S}
\int_{B_R} |Eg(x)|^2 w(x) {\rm d} x &\leq C_{\epsilon,K} R^{\epsilon} R^{n-1} R^{-\frac{n+1}{2r}}R^{\frac{1}{nr}}  \sup_{T \in \mathbb{T}} \sup_{S \in \mathcal{S}_\epsilon(T)} \Big( w^r(S) + {\rm RapDec}_\epsilon(R) \Big)^{1/r} \|g\|_{L^2({\rm d} \lambda)}^2 \\
\label{eq:weightedwp2_S_again}
&\leq C_{\epsilon} R^{\epsilon} R^{n-1} R^{-\frac{n+1}{2r}}R^{\frac{1}{nr}} \sup_{S \in \mathcal{S}} \Big( w^r(S)  \Big)^{1/r} \|g\|_{L^2({\rm d} \lambda)}^2.
\end{align}

(b) For $w$ an arbitrary weight, we have
\begin{align} \label{eq:weightedwp2_S_r=1_X}
\int_{B_R} |Eg(x)|^2 w(x) {\rm d} x &\leq C_{\epsilon,N} R^{\epsilon} R^{n-1} R^{-\frac{n+1}{2r}} R^{\frac{1}{nr}}  \sup_{T \in \mathbb{T}} \sup_{S \in \mathcal{S}_\epsilon(T)} \Big( w(S) + {\rm RapDec}_\epsilon(R) \Big) \|g\|_{L^2({\rm d} \lambda)}^2 \\
\label{eq:weightedwp2_S_r=1_again_X}
&\leq C_{\epsilon} R^{\epsilon} R^{n-1} R^{-\frac{n+1}{2r}} R^{\frac{1}{nr}}  \sup_{S \in \mathcal{S}} w(S) \|g\|_{L^2({\rm d} \lambda)}^2.
\end{align}

\end{theorem}
Since when $p=n(n+1)$ and $1/r + 2/p = 1$ we have
$(n-1) {-\frac{n+1}{2r}} +{\frac{1}{nr}} = \frac{n-3}{2} + \frac{2}{n} - \frac{2}{n^2(n+1)}$, inequality~\eqref{eq:weightedwp2_S_again} of this result is just  
Theorem~\ref{thm:headline} (see also Remark~\ref{rem:pla}), and inequality~\eqref{eq:weightedwp2_S_r=1_again} is just Corollary~\ref{cor:thm_headline}.

Similarly, Theorem~\ref{thm:headline_companion} follows from the extension counterpart of Corollary~\ref{cor:tubes_cor} via the above discussion.

\begin{rem}\label{rem:diff_various_sharpness}
When $n=2$, boundedness inequalities and questions of sharpness of the various results in the context of Fourier extension are essentially equivalent to the corresponding ones in the setting  of refined decoupling. Thus for example, Theorem~\ref{eq:primordial_prior_X} (a) is sharp when $n=2$, because we know that Corollary~\ref{eq:primordial_prior} (a) is sharp, see Remark~\ref{rem:sharpness_implicns}. In the current context this can be seen directly by taking $g$ to be a smoothed version of the characteristic function of an $R^{-1/2}$-arc and $w$ to be the characteristic function of the corresponding dual wave packet tube. See also the discussion of the sharpness of Theorem~\ref{thm:headline_companion} in Remark~\ref{rem:confusion}. However, when $n \geq 3$, questions of sharpness of the various results in the context of Fourier extension are more challenging than the corresponding ones in the setting  of refined decoupling. This is because the natural examples which present themselves for testing the results in the former context (refined decoupling) do not arise in the latter setting (Fourier extension) when $n \geq 3$, as discussed in Footnote 1 on Page \pageref{footnote1}. We do not know if Theorem~\ref{thm:main_ingredient_X} (a) (or equivalently Theorem~\ref{thm:headline}) is sharp when $n \geq 3$; the single wave packet example which shows that Corollary~\ref{thm:main_ingredient} (a) is sharp does not seem to directly relevant in the setting of Theorem~\ref{thm:main_ingredient_X} (a) when $n \geq 3$.

We may also consider the extension counterparts of Corollary~\ref{cor:fat_tubes_cor}, Corollary~\ref{cor:fat_tubes_cor_for_tubeslabs} and Corollary~\ref{cor:tubes_cor} for sharpness.

Each of the parts (a) is sharp as may be seen by taking 
$w = 1_{B_R}$ and $g$ constant on scale $R^{-1}$, but otherwise arbitrary.

Consider the extension counterpart of Corollary~\ref{cor:fat_tubes_cor} (b). When $n=2$ it is sharp because Corollary~\ref{cor:fat_tubes_cor} (b) itself is sharp, see Remark~\ref{rem:sharp_fat_tubes_cor}, but the example showing this does not arise from any $g \in L^2(\Gamma)$ when $n \geq 3$. It therefore seems to be open as to whether 
$$ \int|Eg|^2 w \lesssim R^{\epsilon} R^{n-1} R^{-\frac{n+1}{2r}} \sup_T w(T) \int |g|^2$$
is sharp when $n \geq 3$. 

Now consider the extension counterpart of Corollary~\ref{cor:fat_tubes_cor_for_tubeslabs} (b) for planks $L$. It seems that sharpness of 
$$ \int|Eg|^2 w \lesssim R^{\epsilon} R^{n-1} R^{-\frac{n+1}{2r}} R^{\frac{1}{nr}}\sup_L w(L) \int |g|^2$$
is open for all  $n \geq 2$. (When $n=2$ it reduces to consideration of  
$$ \int|Eg|^2 w \lesssim R^{\epsilon} R^{\frac{1}{3}} \sup_S w(S) \int |g|^2$$
and sharpness of this inequality is not understood.)

The extension counterpart of Corollary~\ref{cor:tubes_cor} is addressed in the discussion of Theorem~\ref{thm:headline_companion} and Corollary~\ref{cor:headline_companion} in the introduction.

\end{rem}

\section{Axiomatic decoupling for well-curved curves}\label{sec:axdec}
It has been previously observed in the setting of hypersurfaces in place of curves (see \cite{Gu22}, \cite{CIW24}), that results analogous to those of Sections~\ref{sec:refineddecoupling_curves} and \ref{sec:rdc_weighted_applns} will actually hold in a somewhat more general axiomatic setting than we have presented above. This is essentially because 
much of the theory of decoupling goes through in this more general axiomatic setting.

One may therefore speculate that our own results from Sections~\ref{sec:refineddecoupling_curves} and \ref{sec:rdc_weighted_applns}
may also actually hold in a certain axiomatic framework which is adapted to well-curved curves in place of hypersurfaces. We shall formulate plausible versions of some of those results in this more general context, and this will lead us to questions of sharpness within this more general context. These questions about sharpness are the main concern of this section. 

We describe the set-up for this framework. Instead of being concerned with Schwartz functions $f$ which are Fourier supported in the curvature sleeve $\Theta_{\Gamma}(R^{-1/n})$ of the curve $\Gamma$, we shall instead be concerned with a wider class of functions $F: \R^n \to \C$ which have a certain internal structural property, which we may describe as saying that $F$ {\em admits an axiomatic decoupling structure} with respect to $\Gamma$ at scale $R$.  

We need some notation. Fix $R \gg 1 $ and $\epsilon > 0$. We call a subinterval $\tau \subset [0,1]$, \textit{admissible} if its diameter $d(\tau)$ is a dyadic number in $[R^{-1/n}, R^{-\vep}] \cup \{1\}$. In this analysis, $[0,1]$ is the only admissible subinterval of length $1$. Denote by $\mc{D}_R$ the set of all admissible subintervals. For every $\tau \in \mc{D}_R$, let $F_\tau: \R^n \rightarrow \C$ be some measurable function. 

\begin{defn}%
{\rm We say that $F:\R^n \to \C$ {\em admits an axiomatic decoupling structure} with respect to the well-curved curve $\Gamma$ at scale $R$ if for each $\tau \in \mc{D}_R$ there exists a measurable function $F_\tau: \R^n \to \C$ such that 
\begin{itemize}
\item $F = F_{[0,1]}$

\medskip

\item (DA0) For $\gamma \in \mc{D}_R$, if  $\gamma = \bigsqcup_{\tau \in \mf{T}} \tau$ where $\mf{T}$ is a family of non-overlapping subintervals in $\mc{D}_R$, then
\begin{align*}
    |F_\gamma| \lesssim \sum_{\tau \subset \gamma}|F_\tau|.
\end{align*}

\medskip

\item (DA1) (Local constancy). For each $\tau \in \mc{D}_R$ the function $|F_\tau|$ is essentially constant on each translate of 
    \[
    \Gamma(\tau)^* \coloneqq \{x\in\R^n: |x\cdot(\xi-\xi_\tau)| \leq 1 \text{ for all } \xi\in \Gamma(\tau)\}
    \]
    where $\Gamma(\tau)$ denotes the part of $\Gamma$ above $\tau$, and $\xi_\tau$ denotes the centre of $\Gamma(\tau)$.

\medskip

\item (DA2) (Local $L^2$-orthogonality). For $\gamma \in \mc{D}_R$, if  $\gamma = \bigsqcup_{\tau \in \mf{T}} \tau$ where $\mf{T}$ is a family of non-overlapping subintervals in $\mc{D}_R$, %
and if $K$ is a symmetric convex body such that the sets 
$\{ \Gamma(\tau) + K^*\}_{\tau \in \mf{T}}$\footnote{Here $K^*$ denotes the dual convex body of $K$, i.e., if $c$ is the centre of $K$, then $K^* \coloneqq \{x \in R^n : |x\cdot(\xi - c)| \leq 1 \text{ for all } \xi\in\Gamma(\tau)\}$.} are boundedly overlapping, then
the estimate 
    \[
    \int_K |F_\gamma|^2 \lesssim \sum_{\tau \subset \gamma} \int_K |F_\tau|^2 %
    \]
    holds, where the implied constant depends on the bounded overlap.

\end{itemize}

}
\end{defn}

Note that  $\Gamma(\tau)^*$ is essentially the dual rectangle to the convex hull of $\Gamma(\tau)$. 

Note that we have introduced the additional axiom (DA0) in comparison with \cite{Gu22} and \cite[Section~7]{CIW24}; we have done so in order to facilitate the broad-narrow argument
which is needed to realise decoupling.

Examples of such $F$ which admit an axiomatic decoupling structure include Schwartz functions $f$ which are Fourier supported in an $R^{-1}$-neighbourhood of $\Gamma$, and special cases of these are $f$'s of the form $f = \widehat{g \,{\rm d} \lambda} \, \Phi( \cdot/R)$ for $g \in L^2({\rm d} \lambda)$; but we shall be specifically {\em not} be concerned with such special cases in this section.

The results of Sections~\ref{sec:refineddecoupling_curves} and \ref{sec:rdc_weighted_applns} most likely remain valid in this more general setting.
In particular, the generalisation of Corollary~\ref{cor:fat_tubes_cor} to this setting is (recalling that $p=n(n+1)$ and $r =(p/2)'$):

\begin{conjthm}%
Let $\Gamma$ be a  well-curved curve, and that
$F$ admits an axiomatic decoupling structure with respect to $\Gamma$ at scale $R$. Let $\T = \bigsqcup_{\tau \in \mathcal{D}_R} \T_{\tau}$ where $\T_{\tau}$ is a tiling of $\R^n$ by translates of $\Gamma(\tau)^*$.
(a) Then we have
\begin{equation*} %
\int_{\R^n} |F(x)|^2 w(x) {\rm d} x \leq C_{\epsilon,K} R^{\epsilon} R^{-\frac{n+1}{2r}}  \sup_{T \in \T} (w^r(T) )^{1/r} \|F\|_{L^2(\R^n)}^2.
\end{equation*}

(b) We also have
\begin{equation*} %
\int_{\R^n} |F(x)|^2 w(x) {\rm d} x \leq C_{\epsilon,K} R^{\epsilon} R^{-\frac{n+1}{2r}} \sup_{T \in \T} w(T)\|F\|_{L^2(\R^n)}^2.
\end{equation*}

\end{conjthm}
From this we may deduce results analogous to Corollaries~\ref{cor:fat_tubes_cor_for_tubeslabs}, \ref {cor:tubes_cor} and \ref{thm:main_ingredient} in this more general setting, in exactly the same way as we did in Section~\ref{sec:rdc_weighted_applns}, with the notation $\mathcal{L}$, $\mathcal{P}$ and $\mathcal{S}$ as defined there. Thus, focusing on the parts (b) of those results, we have

\begin{conjcor}\label{cor:fat_tubes_cor_for_tubeslabs_axiomatic}
Suppose that $F$ admits an axiomatic decoupling structure with respect to the well-curved curve $\Gamma$ at scale  $R$. Then we have
\begin{enumerate}
\item[(i)] 
\begin{equation*}%
\int_{\R^n} |F(x)|^2 w(x) {\rm d} x \leq C_{\epsilon} R^{\epsilon} R^{-\frac{n+1}{2r}} R^{\frac{1}{nr}} \sup_{L \in \mathcal{L}} w(L) \|F\|_{L^2(\R^n)}^2.
\end{equation*}
\item[(ii)] 
\begin{equation*}%
\int_{\R^n} |F(x)|^2 w(x) {\rm d} x \leq C_{\epsilon} R^{\epsilon} R^{-\frac{1}{r}} 
 \sup_{P \in \mathcal{P}} w(P)\|F\|_{L^2(\R^n)}^2.
\end{equation*}

\item[(iii)]

\begin{equation*}%
\int_{\R^n} |F(x)|^2 w(x) {\rm d} x \leq C_{\epsilon} R^{\epsilon} R^{-\frac{n+1}{2r}} R^{\frac{1}{nr}} \sup_{S \in \mathcal{S}} w(S)  \|F\|_{L^2(\R^n)}^2.
\end{equation*}
\end{enumerate}

\end{conjcor}

Since we are not able to say anything about the sharpness of parts (b) of Corollaries~\ref{cor:fat_tubes_cor_for_tubeslabs}, \ref {cor:tubes_cor} and \ref{thm:main_ingredient} as they are stated, in the rest of this section we instead address the sharpness of each of the conclusions (i), (ii) and (iii) of Plausible Corollary~\ref{cor:fat_tubes_cor_for_tubeslabs_axiomatic}. The hope is that its wider scope, and the more bountiful supply of functions admitting axiomatic decoupling structure to which it applies, will yield interesting $F$'s (which are not available in the setting of the previous section) upon which to test it. Indeed, in the case $n=2$, this is exactly what Guth did in \cite{Gu22}, and furthermore all three statements (which coincide) are sharp. So we shall now focus on the case $n \geq 3$.

\begin{rem}\label{rem:axiomatiacally_sharp}
For the sake of clarity, when we say that ``$R^{1/3}$-loss in the planar Mizohata--Takeuchi conjecture is sharp modulo axiomatic decoupling'', what we really mean is that in the case $n=2$, the (equivalent) items in Corollary~\ref{cor:fat_tubes_cor_for_tubeslabs_axiomatic} are genuinely sharp.
\end{rem}

The main result of this section is:

\begin{theorem}\label{thm:axiomatic_sharpness}
Let $n \geq 3$, and suppose that $\Gamma$ is a well-curved $C^{n+1}$ curve. 

(i) Suppose that 
\begin{equation*}
\int_{\R^n} |F(x)|^2 w(x) {\rm d} x \lesssim  R^{a} \sup_{L \in \mathcal{L}} w(L) \|F\|_{L^2(\R^n)}^2
\end{equation*}
for all $F$ which admit an axiomatic structure with respect to $\Gamma$ at scale $R$ and all non-negative weights $w$. Then
$$ a \geq -\frac{n+1}{2} + \frac{2}{n} - \frac{2}{n(n+1)} = -\frac{n+1}{2r} + \frac{1}{nr}.$$

(ii) Suppose that 
\begin{equation*}
\int_{\R^n} |F(x)|^2 w(x) {\rm d} x \lesssim  R^{a} \sup_{P \in \mathcal{P}} w(P) \|F\|_{L^2(\R^n)}^2
\end{equation*}
for all $F$ which admit an axiomatic structure with respect to $\Gamma$ at scale $R$ and all non-negative weights $w$. Then
$$ a \geq -1 + \frac{2}{n(n+1)} = -\frac{1}{r}.$$

(iii) Suppose that 
\begin{equation*}
\int_{\R^n} |F(x)|^2 w(x) {\rm d} x \lesssim  R^{a} \sup_{S \in \mathcal{S}} w(S) \|F\|_{L^2(\R^n)}^2
\end{equation*}
for all $F$ which admit an axiomatic structure with respect to $\Gamma$ at scale $R$ and all non-negative weights $w$. Then
$$ a \geq -\frac{n+1}{2} + \frac{1}{n} = -\frac{n+1}{2r}.$$
\end{theorem}

Parts (i) and (ii) of Theorem~\ref{thm:axiomatic_sharpness} imply that parts (i) and (ii) of Plausible  Corollary~\ref{cor:fat_tubes_cor_for_tubeslabs_axiomatic} are sharp up to $R^\epsilon$-loss, which further means that parts (b) of Corollaries~\ref{cor:fat_tubes_cor_for_tubeslabs} and \ref {cor:tubes_cor} are sharp modulo axiomatic decoupling up to $R^\epsilon$-loss. We have already established part (iii) of Theorem~\ref{thm:axiomatic_sharpness} in Remark~\ref{rem:sharpness_implicns}, but we wish to present another argument for it in the spirit of axiomatic decoupling.

The proof of Theorem~\ref{thm:axiomatic_sharpness} relies on examples constructed by Guth's method \cite{Gu22}, and we mainly follow the line of argument in Section 7 of \cite{CIW24}. 

First recall the rescaled version of Theorem 3 in \cite{Ca09}, which is the key combinatorial input in all three cases in Theorem~\ref{thm:axiomatic_sharpness}:
\begin{theorem}[Carbery, \cite{Ca09}] \label{weight_construction}
    Let $N \geq \mu \geq n+2$. Then there is a configuration of $N$ points in $B_R \subset \R^n$ such that the volume of the convex hull of any $\mu$ of these points is at least $C_n (\mu/N)^{(\mu-1)/(\mu-n)} R^n$.
\end{theorem}

\subsection{Proof of (i).}\label{subsec:ax_L_sharp}

Let $\Gamma$ be as earlier. The scale $\ell = R^{\frac{1}{n} - \frac{2}{n^2(n+1)}}$ (such that $\ell \times \ell^2 \times \cdots \times \ell^n = |L|$) will play a key role in the upcoming argument; thus, denote by $\mc{D}$ the set of all $\tau \in \mc{D}_R$ with $d(\tau) = \ell^{-1}$. For each $\tau \in \mc{D}$, let $\T_\tau$ be a family of finitely overlapping planks in $\R^n$ that intersect and cover $B_R$, each of dimension $\ell \times \ell^2 \times \cdots \times \ell^n$ dual to $\Gamma(\tau)$ (i.e., $\T_\tau$ consists essentially of translates of $\Gamma(\tau)^*$). Let
\begin{align*}
    \T_R \coloneqq \bigsqcup_{\tau \in \mc{D}} \T_\tau.
\end{align*}
By taking $N \sim R^{\frac{n-1}{2} + \frac{1}{n}}$, $\mu \sim \log R$ in Theorem~\ref{weight_construction}, we know that any convex hull of volume $\lsm R^{\frac{n+1}{2} - \frac{1}{n}}$ must only contain $\lsm \log R$ points. So we can conclude that there exists a weight $w: \R^n \rightarrow [0, +\infty)$ such that the following hold:
\begin{itemize}
    \item $w$ is the characteristic function of a union of $\sim R^{\frac{n-1}{2} + \frac{1}{n}}$ unit balls in $B_R$;
    \item Any $L \in \mathcal{L}$ satisfies $w(L) \lsm \log R$;
    \item Any $T \in \T_R$ satisfies $w(T) \lsm \log R$, and fully contains every unit ball in $\supp w$ that it intersects.
\end{itemize}

We want to construct a function $F$ as a sum of wave packets that is large on a big proportion of $\supp w$. For this purpose, we need the following claim:
\begin{clm}\label{first_claim}
    There exist
    \begin{itemize}
        \item a set $\mc{B} = \{B_1, \dots, B_m\}$ of $\gtrsim (\log R)^{-2} R^{\frac{n-1}{2} + \frac{1}{n}}$ disjoint unit balls in $\supp w$, and 
        \item sets $\T_j \subset \T_R$ with $\# \T_j \gtrsim \# \mc{D}$ for every $j = 1, \dots, m$, 
    \end{itemize}
    such that the following hold:
    \begin{itemize}
        \item The tubes in $\T_j$ contain $B_j$, for all $j = 1, \dots, m$.
        \item For each $j = 2, \dots, m$, the tubes in $\T_j$ do not intersect any of the balls $B_1, \dots, B_{j-1}$.
    \end{itemize}
\end{clm}
\begin{proof}
    The construction of $\{B_j\}$ and $\{\mathbb{T}_j\}$ as well as the proof are all exactly the same as those for Claim 7.1 in \cite{CIW24}.
\end{proof}

We can now construct a sequence $(F_\tau)_{\tau \in \mc{D}_R}$ of functions $F_\tau: \R^n \rightarrow \C$ as follows:
\begin{itemize}
    \item For each $\tau \in \mc{D}_R$ with $R^{-\frac{1}{n}} \lsm d(\tau) < \ell^{-1}$, define $F_\tau \coloneqq d(\tau)^{\frac{1}{2}} 1_{B_R}$.
    \item For each $\tau \in \mc{D}_R$ with $d(\tau) = \ell^{-1}$ (or, more precisely, for each $\tau \in \mc{D}$), define
    \begin{align*}
        F_\tau \coloneqq \sum_{T \in \T_\tau} c_T e^{-2\pi i \langle \cdot, \xi_\tau \rangle} d(\tau)^{\frac{1}{2}} \phi_T,
    \end{align*}
    where $\phi_T$ is a bump function on $T$ and $\xi_\tau$ is the centre of $\Gamma(\tau)$. The coefficients $c_T \in \C$ will be defined below.
    \item For $\gamma \in \mc{D}_R$ with $\ell^{-1} < d(\gamma) \leq 1$, define
    \begin{align*}
        F_\gamma \coloneqq \sum_{\tau \in \mc{D}, \tau \subset \gamma} F_\tau.
    \end{align*}
\end{itemize}
Let $F \coloneqq F_{[0,1]} = \sum_{\tau \in \mc{D}} F_\tau$. Through exactly the same procedure as that in Section 7 of \cite{CIW24}, by using Claim~\ref{first_claim}, the coefficients $c_T \in \C$ can be chosen to have modulus $1$, and such that 
\begin{align*}
    |F| \gtrsim \ell^{\frac{1}{2}} \text{ on } \bigcup_{B \in \mc{B}} B.
\end{align*}
It is easy to verify that $(F_\tau)_{\tau \in \mc{D}_R}$ satisfies the decoupling axioms (DA0) -- (DA2).

By direct computation, we have
\begin{align*}
    \int_{\R^n} |F|^2w \gtrsim \ell \cdot \# \mc{B} \gtrsim (\log R)^{-2} R^{\frac{n-1}{2} + \frac{2}{n} - \frac{2}{n(n+1)}}
\end{align*}
and
\begin{align*}
    \int_{\R^n} |F|^2 \lsm \sum_{\tau \in \mc{D}} \int_{\R^n} |F_\tau|^2 &\lsm \sum_{\tau \in \mc{D}} \sum_{T \in \T_\tau} \int_T |c_T d(\tau)^{\frac{1}{2}}|^2 \\
    &= \sum_{\tau \in \mc{D}} \sum_{T \in \T_\tau} |T| \cdot |\tau| = |[0,1]| \cdot |B_R|
    \sim R^n.
\end{align*}
Here in the first step of the upper bound above we can use (DA2) by virtue of the fact that $(F_\tau)_{\tau \in \mc{D}_R}$ are concentrated on $B_R$. Thus we obtain

\begin{align*}
    (\log R)^{-3} R^{-\frac{n+1}{2} + \frac{2}{n} - \frac{2}{n(n+1)}} \sup_{L \in \mathcal{L}} w(L) \norm{F}_{L^2(\R^n)}^2
    \lsm \int_{\R^n} |F|^2 w.
\end{align*}

Note that the main term on the left-hand side above is just $R^{-\frac{n+1}{2r} + \frac{1}{nr}}$, so we are done.

\subsection{Proof of (ii).}\label{subsec:ax_P_sharp}
    
The proof is similar to that for (i) in Section~\ref{subsec:ax_L_sharp}. Let $\Gamma$ be as earlier. The scale $\ell = R^{\frac{2}{n(n+1)}}$ (such that $\ell \times \ell^2 \times \cdots \times \ell^n = |P|$) will play a key role in the upcoming argument; thus, denote by $\mc{D}$ the set of all $\tau \in \mc{D}_R$ with $d(\tau) = \ell^{-1}$. For each $\tau \in \mc{D}$, let $\T_\tau$ be a family of finitely overlapping planks in $\R^n$ that intersect and cover $B_R$, each of dimension $\ell \times \ell^2 \times \cdots \times \ell^n$ dual to $\Gamma(\tau)$ (i.e., $\T_\tau$ consists essentially of translates of $\Gamma(\tau)^*$). Let
\begin{align*}
    \T_R \coloneqq \bigsqcup_{\tau \in \mc{D}} \T_\tau.
\end{align*}

By taking $N \sim R^{n-1}$, $\mu \sim \log R$ in Theorem~\ref{weight_construction}, we know that any convex hull of volume $\lsm R$ must only contain $\lsm \log R$ points. So we can conclude that there exists a weight $w: \R^n \rightarrow [0, +\infty)$ such that the following hold:
\begin{itemize}
    \item $w$ is the characteristic function of a union of $\sim R^{n-1}$ unit balls in $B_R$;
    \item Any $P \in \mathcal{P}$ satisfies $w(P) \lsm \log R$;
    \item Any $T \in \T_R$ satisfies $w(T) \lsm \log R$, and fully contains every unit ball in $\supp w$ that it intersects.
\end{itemize}

We still want to construct a function $F$ as a sum of wave packets that is large on a big proportion of $\supp w$, and the procedure is the same as that in Section~\ref{subsec:ax_L_sharp} word by word, except that in Claim~\ref{first_claim} our $\mathcal{B}$ satisfies $\# \mathcal{B} \gtrsim (\log R)^{-2} R^{n-1}$, so we omit the details.

By direct computation in this slightly modified setting, we have
\begin{align*}
    \int_{\R^n} |F|^2w \gtrsim \ell \cdot \# \mc{B} \gtrsim (\log R)^{-2} R^{n - 1 + \frac{2}{n(n+1)}}
\end{align*}
and
\begin{align*}
    \int_{\R^n} |F|^2 \lsm \sum_{\tau \in \mc{D}} \int_{\R^n} |F_\tau|^2 &\lsm \sum_{\tau \in \mc{D}} \sum_{T \in \T_\tau} \int_T |c_T d(\tau)^{\frac{1}{2}}|^2 \\
    &= \sum_{\tau \in \mc{D}} \sum_{T \in \T_\tau} |T| \cdot |\tau| = |[0,1]| \cdot |B_R|
    \sim R^n.
\end{align*}
Here in the first step of the upper bound above we can use (DA2) by virtue of the fact that $(F_\tau)_{\tau \in \mc{D}_R}$ are concentrated on $B_R$. Thus we obtain
\begin{align*}
    (\log R)^{-3}R^{-1 + \frac{2}{n(n+1)}} \sup_{P \in \mathcal{P}}w(P) \norm{F}_{L^2(\R^n)}^2
    \lsm \int_{\R^n} |F|^2 w.
\end{align*}

Note that the main term on the left-hand side above is just $R^{-\frac{1}{r}}$, so we are done.

\subsection{Proof of (iii).}\label{subsec:ax_S_sharp}
Here comes a vital change if we try following the arguments in Sections~\ref{subsec:ax_L_sharp} and \ref{subsec:ax_P_sharp}: If we naturally impose $w(S) \lsm \log R$ when constructing $w$, then we cannot identify a useful scale (like $R^{\frac{1}{n} - \frac{2}{n^2(n+1)}}$ in Section~\ref{subsec:ax_L_sharp} or $R^{\frac{2}{n(n+1)}}$ in Section~\ref{subsec:ax_P_sharp}) other than the canonical scale $R^{\frac{1}{n}}$. This indicates a fundamental difference between the $n=2$ and $n \geq 3$ cases of Corollary~\ref{cor:fat_tubes_cor_for_tubeslabs_axiomatic} (iii). To be more precise, when $n \geq 3$, if we imitate the arguments in \cite{CIW24} to identify a useful scale $\ell$, then we need $|T| \sim R^{n-1}$ for any dual Vinogradov plank $T$ of dimension $\ell^{1} \times \ell^{2} \times \cdots \times \ell^{n}$, which implies that $\ell^{\frac{n(n+1)}{2}} \sim R^{n-1}$, i.e., $\ell \sim R^{\frac{2(n-1)}{n(n+1)}}$. However, we then have $\ell^n \sim R^{\frac{2(n-1)}{n+1}} \geq R$ whenever $n \geq 3$, which means that the largest $T$ we can fit into $B_R$ with $|T| \leq R^{n-1}$ is just of canonical dimension $R^{\frac{1}{n}} \times R^{\frac{2}{n}} \times \cdots \times R$, and so the useful scale is exactly the canonical scale $R^{\frac{1}{n}}$. By contrast, when $n=2$, by solving $\ell \times \ell^2 \sim R$ we get $\ell \sim R^{\frac{1}{3}}$; but now $\ell^2 \sim R^{\frac{2}{3}} \ll R$, which means that the largest rectangle $T$ we can fit into $B_R$ is of dimension $R^{\frac{1}{3}} \times R^{\frac{2}{3}}$, and so the useful scale is $R^{\frac{1}{3}}$ instead of the canonical scale $R^{\frac{1}{2}}$.

\begin{rem}
    Recall that for the sphere/paraboloid (co-dimensional $1$) case in \cite{CIW24}, the weight $w$ is expected to satisfy $w(P) \lsm \log R$ for any $1$-tube $P$. So to find the useful scale $\ell$ in this setting, let the volume of a canonical dual tube of dimension $\ell \times \cdots \times \ell \times \ell^2$ to be equal to $|P| \sim R$, by solving which we get $\ell \sim R^{\frac{1}{n+1}}$. Note that $\ell^2 \sim R^{\frac{2}{n+1}} \ll R$ whenever $n \geq 2$, so the largest tube $T$ we can fit into $B_R$ with $|T| \leq R$ is just $R^{\frac{1}{n+1}} \times R^{\frac{1}{n+1}} \times \cdots \times R^{\frac{2}{n+1}}$, which means the useful scale is $R^{\frac{1}{n+1}}$ instead of the canonical scale $R^{\frac{1}{2}}$. Since the moment curve and the sphere/paraboloid are the most ``well-curved'' submanifolds\footnote{For a precise definition of this intuitive concept, one may consult Theorem~1 of \cite{Gre19}.} in $\R^n$ of co-dimensions $1$ and $(n-1)$ respectively, this indicates that the wave packet approach to the Mizohata--Takeuchi inequality for higher co-dimensional submanifolds may be very different from that for hypersurfaces.
\end{rem}

For these reasons, we are not able to obtain a necessary condition that exactly matches the bound in Corollary~\ref{cor:fat_tubes_cor_for_tubeslabs_axiomatic} (iii).

Let $\Gamma$ be as earlier. As has been pointed out, the scale $\ell = R^{\frac{1}{n}}$ will play a key role in the upcoming argument; thus, denote by $\mc{D}$ the set of all $\tau \in \mc{D}_R$ with $d(\tau) = R^{-\frac{1}{n}}$. For each $\tau \in \mc{D}$, let $\T_\tau$ be a family of finitely overlapping planks in $\R^n$ that intersect and cover $B_R$, of dimension $R^{\frac{1}{n}} \times R^{\frac{2}{n}} \times \cdots \times R^{\frac{n}{n}}$ dual to $\Gamma(\tau)$ (essentially translates of $\Gamma(\tau)^*$). Let
\begin{align*}
    \T_R \coloneqq \bigsqcup_{\tau \in \mc{D}} \T_\tau.
\end{align*}

By taking $N \sim R$, $\mu \sim \log R$ in Theorem~\ref{weight_construction}, we know that any convex hull of volume $\lsm R^{n-1}$ must only contain $\lsm \log R$ points. So we can conclude that there exists a weight $w: \R^n \rightarrow [0, +\infty)$ such that the following hold:
\begin{itemize}
    \item $w$ is the characteristic function of a union of $\sim R$ unit balls in $B_R$;
    \item Any $S \in \mathcal{S}$ satisfies $w(S) \lsm \log R$;
    \item Any $T \in \T_R$ satisfies $w(T) \lsm \log R$, and fully contains every unit ball in $\supp w$ that it intersects.
\end{itemize}

Once again, we want to construct a function $F$ as a sum of wave packets that is large on a big proportion of $\supp w$. Since the construction is somewhat different now, we include full details. First, we need the following analogue of Claim~\ref{first_claim}:
\begin{clm}\label{third_claim}
    There exist
    \begin{itemize}
        \item a set $\mc{B} = \{B_1, \dots, B_m\}$ of $\gtrsim (\log R)^{-2} R$ disjoint unit balls in $\supp w$, and 
        \item sets $\T_j \subset \T_R$ with $\# \T_j \gtrsim \# \mc{D}$ for every $j = 1, \dots, m$, 
    \end{itemize}
    such that the following hold:
    \begin{itemize}
        \item The tubes in $\T_j$ contain $B_j$, for all $j = 1, \dots, m$.
        \item For each $j = 2, \dots, m$, the tubes in $\T_j$ do not intersect any of the balls $B_1, \dots, B_{j-1}$.
    \end{itemize}
\end{clm}
\begin{proof}
    The construction of $\{B_j\}$ and $\{\mathbb{T}_j\}$ as well as the proof are all exactly the same as those for Claim 7.1 in \cite{CIW24}.
\end{proof}

We can now construct a sequence $(F_\tau)_{\tau \in \mc{D}_R}$ of functions $F_\tau: \R^n \rightarrow \C$, and here comes the main difference with Section~\ref{subsec:ax_L_sharp} and \ref{subsec:ax_P_sharp}, as well as with \cite{CIW24}. In both these previous cases, $\mathcal{B}$ is abundant, so {\em a priori}, we need to take all $\mathbb{T}_\tau$ into consideration, and we will lose nothing in doing this. However, now the $\mathcal{B}$ constructed in Claim~\ref{third_claim} is sparse, so it would be most efficient to only take those tubes involved, i.e., $\widetilde{\mathbb{T}} \coloneqq \bigsqcup_{j=1}^m \mathbb{T}_j$, into account. This motivates the following modified construction:
\begin{itemize}
    \item For each $\tau \in \mc{D}_R$ with $d(\tau) = R^{-\frac{1}{n}}$ (or, more precisely, for each $\tau \in \mc{D}$), define
    \begin{align*}
        F_\tau \coloneqq \sum_{T \in \T_\tau \cap \widetilde{\mathbb{T}}} c_T e^{-2\pi i \langle \cdot, \xi_\tau \rangle} d(\tau)^{\frac{1}{2}} \phi_T,
    \end{align*}
    where $\phi_T$ is a bump function on $T$ and $\xi_\tau$ is the centre of $\Gamma(\tau)$. The coefficients $c_T \in \C$ will be defined below.
    \item For $\gamma \in \mc{D}_R$ with $R^{-\frac{1}{n}} < d(\gamma) \leq 1$, define
    \begin{align*}
        F_\gamma \coloneqq \sum_{\tau \in \mc{D}, \tau \subset \gamma} F_\tau.
    \end{align*}
\end{itemize}
Let $F \coloneqq F_{[0,1]} = \sum_{\tau \in \mc{D}} F_\tau$. Through exactly the same procedure as that in \cite[Section 7]{CIW24}, by using Claim~\ref{third_claim}, the coefficients $c_T \in \C$ can be chosen to have modulus $1$, and such that 
\begin{align*}
    |F| \gtrsim R^{\frac{1}{2n}} \text{ on } \bigcup_{B \in \mc{B}} B.
\end{align*}
It is easy to verify that $(F_\tau)_{\tau \in \mc{D}_R}$ satisfies the decoupling axioms (DA0) -- (DA2).

By direct computation, we have
\begin{align*}
    \int_{\R^n} |F|^2 w \gtrsim R^{\frac{1}{n}} \# \mc{B} \gtrsim (\log R)^{-2} R^{1 + \frac{1}{n}}
\end{align*}
and
\begin{align*}
    \int_{\R^n} |F|^2 & \lsm \sum_{\tau \in \mc{D}} \int_{\R^n} |F_\tau|^2\\
    &\lsm \sum_{\tau \in \mc{D}} \sum_{T \in \T_\tau \cap \widetilde{\mathbb{T}}} \int_T |c_T d(\tau)^{\frac{1}{2}}|^2 \\
    &= \sum_{\tau \in \mc{D}} \sum_{T \in \T_\tau \cap \widetilde{\mathbb{T}}} |T| \cdot |\tau| \\
    &= |[0,1]| \cdot \sum_{T \in \T_\tau \cap \widetilde{\mathbb{T}}} |T|\\
    &= R \cdot R^{\frac{n+1}{2}} = R^{\frac{n+3}{2}}.
\end{align*}
Here in the first step of the upper bound above we can use (DA2) by virtue of the fact that $(F_\tau)_{\tau \in \mc{D}_R}$ are concentrated on $B_R$. Thus we obtain
\begin{align*}
    (\log R)^{-3} R^{-\frac{n+1}{2} + \frac{1}{n}} \sup_{S\in\mathcal
    S}w(S) \norm{F}_{L^2(\R^n)}^2
    \lsm \int_{\R^n} |F|^2 w.
\end{align*}

Note that the main term on the left-hand side above is just $R^{-\frac{n+1}{2r}}$, so we are done.

\subsection{Final remarks.}%
Here we provide further comments on the proof of Theorem~\ref{thm:axiomatic_sharpness} (iii).

In our construction of $F$ in Section~\ref{subsec:ax_S_sharp}, there is no hierarchy of $F_\tau$'s with $R^{-\frac{1}{n}} \lsm d(\tau) < \ell^{-1}$ (as $\ell = R^{\frac{1}{n}}$), which means $F_\gamma = \sum_{\tau \subset \gamma} F_\tau$ is true down to the smallest scale $R^{-\frac{1}{n}}$, and so $(F_\tau)_{\tau \in \mc{D}_R}$ can be realised by genuine wave packets. This shows the difference between ``sharp modulo axiomatic decoupling'' and ``sharp modulo wave packet decomposition'': For the former we are allowed to have an imaginary hierarchy which does not necessarily satisfy $F_\gamma = \sum_{\tau \subset \gamma} F_\tau$, such as those $F_\tau$'s constructed in Section~\ref{subsec:ax_L_sharp} and \ref{subsec:ax_P_sharp} when $R^{-\frac{1}{n}} \lsm d(\tau) < \ell^{-1}$. (Recall Remark~\ref{rem:diff_various_sharpness} for the difference between ``sharp modulo wave packet decomposition'' and ``sharp for the original Fourier extension operator''.)

Indeed, as we saw above in Remark~\ref{rem:sharpness_implicns}, the necessity of at least $R^{-\frac{n+1}{2r}}$-loss can also be directly seen from both the single bush example in Remark~\ref{rem:sharp_fat_tubes_cor} and the single wave packet example in Remark~\ref{rem:sharpness_implicns}. So the main point of the much more complicated arguments in Section~\ref{subsec:ax_S_sharp} is that we can construct $F$ as an essentially disjoint union of multiple bushes of wave packets.

Also note that $R^{-\frac{1}{n}} \lsm d(\tau) < \ell^{-1}$ can arise only if the useful scale $\ell$ is not canonical, so Guth's argument \cite
{Gu22} will be advantageous only when a ``noncanonical'' scale naturally appears -- otherwise the power of it should be no stronger than genuine wave packet constructions.

	\section{Appendix: Some detailed arguments}\label{sec:appendix}
	In this section we give some of the details suppressed in earlier sections. %

	\subsection{Refined decoupling for well-curved curves -- some details}\label{sec:appendix_subsec1}
	
	Fix $n \geq 2$ and a constant $c_0 > 1$. We consider a family of well-curved curves $\mathcal{C}(c_0)$ in $\R^n$, namely all $C^{n+1}$ curves $\Gamma \colon [-1,1] \to \R^n$, with $$\|\Gamma^{(j)}\|_{L^{\infty}} \leq c_0 n \quad \text{for all $j=1,\dots,n+1$, and} \quad c_0^{-1} \leq |\Gamma' \wedge \dots \wedge \Gamma^{(n)}|(t)   \leq c_0 \quad \forall t.$$ 
	An example is the moment curve $\Gamma_0(t) = (\frac{t}{1!},\frac{t^2}{2!},\dots,\frac{t^n}{n!})$, defined for $-1 \leq t \leq 1$. It is in $\mathcal{C}(c_0)$ for any $c_0 > 1$. Any curve $\Gamma$ in $\mathcal{C}(c_0)$ is locally, up to a small error, an affine image of the moment curve $\Gamma_0$: 
	\[
	\Gamma(a+t\delta) = \Gamma(a)+ \sum_{j=1}^n \frac{\delta^j t^j}{j!} \Gamma^{(j)}(a) + \text{error}(t),
	\]
	\[
	\text{error}(t):=\int_0^t \frac{\delta^{n+1} \Gamma^{(n+1)}(a+s\delta)}{n!} (t-s)^{n} {\rm d} s.
	\]
	In fact, let $L^{\Gamma}_a$ and $D_{\delta}$ be linear maps from $\R^n$ to $\R^n$ given by
	\[
	L^{\Gamma}_a \xi = \sum_{j=1}^n \xi_{j} \Gamma^{(j)}(a) \quad \text{and} \quad
	D_{\delta}\xi = (\delta \xi_1, \dots, \delta^n \xi_n), \quad \xi = (\xi_1,\dots,\xi_n)
	\]
	and let $A^{\Gamma}_{a,\delta}$ be the affine map on $\R^n$ defined by $$A^{\Gamma}_{a,\delta}\xi = \Gamma(a) + L^{\Gamma}_a D_{\delta} \xi.$$
	Then for any $t$ with $|t| \lesssim 1$, %
	$$\Gamma(a+t\delta)-\Gamma(a)
	= L^{\Gamma}_a D_{\delta} \Gamma_0(t) + \text{error}(t), \quad \|\text{error}(t)\|_{L^{\infty}({\rm d}t)} \leq \frac{c_0 n \delta^{n+1}}{(n+1)!} $$
	so
	$$
	(A^{\Gamma}_{a,\delta})^{-1} \Gamma(a+t\delta)
	= \Gamma_0(t) +  D_{\delta}^{-1} (L^{\Gamma}_a)^{-1} \text{error}(t).$$
	We have $|\det L^{\Gamma}_a| \geq c_0^{-1}$, and all $(n-1) \times (n-1)$ cofactors of the matrix $L^{\Gamma}_a$ has absolute value bounded by $c_0 n$. Thus all entries of the matrix $ (L^{\Gamma}_a)^{-1}$ are bounded by $c_0 (c_0 n)^{n-1}$, and hence all entries of the vector $(L^{\Gamma}_a)^{-1} \text{error}(t)$ are bounded by $\frac{c_0^{n+1} n^n \sqrt{n}}{(n+1)!} \delta^{n+1}$.
	As a result,
	\[
	\|D_{\delta}^{-1} (L^{\Gamma}_a)^{-1} \text{error}(t)\|_{L^{\infty}({\rm d}t)} \leq \frac{(c_0 n)^{n+1}}{(n+1)!} \delta.
	\]
	Similarly, for $j = 1, \dots, n+1$ we claim
	\[
	\Big| \frac{{\rm d}^j}{{\rm d}t^j} \text{error}(t) \Big| \leq \frac{c_0 n \delta^{n+1}}{(n-j+1)!};
	\]
	in fact for $j=1,\dots, n$ we have
	\[
	\frac{{\rm d}^j}{{\rm d}t^j} \text{error}(t) = \int_0^t \frac{\delta^{n+1} \Gamma^{(n+1)}(a+s\delta)}{(n-j)!} (t-s)^{n-j} {\rm d}s
	\]
	and for $j=n+1$ we have
	\[
	\frac{{\rm d}^j}{{\rm d}t^j} \text{error}(t) = \delta^{n+1} \Gamma^{(n+1)}(a+t\delta).
	\]
	In either case the asserted bound for $\frac{{\rm d}^j}{{\rm d}t^j} \text{error}(t) $ holds. Hence we have
	\[
	\|\frac{{\rm d}^j}{{\rm d}t^j} D_{\delta}^{-1} (L^{\Gamma}_a)^{-1} \text{error}(t)\|_{L^{\infty}({\rm d}t)} \leq \frac{(c_0 n)^{n+1}}{(n-j+1)!} \delta.
	\]
	Thus the above shows that if $$\tilde{\Gamma} (t):=(A^{\Gamma}_{a,\delta})^{-1} \Gamma(a+t\delta) ,$$ then there exists a finite constant $a(n,c_0)$ such that
	\[
	\|\tilde{\Gamma}^{(j)}\|_{L^{\infty}} \leq (1+a(n,c_0)\delta) n, \quad \text{for all $j=1,\dots,n+1$}.
	\]
	Furthermore, by making $a(n,c_0)$ larger but still depending only on $n$ and $c_0$, since $\Gamma_0' \wedge \dots \wedge \Gamma_0^{(n)}(t) \equiv 1$ for all $t$, we conclude
	\[
	|\tilde{\Gamma}' \wedge \dots \wedge \tilde{\Gamma}^{(n)}(t) - 1| \leq a(n,c_0) \delta.
	\]
	This shows that if 
	\begin{equation} \label{eq:delta0_def}
		\delta_0(n,c_0) := (c_0-1)/a(n,c_0)
	\end{equation}
	and $\delta \leq \delta_0$, one has $\tilde{\Gamma} \in \mathcal{C}(c_0)$ as well.

	For $R \geq 1$ and $\Gamma \in \mathcal{C}(c_0)$, let $r = \lceil \log_2 R^{1/n} \rceil$ and let $\Theta_{\Gamma}(R^{-1/n})$ be the family of finitely overlapping parallelepipeds given by\footnote{This definition of $\Theta_{\Gamma}(R^{-1/n})$ is a little different from the one introduced at the beginning of Section \ref{sec:refineddecoupling_curves}, in that we consider parallelepipeds of the biggest dyadic scale not exceeding $R^{-1/n}$. It will be more convenient to use this new definition from now on for the purposes of rescaling.}
	\[
	\Theta_{\Gamma}(R^{-1/n}) = \{ A^{\Gamma}_{j 2^{-r}, 2^{-r}} [-1,1]^n \colon j \in [-2^r,2^r] \cap \mathbb{Z} \}.
	\]
	It will be convenient to write\footnote{Again this definition of $A_{\theta}$ is a little different from the one introduced at the beginning of Section \ref{sec:refineddecoupling_curves}, in that the dilation factor is always a dyadic number, and we will use this new convention from now on.} $A^{\Gamma}_{j 2^{-r}, 2^{-r}}$ as $A_{\theta}$ if $\theta = A^{\Gamma}_{j 2^{-r}, 2^{-r}} [-1,1]^n$.  For later use, one can check that if $\Gamma \in \mathcal{C}(c_0)$, $R^{-1/n} \leq 2^{-s} \leq \delta_0(n,c_0)$, $\beta \in \Theta_{\Gamma}(2^{-s})$ and $\theta \in \Theta_{\Gamma}(R^{-1/n})$ so that the centre of $\theta$ is contained in $\beta$, then $A_{\beta}^{-1} \theta \in \Theta_{\tilde{\Gamma}}(2^s R^{-1/n})$, where $\tilde{\Gamma} \in \mathcal{C}(c_0)$ is given by $\tilde{\Gamma}(t) = A_{\beta}^{-1} \Gamma(a + t 2^{-s})$ and $a$ is the ``centre'' of $\beta$.
	
	The celebrated Bourgain--Demeter--Guth decoupling theorem \cite{BDG16} for the moment curve says the following:
	
	\begin{theorem}[Bourgain--Demeter--Guth] \label{thm:BDG}
		Fix $n \geq 1$ and $c_0 > 1$. 
		Suppose $\Gamma \in \mathcal{C}(c_0)$ is a well-curved $C^{n+1}$ curve in $\R^n$ and $R \geq 1$. Assume, for each $\theta \in \Theta_{\Gamma}(R^{-1/n})$, $f_{\theta}$ is a Schwartz function on $\R^n$ whose Fourier support is contained in $\theta$. If
		\[
		f = \sum_{\theta \in \Theta_{\Gamma}(R^{-1/n})} f_{\theta},
		\]
		then for every $2 \leq p \leq n(n+1)$ and every $\epsilon > 0$,
		\begin{equation} \label{eq:BDGeq}
			\|f\|_{L^p(\R^n)} \lesssim_{\epsilon} R^{\epsilon} \Big(\sum_{\theta \in \Theta_{\Gamma}(R^{-1/n})} \|f_{\theta}\|_{L^p(\R^n)}^2 \Big)^{1/2}.
		\end{equation}    
		The implicit constant above can depend on $n$, $c_0$ and $\epsilon$, but is independent of $\Gamma$.
	\end{theorem}
	
	The fact that the implicit constant in \eqref{eq:BDGeq} can be chosen uniformly for all $\Gamma$ in class $\mathcal{C}(c_0)$ is known, but not explicitly stated in the literature. One can either use a Pramanik--Seeger iteration to deduce this uniformity from the corresponding theorem for the moment curve $\Gamma_0$, as in Lemma 3.6 of \cite{GLYZ21}; or one can adapt the proof given in \cite{GLYZ21} to give a direct proof of the asserted uniformity. To do the latter, the key observations are that if $\Gamma \in \mathcal{C}(c_0)$, then the rescaling of any of its segment on an interval of length $\leq \delta_0(c_0,n)$ is still a curve $\tilde{\Gamma} \in \mathcal{C}(c_0)$ as we have shown above; additionally, for any $j=1,\dots,n-1$ and any $\delta_1 > 0$, the projection of $\tilde{\Gamma}|_{[\delta_1,1]}$ or $\tilde{\Gamma}|_{[-1,-\delta_1]}$ onto the first $j$ coordinate axes is, after rescaling to $[-1,1]$, uniformly a curve in $\R^j$ in the corresponding class $\mathcal{C}(c_1)$ for some $c_1$ that depends only on $\delta_1$, $c_0$, $n$ and $j$. The proof of the latter fact proceeds via Lemma 3.5 of \cite{GLYZ21}. We omit the details.
	
	The $f_{\theta}$'s are morally speaking locally constant on translates of the parallelepiped $\theta^* := T_{\theta}^{-1}[-1,1]^n$, where $T_{\theta}$ is the transpose of the linear part of $A_{\theta}$.\footnote{Since $A_{\theta}$ in this section involves dilation factors that are powers of 2, the same is true for $T_{\theta}$.} This can be made precise using weights, which we introduce as follows.

	To each $\theta \in \Theta_{\Gamma}(R^{-1/n})$ and each $N > n$, we associate a weight 
	\[
	w_{\theta,N}(x) := |\det(T_{\theta})| (1+|T_{\theta} x|)^{-N}.
	\] 
	It can be shown that
	\begin{enumerate}[(a)]
		\item pointwisely 
		\begin{equation} \label{eq:C_nN1}
			|f_{\theta}| \lesssim_{n,N} |f_{\theta}|*w_{\theta,N},
		\end{equation}
		and hence 
		taking $N=n+1$ for instance,
		\begin{equation} \label{eq:C_nN2} %
			\|f_{\theta}\|_{L^{\infty}(\R^n)} \lesssim_n R^{-\frac{n+1}{2p}} \|f_{\theta}\|_{L^p(\R^n)}
		\end{equation}
		for all $p \geq 2$; 
		\item $w_{\theta,N}(x+y) \leq (1+|T_{\theta} y|)^{N} w_{\theta,{N}}(x)$ for all $x, y \in \R^n$, which in particular means 
		\begin{equation} \label{eq:w_compare}
			w_{\theta,N}(x+y) \lesssim_{n,N} w_{\theta,N}(x) \quad \text{if $y \in \theta^*$}.
		\end{equation}
	\end{enumerate}
	In fact, by rescaling we just need to check these when $R  = 1$ and $\theta = [-1,1]^n$, the details of which we omit. Now we can explain the meaning of the locally constant property. Note that by combining (\ref{eq:C_nN1}) with (\ref{eq:w_compare}), we have: For any $u\in\R^n$,
	\begin{align*}\label{eq:locally_constant_add}    \|f_{\theta}\|_{L^{\infty}(u+\theta^*)} &\lesssim_{n,N}
		\sup_{x\in u+\theta^*} \int_{\R^n} |f_\theta(y)|w_{\theta,N}(x-y){\rm d}y\\
		& \leq \int_{\R^n} |f_\theta(y)|\sup_{x\in u+\theta^*} w_{\theta,N}(x-y){\rm d}y\\
		& \lesssim_{n,N} 
		\int_{\R^n} |f_\theta(y)|w_{\theta,N}(u-y){\rm d}y = |f_\theta|*w_{\theta,N} (u).
	\end{align*}
	Since $w_{\theta,N}(u-\cdot)$ is an $L^1$-normalised bump adapted to $u+\theta^*$, this morally tells us that $f_\theta$ satisfies the ``reverse Hölder inequality'' and so behaves like a constant function on each $u+\theta^*$. In other words, for all our purposes here, it does not harm anything if we simply regard $f_\theta$ as genuinely constant on each $u+\theta^*$, although this is not literally true.
	
	Along a similar line, one can obtain a local version of the decoupling inequality \eqref{eq:BDGeq}: if $c_0$, $\Gamma$, $p$, $R$ and $f=\sum_{\theta \in \Theta_{\Gamma}(R^{-1/n})} f_{\theta}$ are as in the statement of Theorem~\ref{thm:BDG}, then for any $N > 0$, there exists a constant $A_N$, depending only on $N$ and the dimension $n$, such that for any cube $Q \subset \R^n$ of side length $R$,
	\begin{equation} \label{eq:BDGlocal}
		\|f\|_{L^p(Q)} \leq D_{\epsilon} R^{\epsilon} \Big( \sum_{\theta \in \Theta_{\Gamma}(R^{-1/n})} \|f_{\theta}\|_{L^p(\R^n, A_N (1+ |Q|^{-1/n} |x-c_Q|)^{-N} {\rm d}x)}^2 \Big)^{1/2}.
	\end{equation}
	Here $D_{\epsilon}$ is a constant depending only on $\epsilon, c_0, n$ and $p$, and $c_Q$ is the centre of $Q$. This function $A_N$ of $N$ will be fixed throughout this section.
	
	A refined decoupling theorem can be stated as follows: 
	\begin{theorem} \label{thm:refined_dec_theta}
		Fix $c_0 > 1$ and $2 \leq p \leq n(n+1)$. For every $\epsilon > 0$ and $N > n$, there exists a constant $C_{\epsilon,N}$, depending only on $\epsilon, N, c_0$ and $n$, such that the following is true:
		Let $\Gamma \in \mathcal{C}(c_0)$, $R \geq 1$ and let $f=\sum_{\theta \in \Theta_{\Gamma}(R^{-1/n})} f_{\theta}$ be as in the statement of Theorem~\ref{thm:BDG}.
		For any $M \geq 1$, if $Y_M$ is a union of finitely many (say $X$) disjoint cubes $Q$ of side length $R^{1/n}$, such that for each $Q \subset Y_M$, there are at most $M$ parallelepipeds $\theta \in \Theta_{\Gamma}(R^{-1/n})$ for which 
		\begin{equation} \label{eq:significant_thetas}
			\sup_{x \in Q} |f_{\theta}|*w_{\theta,N}(x) > R^{-1} X^{-1}  \|f_{\theta}\|_{L^{\infty}(\R^n)},
		\end{equation}
		then 
		\begin{equation} \label{eq:refined_dec_theta}
			\|f\|_{L^p(Y_M)} \leq C_{\epsilon,N} R^{\epsilon} M^{\frac{1}{2}-\frac{1}{p}} \Big( \sum_{\theta \in \Theta_{\Gamma}(R^{-1/n})} \|f_{\theta}\|_{L^p(\R^n)}^p \Big)^{1/p}.
		\end{equation}
	\end{theorem}

	Intuitively, given $Q \subset Y_M$, the $\theta$'s that satisfy \eqref{eq:significant_thetas} are those that are significant to $Q$. The main point of this definition is that %
	if $\theta$ does not satisfy \eqref{eq:significant_thetas}, then by  \eqref{eq:C_nN1} and \eqref{eq:C_nN2},
	\[
	\begin{split}
		\|f_{\theta}\|_{L^{\infty}(Q)} 
		& \lesssim_{n,N} \sup_{x \in Q} |f_{\theta}|*w_{\theta,N}(x) \\
		&\leq  R^{-1} X^{-1} \|f_{\theta}\|_{L^{\infty}(\R^n)} \\
		&\lesssim_n  R^{-1 -\frac{n+1}{2p}} X^{-1} \|f_{\theta}\|_{L^p(\R^n)}
	\end{split}
	\]
	so there exists a constant $C_{n,N}$ depending only on $n$ and $N$ such that 
	\begin{equation} \label{eq:insignificant}
		\|f_{\theta}\|_{L^{\infty}(Q)} \leq C_{n,N} R^{-1} X^{-1} \|f_{\theta}\|_{L^p(\R^n)}.
	\end{equation}
	(Here, $C_{n,N}$ is the product of the implicit constants in \eqref{eq:C_nN1} and \eqref{eq:C_nN2}, and we just bound $R^{-\frac{n+1}{2p}}$ by $1$.) Note that $|Q| = R$ and $\#(Q \subset Y_M) = X$, so \eqref{eq:insignificant} allows us to have a good control of
    \[
    \sum_{Q \subset Y_M} \int_Q \Big| \sum_{\theta \text{ insignficant to $Q$}} f_{\theta} \Big|.
    \]

	The proof of Theorem~\ref{thm:refined_dec_theta}, which is based upon Theorem~\ref{thm:BDG}, will be given in Section~\ref{sec:pf_thm_6.2}. 
	
	Theorem~\ref{thm:refined_dec_theta} as formulated here is a little unconventional insofar as it does not mention wave packets. It can be seen as an intermediate version between Bourgain--Demeter--Guth (Theorem~\ref{thm:BDG}), which only sees the decomposition of $f$ into sum of $f_{\theta}$'s, and the eventual refined decoupling theorem (Theorem~\ref{thm:refined_dec_wp} below) which is formulated in a more standard way in terms of wave packets. The key in Theorem~\ref{thm:refined_dec_theta} is in clarifying what it means for a given $f_{\theta}$ to contribute significantly to a given small cube $Q$ (which is what condition \eqref{eq:significant_thetas} addresses). Unfortunately that condition is not very clean; it involves a parameter $N$, and in applications we do need to choose $N$ to be large depending on the small power $\epsilon$ that we want to put on $R$. %
	
	We now reformulate Theorem~\ref{thm:refined_dec_theta} using wave packets. 
        It will be convenient to consider a situation where for each $\theta \in \Theta_{\Gamma}(R^{-1/n})$, $f_{\theta}$ is Fourier supported in a slightly smaller subset $A_{\theta}[-1/4,1/4]^n$ of $\theta = A_{\theta}[-1,1]^n$. 
		For each $\theta \in \Theta_{\Gamma}(R^{-1/n})$ we expand $f_{\theta} = \sum_{T \in \T_{\theta}} f_T$ where $\T_{\theta}$ and $f_T$ are defined as in \eqref{eq:T_theta_def} and \eqref{eq:fTdef0} respectively.
	Note that for every $T \in \T_{\theta}$,
	\[
	\widehat{f_T}(\xi) = a_{m,\theta} \widehat{\Phi}(T_{\theta}^{-t}(\xi-\Gamma(\theta))) e^{-2\pi i m \cdot T_{\theta}^{-t} (\xi-\Gamma(\theta))}
	\]
	so $\widehat{f_T}$ is supported in $\Gamma(\theta) + T_{\theta}^t [-1/2,1/2]^n = A_{\theta}[-1/2,1/2]^n \subset \theta$.
	
	Let $\T = \bigsqcup_{\theta \in \Theta_{\Gamma}(R^{-1/n})} \T_{\theta}$. We then have
	\begin{equation} \label{eq:wavepacket_decomp}
		f = \sum_{T \in \T} f_T
	\end{equation}
	
	For future reference, we note that
	\[
	\|f_T\|_{L^p(\R^n)} = |a_{m,\theta}| |\det T_{\theta}|^{1-\frac{1}{p}} \|\Phi\|_{L^p(\R^n)} \sim %
	|a_{m,\theta}| |T|^{-(1-\frac{1}{p})}
	\]
	so
	\[
	\|f_T\|_{L^p(\R^n)} \sim |T|^{\frac{1}{p}-\frac{1}{2}}\|f_T\|_{L^2(\R^n)}.
	\]
	Also
	\begin{align*}
		\sum_{T \in \mathbb{T}_{\theta}} \|f_T\|_{L^2(\R^n)}^2 
		&= |\det T_{\theta}| \|\Phi\|_{L^2(\R^n)}^2 \sum_{m \in \mathbb{Z}^n} |a_{m,\theta}|^2 \\
		&= |\det T_{\theta}| \|\Phi\|_{L^2(\R^n)}^2  \|\widehat{f_{\theta}} \circ A_{\theta}\|_{L^2{(}[-1/2,1/2]^n{)}}^2 \\
		&= \|\Phi\|_{L^2(\R^n)}^2 \|\widehat{f_{\theta}}\|_{L^2(\R^n)}^2 \\
		&= \|\Phi\|_{L^2(\R^n)}^2 \|f_{\theta}\|_{L^2(\R^n)}^2
	\end{align*}
	so that 
	\begin{equation} \label{eq:fTortho}
		\sum_{T \in \T} \|f_T\|_{L^2(\R^n)}^2 \sim \|f\|_{L^2(\R^n)}^2.
	\end{equation}
	
	The wave packet reformulation of Theorem~\ref{thm:refined_dec_theta} is now:
	
	\begin{theorem} \label{thm:refined_dec_wp}
		Fix $c_0 > 1$ and $2 \leq p \leq n(n+1)$. Let $R \geq 1$ and $f = \sum_{T \in \T} f_T$ be the wave packet decomposition of a function whose Fourier support is in the curvature sleeve of some curve $\Gamma \in \mathcal{C}(c_0)$ as in the discussion culminating in \eqref{eq:wavepacket_decomp}. Let $B$ be a ball of radius $R$, and $\T(B)$ be a collection of $T \in \T$ for which $T \cap B \ne \emptyset$. Let $\epsilon' > 0$, and $\T_{\epsilon'}(B)$ be a subcollection of $\T(B)$ such that $R^{\epsilon'}T_1 \cap R^{\epsilon'}T_2 = \emptyset$ for any two parallel $T_1, T_2 \in \T_{\epsilon'}(B)$. Let $M \geq 1$ and let $Y_M$ be a union of disjoint cubes $Q$ of side length $R^{1/n}$ in $\R^n$, such that every $Q \subset Y_M$ is contained in $R^{\epsilon'}T$ for at most $M$ tubes $T \in \T_{\epsilon'}(B)$. Then for any $\epsilon > 0$,
		\begin{equation} \label{eq:refined_dec_wp}
			\Big\| \sum_{T\in \T_{\epsilon'}(B)} f_T \Big\|_{L^p(Y_M)} \lesssim_{\epsilon,\epsilon'} R^{\epsilon} M^{\frac{1}{2}-\frac{1}{p}} \Big( \sum_{T \in \T_{\epsilon'}(B)} \|f_T\|_{L^p(\R^n)}^p \Big)^{1/p}.
		\end{equation}
		The implicit constant will depend only on $\epsilon, \epsilon', c_0$ and $n$, but not on the curve $\Gamma$.
		
	\end{theorem}

	Note that the parameters $\epsilon$ and $\epsilon'$ are independent of each other, and this result with $\epsilon' = \epsilon$ implies Theorem~\ref{thm:refined_dec_wp_prior}, because the planks $\T(B)$ in Theorem~\ref{thm:refined_dec_wp_prior} can be written as the disjoint union of $R^{O(\epsilon)}$ families $\T_{\epsilon}(B)$, and $\| \sum_{T \in \T(B)} f_T \|_{L^p(Y_M)}$ is bounded by $R^{O(\epsilon)}$ times the left hand side of \eqref{eq:refined_dec_wp} for one of the $\T_{\epsilon}(B)$'s. The case $n=2$ of Theorem~\ref{thm:refined_dec_wp} is a result of \cite{GIOW20}, while the case $n=3$ appeared in \cite{DGW20}.
	
	\begin{proof}[Proof of Theorem~\ref{thm:refined_dec_wp}]
		Let $N$ be sufficiently large, depending only on $\epsilon'$ and $n$, so that 
		\begin{equation} \label{eq:Nchoice}
			\frac{\epsilon' (N-2n)}{2} \geq n + n \epsilon' + 1.
		\end{equation}
		Then there exists a constant $C'_{n,N}$, depending only on $n$ and $N$, such that
		\[
		\sup_{x \in [-1,1]^n} \sum_{m \in \mathbb{Z}^n \colon m \notin R^{a} [-1,1]^n} |\Phi(x-m)| \leq C'_{n,N} R^{-a (N-n)}
		\]
		for all $a > 0$. 
		Here $\Phi$ is the Schwartz function we used to define our wave packets. 
		Thus for every $a > 0$, $\theta \in \Theta_{\Gamma}(R^{-1/n})$ and every $x \in \R^n$,
		\begin{equation} \label{eq:tail1}
			\sum_{T \in \T_{\theta} \cap \T_{\epsilon'}(B) \colon x \notin R^{a}T} |f_T(x)| \leq C'_{n,N} R^{-a (N-n)} \sup_{T \in \T_{\theta} \cap \T_{\epsilon'}(B) \colon x \notin R^{a}T} \|f_T\|_{L^{\infty}(\R^n)}.
		\end{equation}
		We pick $R_{\epsilon'}$ depending only on $\epsilon'$, $n$ and $c_0$ so that 
		\begin{enumerate}[(i)]
			\item $C'_{n,N} R^{-\epsilon' (N-n)} \leq 1/2$ whenever $R \geq R_{\epsilon'}$;
			\item $R_{\epsilon'}^{-1/n} \leq \delta_0(n,c_0)$ with $\delta_0(n,c_0)$ defined in \eqref{eq:delta0_def}; and
			\item $R_{\epsilon'} \geq K_{N,n} $ where $K_{N,n}$ is the constant in the inequality \eqref{eq:insignificant_wp} below.
		\end{enumerate}
		
		For $R \leq R_{\epsilon'}$, the desired inequality  \eqref{eq:refined_dec_wp} holds true by simply taking the implicit constant to be sufficiently large depending on $\epsilon'$, $n$ and $c_0$. Thus we assume $R > R_{\epsilon'}$ from now on, and \eqref{eq:tail1} with $a = \epsilon'$ now says for every $\theta \in \Theta_{\Gamma}(R^{-1/n})$ and every $x \in \R^n$,
		\begin{equation} \label{eq:tail2}
			\sum_{T \in \T_{\theta} \cap \T_{\epsilon'}(B) \colon x \notin R^{\epsilon'}T} |f_T(x)| \leq \frac{1}{2} \sup_{T \in \T_{\theta} \cap \T_{\epsilon'}(B) \colon x \notin R^{\epsilon'}T} \|f_T\|_{L^{\infty}(\R^n)}.
		\end{equation}
		
		Write
		\[
		\sum_{T \in \T_{\epsilon'}(B)} f_T = \sum_{\theta} F_{\theta} \quad \text{where} \quad \text{$F_{\theta} := \sum_{T \in \T_{\theta} \cap \T_{\epsilon'}(B)} f_T$.}
		\]
		Note that $\supp \, \widehat{F_{\theta}} \subset \theta$ for every $\theta$. We have
		\[
		\|F_{\theta}\|_{L^{\infty}(\R^n)} \sim \sup_{T \in \T_{\theta} \cap \T_{\epsilon'}(B)} \|f_T\|_{L^{\infty}(\R^n)}.
		\]
		Indeed, by the separation of those $T \in  \T_{\theta} \cap \T_{\epsilon'}(B)$, we know any $x \in \R^n$ belongs to at most one $R^{\epsilon'}T$ where $T \in \T_{\theta} \cap \T_{\epsilon'}(B)$ (say $T_0$).
		Thus using \eqref{eq:tail2}, we obtain
		\[
		|F_{\theta}(x)| \leq |f_{T_0}(x)| + \sum_{T \in \T_{\theta} \cap \T_{\epsilon'}(B) \colon x \notin R^{\epsilon'}T} |f_T(x)| \leq \frac{3}{2}  \sup_{T \in \T_{\theta} \cap \T_{\epsilon'}(B)} \|f_T\|_{L^{\infty}(\R^n)}.
		\]
		Conversely, if $\sup_{T \in \T_{\theta} \cap \T_{\epsilon'}(B)} \|f_T\|_{L^{\infty}(\R^n)} = |f_{T_0}(x_0)|$ for some $T_0 \in \T_{\theta} \cap \T_{\epsilon'}(B)$ and $x_0 \in \R^n$, then $x_0 \in R^{\epsilon'} T_0$ and 
		\[
		|f_{T_0}(x_0)| \leq |F_{\theta}(x_0)| + \sum_{T \in \T_{\theta} \cap \T_{\epsilon'}(B) \colon  x_0 \notin R^{\epsilon'}T} |f_T(x_0)| \leq |F_{\theta}(x_0)| + \frac{1}{2} |f_{T_0}(x_0)|
		\]
		(we used \eqref{eq:tail2} in the last inequality). Thus 
		\[
		|f_{T_0}(x_0)| \leq 2 \|F_{\theta}\|_{L^{\infty}(\R^n)}.
		\]
		
		We now show that if $\theta \in \Theta_{\Gamma}(R^{-1/n})$ and $x \notin R^{\epsilon'}T$ for all $T \in \T_{\theta} \cap \T_{\epsilon'}(B)$, then
		\begin{equation} \label{eq:goodpoint}
			|F_{\theta}|*w_{\theta,N}(x) \lesssim_{N,n} R^{-n-n\epsilon'-1} \|F_{\theta}\|_{L^{\infty}(\R^n)}.
		\end{equation}
		In fact, 
		\begin{align*}
			&|F_{\theta}|*w_{\theta,N}(x) \\
			\leq &\,\|w_{\theta,N}\|_{L^{\infty}(R^{\epsilon'/2} \theta^*)} \int_{R^{\epsilon'/2} \theta^*} |F_{\theta}(x-y)| {\rm d} y + \|F_{\theta}\|_{L^{\infty}(\R^n)}  \int_{\R^n \setminus R^{\epsilon'/2} \theta^*} w_{\theta,N}(y) {\rm d} y.
		\end{align*}
		We bound the second term using
		\[
		\int_{\R^n \setminus R^{\epsilon'/2} \theta^*} w_{\theta,N}(y) {\rm d} y = \int_{\R^n \setminus R^{\epsilon'/2} [-1,1]^n} (1+|z|)^{-N} {\rm d}z \lesssim_n R^{-\frac{\epsilon'}{2}(N-n)}.
		\]
		For the first term, we have
		\[
		\|w_{\theta,N}\|_{L^{\infty}(R^{\epsilon'/2} \theta^*)} \int_{R^{\epsilon'/2} \theta^*} |F_{\theta}(x-y)| {\rm d} y 
		\lesssim_n \frac{1}{|\theta^*|} \int_{R^{\epsilon'/2} \theta^*} |F_{\theta}(x-y)| {\rm d} y 
		\leq R^{n\epsilon'/2} \sup_{y \in R^{\epsilon'/2} \theta^*} |F_{\theta}(x-y)|.
		\]
		Since we assumed $x \notin R^{\epsilon'}T$ for all $T \in \T_{\theta} \cap \T_{\epsilon'}(B)$, we have $x-y \notin R^{\epsilon'/2} T$ for any $y \in R^{\epsilon'/2} \theta^*$. 
		Thus by \eqref{eq:tail1} with $a=\epsilon'/2$,
		\[
		\sup_{y \in R^{\epsilon'/2} \theta^*} |F_{\theta}(x-y)| \leq C'_{n,N} R^{-\frac{\epsilon'(N-n)}{2}} \sup_{T \in \T_{\theta} \cap \T_{\epsilon'}(B)} \|f_T\|_{L^{\infty}(\R^n)} \leq C'_{n,N} R^{-\frac{\epsilon'(N-n)}{2}} \|F_{\theta}\|_{L^{\infty}(\R^n)}.
		\]
		Altogether, we have
		\[
		|F_{\theta}|*w_{\theta,N}(x) \lesssim_{N,n} R^{-\frac{\epsilon'(N-2n)}{2}} \|F_{\theta}\|_{L^{\infty}(\R^n)}
		\]
		and \eqref{eq:goodpoint} follows from our choice of $N$ in \eqref{eq:Nchoice}.
		
		Now  if $Q \subset Y_M$, $\theta \in \Theta_{\Gamma}(R^{-1/n})$ and $Q$ is not contained in $R^{\epsilon'}T$ for any $T \in \T_{\theta} \cap \T_{\epsilon'}(B)$, then picking a point $x_0 \in Q$ so that $x_0 \notin R^{\epsilon'}T$ for any $T \in \T_{\theta} \cap \T_{\epsilon'}(B)$, we have, by applying \eqref{eq:w_compare} followed by \eqref{eq:goodpoint} with $x_0$ in place of $x$, that
		\[
		\sup_{x \in Q} |F_{\theta}|*w_{\theta,N}(x) \lesssim_{N,n} |F_{\theta}|*w_{\theta,N}(x_0) \lesssim_{N,n} R^{-n-n\epsilon'-1} \|F_{\theta}\|_{L^{\infty}(\R^n)}.
		\]
		Since all cubes $Q \subset Y_M$ intersect $R^{\epsilon'}T$ for at least one $T \in \T_{\epsilon'}(B)$, 
		all $Q$'s are contained in a ball of radius $R^{1+\epsilon'}$. Thus the number of cubes $Q$ in $Y_M$ is $X \lesssim_n (R^{1+\epsilon'}/R^{1/n})^n = R^{n-1+n\epsilon'}$. 
The above display then implies that there exists a constant $K_{N,n}$ such that %
		\begin{equation}  \label{eq:insignificant_wp}
			\sup_{x \in Q} |F_{\theta}|*w_{\theta,N}(x) \leq K_{N,n} R^{-2} X^{-1} \|F_{\theta}\|_{L^{\infty}(\R^n)}
		\end{equation}
		unless $Q$ is contained in $R^{\epsilon'}T$ for some $T \in \T_{\theta} \cap \T_{\epsilon'}(B)$. By our choice of $R_{\epsilon'}$ in (iii) just after \eqref{eq:tail1}, \eqref{eq:insignificant_wp} further implies
		\begin{equation} \label{eq:insignificant_wp2}
			\sup_{x \in Q} |F_{\theta}|*w_{\theta,N}(x) \leq R^{-1} X^{-1} \|F_{\theta}\|_{L^{\infty}(\R^n)}.
		\end{equation} 
		Since each $Q \subset Y_M$ is contained in $R^{\epsilon'}T$ for at most $M$ tubes $T \in \T_{\epsilon'}(B)$, and since $\T_{\epsilon'}(B)$ is a disjoint union $\bigsqcup_{\theta} \T_{\theta} \cap \T_{\epsilon'}(B)$, this says there are at most $M$ parallepipeds $\theta$ for which \eqref{eq:insignificant_wp2} is violated. We now apply Theorem~\ref{thm:refined_dec_theta}, with $N$ chosen in \eqref{eq:Nchoice} %
		depending only on $\epsilon'$ and $n$. Thus \eqref{eq:refined_dec_wp} follows.
	\end{proof}

	\subsection{Weighted formulations of refined decoupling -- Proof of Theorem~\ref{main_corrected_reprise}}\label{sec:proof_2.2}
	
	In this section we give the full details of the proof of Theorem~\ref{main_corrected_reprise}. We first give a version of that result with an additional local constancy hypothesis on the weight.  
	\begin{theorem}\label{thm:weightedwp}
		Let $n \geq 2$, $\Gamma$ be a well-curved $C^{n+1}$ curve in $\R^n$, and fix any parameter $N$.  For any $\epsilon > 0$, there exists a constant $C_{\epsilon,N}$ depending only on $\epsilon, N, \Gamma$ and $n$, such that the following is true: If $f$ has Fourier support in the curvature sleeve $\Theta_{\Gamma}(R^{-1/n})$ of $\Gamma$ and $f=\sum_{T \in \T} f_T$ is as in \eqref{eq:wavepacket_decomp}, and $w$ is any non-negative weight so that $w(x+y) \leq C_N(1+|y|)^{N} w(x)$ for all $x,y \in \R^n$, $p = n(n+1)$ and $\frac{1}{r}=1-\frac{2}{p}$, then
		\begin{equation} \label{eq:weightedwp1}
			\int_{\R^n} |f(x)|^2 w(x) {\rm d} x \leq C_{\epsilon,N} R^{\epsilon} \Big( \sum_{T \in \T} \|f_T\|_{L^2(\R^n)}^2 \frac{w^r(R^{\epsilon} T)}{|T|} \Big)^{1/r} \|f\|_{L^2(\R^n)}^{4/p}. %
		\end{equation}
		In particular, 
		\begin{equation} \label{eq:weightedwp2}
			\int_{\R^n} |f(x)|^2 w(x) {\rm d} x \leq C_{\epsilon,N} R^{\epsilon} R^{-\frac{n+1}{2r}} \Big( \sup_{T \in \T} w^r(R^{\epsilon}T) \Big)^{1/r} \|f\|_{L^2(\R^n)}^2.
		\end{equation}
	\end{theorem}
	
	The technical condition $w(x+y) \lesssim_N (1+|y|)^{N} w(x)$ for all $x,y \in \R^n$ which is placed on the weight $w$ in this result is very mild, and is satisfied for example if $w$ is the convolution of a non-negative measure with $(1+|x|)^{-N}$. 
	
	\begin{proof} Suppose $N > n$, $\epsilon > 0$. Set $\epsilon' = \frac{\epsilon}{8n}$. 
		We first show that for any $N'$, we have
		\begin{equation} \label{eq:outsideT}
			\int_{\R^n \setminus R^{\epsilon'/2} T} |f_T(x)|^2 w(x) {\rm d}x \lesssim_{\epsilon', N, N'} R^{-N'}\|f_T\|_{L^{\infty}(\R^n)}^2 w(T).
		\end{equation}
		In fact, if $T \in \T_{\theta}$, then $\R^n \setminus R^{\epsilon'/2} T$ is contained in a disjoint union of $T + T_{\theta}^{-1}m$ for $m \in \mathbb{Z}^n \setminus R^{\epsilon'/2}[-1/2,1/2]^n$. %
		
		Given $N$, we choose $N''$ large enough so that $(2 N''-N-n) \epsilon'/ 2 > N+N'$. Then $\sup_{x \in T + T_{\theta}^{-1}m} |f_T(x)| \lesssim_{N''} \|f_T\|_{L^{\infty}(\R^n)} (1+|m|)^{-N''}$, hence
		\begin{align*}
			\int_{\R^n \setminus R^{\epsilon'/2} T} |f_T|^2 w 
			&\lesssim_{N''} \sum_{m \in \mathbb{Z}^n \setminus R^{\epsilon'/2}[-1,1]^n}  \|f_T\|_{L^{\infty}(\R^n)}^2 (1+|m|)^{-2 N''} w( T + T_{\theta}^{-1}m ) \\
			&\lesssim_N  \|f_T\|_{L^{\infty}(\R^n)}^2 \sum_{m \in \mathbb{Z}^n \setminus R^{\epsilon'/2}[-1/2,1/2]^n} (1+|m|)^{-2 N''} (1+|T_{\theta}^{-1} m|)^N w( T ) \\
			&\lesssim  \|f_T\|_{L^{\infty}(\R^n)}^2 w(T) \sum_{m \in \mathbb{Z}^n \setminus R^{\epsilon'/2}[-1/2,1/2]^n} (1+|m|)^{-2 N''} (1+R |m|)^N
		\end{align*}
		and the last sum is at most $R^{-(2N''-N-n)\epsilon'/2} R^N \leq R^{-N'}$ by our choice of $N''$. Since $N''$ depends only on $\epsilon', N$ and $N'$, we have the desired inequality \eqref{eq:outsideT}.
		
		Now we cover $\R^n$ by finitely overlapping balls $\{B\}$ of radius $R$ and write $\T$ as a disjoint union $\bigsqcup_B \T(B)$ so that each $\T(B)$ is a collection of $T \in \T$ that intersects $B$. We shall show that %
		\begin{equation} \label{eq:splitball}
			\int_{\R^n} |f|^2 w \lesssim R^{n\epsilon'} \sum_B \int_{\R^n} \Big|\sum_{T \in \T(B)} f_T \Big|^2 w + \text{acceptable error}.
		\end{equation}
		This would be true without any error if each $\sum_{T \in \T(B)} f_T$ is supported in $B$. 
		The error is actually bounded above by
		\[
		\sum_{\text{dist}(B,B') \geq R^{1+\epsilon'}} \int_{\R^n} \sum_{T \in \T(B)} \sum_{T' \in \T(B')} |f_T f_{T'}| w \leq \sum_{\text{dist}(T,T') \geq R^{1+\epsilon'}} \int_{\R^n} |f_T f_{T'}| w.
		\]
		For fixed $T, T' \in \T$ with $\text{dist}(T,T') \geq R^{1+\epsilon'}$, we estimate $\int_{\R^n} |f_T f_{T'}| w$ by integrating over the set of $x$ for which $\text{dist}(x,T) \geq \text{dist}(T,T') / 2$, and the set of $x$ for which $\text{dist}(x,T') \geq \text{dist}(T,T') / 2$. (Any $x \in \R^n$ satisfies at least one of these conditions by the triangle inequality.) In the first case, for any $N'$, we have 
		\[
		\begin{split}
			\int_{x \colon \text{dist}(x,T) \geq \text{dist}(T,T')/2} |f_T f_{T'}| w 
			&\leq \sup_{x \colon \text{dist}(x,T) \geq \text{dist}(T,T')/2} |f_T(x)| \int_{\R^n} |f_{T'}| w \\
			&\lesssim \Big( \frac{R}{\text{dist}(T,T')} \Big)^{N'} \|f_T\|_{L^{\infty}(\R^n)} \|f_{T'}\|_{L^{\infty}(\R^n)} w(R^{\epsilon'} T')
		\end{split}
		\]
		The last inequality can be established with a proof similar to that of \eqref{eq:outsideT}. 
		But 
		\[
		w(R^{\epsilon'} T') \lesssim R^{n \epsilon' + (1+\epsilon') N} w(T')
		\]
		since $R^{\epsilon'} T'$ can be covered by $R^{n \epsilon'}$ copies of translates of $T'$, and $w$ satisfies the bound $w(x+y) \leq C_N (1+|y|)^N w(x)$ for every $x, y$. Hence
		\[
		\int_{x \colon \text{dist}(x,T) \geq \text{dist}(T,T')/2} |f_T f_{T'}| w \lesssim \Big( \frac{R}{\text{dist}(T,T')} \Big)^{N'} \|f_T\|_{L^{\infty}(\R^n)} \|f_{T'}\|_{L^{\infty}(\R^n)} w(T')  R^{n \epsilon' + (1+\epsilon') N} .
		\]
		
		Now we let $B, B'$ be the balls of radius $R$ containing $T$ and $T'$, so that $\text{dist}(B,B') \leq \text{dist}(T,T')$, and estimate 
		\[
		w(T') \leq w(B') \lesssim \text{dist}(B,B')^N w(B).
		\]
		(These $B, B'$ are just used in the next few lines and not the finitely overlapping balls we had before.)
		Furthermore, 
		\begin{align*}
			\frac{w(B)}{w(T)} 
			&\lesssim \sum_{1 \leq k_1 \leq R^{1-1/n}} \sum_{1 \leq k_2 \leq R^{1-2/n}} \dots  \sum_{1 \leq k_{n-1} \leq R^{1-(n-1)/n}} (1+k_1 R^{1/n})^N (1+k_2 R^{2/n})^N \dots (1+k_{n-1} R^{(n-1)/n})^N \\
			& \sim (R^{1-1/n} R^N) (R^{1-2/n} R^N) \dots (R^{1-(n-1)/n} R^N) \leq R^{(N+1)(n-1)}.
		\end{align*}
		Hence
		\begin{align*}
			&\int_{x \colon \text{dist}(x,T) \geq \text{dist}(T,T')/2} |f_T f_{T'}| w \\
			&\lesssim \Big( \frac{R}{\text{dist}(T,T')} \Big)^{N'} \|f_T\|_{L^{\infty}(\R^n)} \|f_{T'}\|_{L^{\infty}(\R^n)} w(T)^{1/2} w(T')^{1/2} \text{dist}(T,T')^{N/2} R^{(N+1)(n-1)/2}  R^{n \epsilon' + (1+\epsilon') N}.
		\end{align*}
		Similarly we have the estimate for $\int_{x \colon \text{dist}(x,T') \geq \text{dist}(T,T')/2} |f_T f_{T'}| w$.
		We can estimate $\|f_T\|_{L^{\infty}}$ by $\|f_T\|_{L^2} |T|^{-1/2}$, and similarly for $\|f_{T'}\|_{L^{\infty}}$. Hence for any $N'' > 0$,
		\[
		\int_{\R^n} |f_T f_{T'}| w \lesssim_{N''} R^{\frac{(N+1)(n-1)+N}{2}+ n \epsilon' + (1+\epsilon') N} \|f_T\|_{L^2} \Big(\frac{w(T)}{|T|}\Big)^{1/2} (\text{dist}(T,T')/R)^{-N''} \|f_{T'}\|_{L^2} \Big(\frac{w(T')}{|T'|}\Big)^{1/2}.
		\]
		Now note that for $N''> n$, we have
		\[
		\begin{split}
			\sup_{T \in \T} \sum_{\text{dist}(T,T') \geq R^{1+\epsilon'}} (\text{dist}(T,T')/R)^{-N''} 
			&\lesssim 
			\sup_{T \in \T} \sum_{k \colon 2^k \geq R^{\epsilon'}} 2^{-k N''}  \#\{T' \colon \text{dist}(T,T') \sim 2^k R\} \\
			&\lesssim \sum_{k \colon 2^k \geq R^{\epsilon'}} 2^{-k N''} \frac{(2^k R)^n}{|T'|} R^{1/n} \\
			&\sim R^{\frac{1}{n}+\frac{n-1}{2}} \sum_{k \colon 2^k \geq R^{\epsilon'}} 2^{-k (N''-n)} \\ 
			&\lsm R^{\frac{1}{n}+\frac{n-1}{2}-(N''-n) \epsilon'}
		\end{split}
		\]
		and similarly with the roles of $T$ and $T'$ reversed. Schur's test (or Cauchy--Schwarz) then says
		\[
		\begin{split}
			&	R^{\frac{(N+1)(n-1)+N}{2}+n \epsilon' + (1+\epsilon') N} \sum_{\text{dist}(T,T') \geq R^{1+\epsilon'}} \|f_T\|_{L^2} \Big(\frac{w(T)}{|T|}\Big)^{1/2} (\text{dist}(T,T')/R)^{-N''} \|f_{T'}\|_{L^2} \Big(\frac{w(T')}{|T'|}\Big)^{1/2} \\
			\lesssim &  R^{\frac{(N+1)(n-1)+N}{2}+n \epsilon' + (1+\epsilon') N} (R^{\frac{1}{n}+\frac{n-1}{2}-(N''-n) \epsilon'}) \sum_{T} \|f_T\|_{L^2}^2 \frac{w(T)}{|T|}.
		\end{split}
		\]
		
		The power of $R$ can be made non-positive as long as $N''$ is large enough with respect to our fixed $N$, $n$ and $\epsilon'$. Finally,
		\[
		\sum_{T} \|f_T\|_{L^2}^2 \frac{w(T)}{|T|}
		\leq \sum_T \Big( \|f_T\|_{L^2}^2  \frac{w^r(T)}{|T|} \Big)^{1/r} \|f_T\|_{L^2}^{2 \cdot 2/p} \\
		\leq \Big( \sum_T  \|f_T\|_{L^2}^2  \frac{w^r(T)}{|T|} \Big)^{1/r} \|f\|_{L^2}^{4/p}.
		\]
		This shows that the error term in \eqref{eq:splitball} is negligible. Hence we shall now focus only on the first term on the right hand side of \eqref{eq:splitball}.
		
		To treat the first term on the right hand side of \eqref{eq:splitball}, we further decompose each $\T(B)$ as the disjoint union of $R^{n\epsilon'}$ `sparse' subcollections, so that any two distinct parallel $T_1, T_2$ in the same subcollection satisfy $R^{\epsilon'}T_1 \cap R^{\epsilon'}T_2 = \emptyset$. Then for each $B$, there exists a `sparse' subcollection $\overline{\T}_{\epsilon'}(B)$ of $\T(B)$, such that 
		\begin{equation} \label{eq:sparsepigeon}
			\int_{\R^n} \Big|\sum_{T \in \T(B)} f_T \Big|^2 w \leq R^{2n\epsilon'} \int_{\R^n} \Big|\sum_{T \in \overline{\T}_{\epsilon'}(B)} f_T \Big|^2 w.
		\end{equation}
		We split $\overline{\T}_{\epsilon'}(B)$ into two subcollections, say $\T_{\epsilon',1}(B)$ are those $T \in \overline{\T}_{\epsilon'}(B)$ with 
		\begin{equation} \label{eq:threshold_fT}
			\|f_T\|_{L^2(\R^n)}^2 \leq  R^{-\frac{n-1}{2}(\frac{p}{2}+1)}  \sum_{T' \in \overline{\T}_{\epsilon'}(B)} \|f_{T'}\|_{L^2(\R^n)}^2,
		\end{equation} and $\T_{\epsilon',2}(B) := \overline{\T}_{\epsilon'}(B) \setminus \T_{\epsilon',1}(B)$. We use a trivial bound for the contributions from $\T_{\epsilon',1}(B)$: we use %
		\[
		\begin{split}
			\int_{\R^n} \Big| \sum_{T \in \T_{\epsilon',1}(B)} f_T \Big|^2 w 
			&\lesssim (\#\T_{\epsilon',1}(B))  \sum_{T \in \T_{\epsilon',1}(B)} \int_{\R^n} |f_T|^2 w\\
			&\lsm R^{\frac{n-1}{2}} \sum_{T \in \T_{\epsilon',1}(B)} \|f_T\|_{L^{\infty}(\R^n)}^2 w(R^{\epsilon'}T)
		\end{split}
		\]
		where we used \eqref{eq:outsideT} to integrate outside $R^{\epsilon'}T$. Next,
		\begin{equation} \label{eq:HolderT}
			\begin{split}
				&\quad \sum_{T \in \T_{\epsilon',1}(B)} \|f_T\|_{L^{\infty}(\R^n)}^2 w(R^{\epsilon'}T)\\
				&\sim \sum_{T \in \T_{\epsilon',1}(B)} \|f_T\|_{L^2(\R^n)}^2 \frac{w(R^{\epsilon'}T)}{|T|} \\
				&\leq \Big( \sum_{T \in \T_{\epsilon',1}(B)} \|f_T\|_{L^2(\R^n)}^2 \Big(\frac{w(R^{\epsilon'}T)}{|T|} \Big)^r \Big)^{1/r} \Big( \sum_{T \in \T_{\epsilon',1}(B)} \|f_T\|_{L^2(\R^n)}^2 \Big)^{2/p}.
			\end{split}
		\end{equation}
		Furthermore, H\"{o}lder's inequality implies $(w(R^{\epsilon'}T))^r \leq w^r(R^{\epsilon'}T) (R^{n\epsilon'}|T|)^{r-1}$, so 
		\begin{equation} \label{eq:Holderw}
			\Big( \frac{w(R^{\epsilon'}T)}{|T|} \Big)^r \leq R^{n \epsilon'(r-1)} \frac{w^r(R^{\epsilon'}T)}{|T|}.
		\end{equation}
		Using \eqref{eq:threshold_fT}, we have
		\[
		\begin{split}
			\Big( \sum_{T \in \T_{\epsilon',1}(B)} \|f_T\|_{L^2(\R^n)}^2 \Big)^{2/p}  
			&\leq \Big( R^{\frac{n-1}{2}} R^{-\frac{n-1}{2}(\frac{p}{2}+1)} \sum_{T' \in \overline{\T}_{\epsilon'}(B)} \|f_{T'}\|_{L^2(\R^n)}^2 \Big)^{2/p}\\
			&= R^{-\frac{n-1}{2}} \Big(\sum_{T \in \overline{\T}_{\epsilon'}(B)} \|f_{T}\|_{L^2(\R^n)}^2 \Big)^{2/p}.
		\end{split}
		\]
		Altogether, remembering $n\epsilon' = \epsilon/8$, the four displayed inequalities above imply that
		\begin{equation} \label{eq:estepsilon'1}
			\int_{\R^n} \Big| \sum_{T \in \T_{\epsilon',1}(B)} f_T \Big|^2 w 
			\leq R^{\frac{\epsilon}{8}(1-\frac{1}{r})} \Big( \sum_{T \in \T(B)} \|f_T\|_{L^2(\R^n)}^2 \frac{w^r(R^{\epsilon'}T)}{|T|}  \Big)^{1/r} \Big( \sum_{T \in \T(B)} \|f_T\|_{L^2(\R^n)}^2 \Big)^{2/p}.
		\end{equation}
		
		Next, if $T \in \T_{\epsilon',2}(B)$, then \[ \frac{\|f_T\|_{L^2(\R^n)}}{\Big(\sum_{T' \in \overline{\T}_{\epsilon'}(B)} \|f_{T'}\|_{L^2(\R^n)}^2 \Big)^{1/2}} \in [R^{-\frac{n-1}{4}(\frac{p}{2}+1)}, 1].\] 
		We can thus pigeonhole and obtain a subcollection $\T_{\epsilon'}(B)$ of $\T_{\epsilon',2}(B)$ so that 
		\begin{equation} \label{eq:estepsilon'2}
			\int_{\R^n} \Big| \sum_{T \in \T_{\epsilon',2}(B)} f_T \Big|^2 w   \lesssim (\log R)^2 \int_{\R^n} \Big| \sum_{T \in \T_{\epsilon'}(B)} f_T \Big|^2 w, 
		\end{equation}
		and all $\|f_T\|_{L^2(\R^n)}$'s are comparable when $T \in \T_{\epsilon'}(B)$. We integrate over $\R^n$ by partitioning it into lattice squares $\{Q\}$ of side length $R^{1/n}$. Let $Y_0$ be those $Q$'s that are not contained in $R^{\epsilon'}T$ for any $T \in \T_{\epsilon}(B)$. The union of these cubes is contained in $\R^n \setminus R^{\epsilon'/2}T$ for any $T \in \T_{\epsilon'}(B)$. 
		So by \eqref{eq:outsideT}, we have
		\[
		\int_{Y_0} \Big| \sum_{T \in \T_{\epsilon'}(B)} f_T \Big|^2 w \leq R^{\frac{n-1}{2}} \sum_{T \in \T_{\epsilon'}(B)} \int_{\R^n \setminus R^{\epsilon'/2}T} |f_T|^2 w \lesssim_{\epsilon',N} \sum_{T \in \T_{\epsilon'}(B)} \|f_T\|_{L^{\infty}(\R^n)}^2 w(T).
		\]
		Using H\"{o}lder's inequality twice, as in \eqref{eq:HolderT} and \eqref{eq:Holderw}, we obtain
		\begin{equation} \label{eq:estY0}
			\int_{Y_0} \Big| \sum_{T \in \T_{\epsilon'}(B)} f_T \Big|^2 w\lesssim  \Big( \sum_{T \in \T(B)} \|f_T\|_{L^2(\R^n)}^2 \frac{w^r(R^{\epsilon'}T)}{|T|}  \Big)^{1/r} \Big( \sum_{T \in \T(B)} \|f_T\|_{L^2(\R^n)}^2 \Big)^{2/p}.
		\end{equation}
		Finally, $\R^n \setminus Y_0$ can be partitioned into $\bigsqcup_M Y_M$ where $Y_M$ is the union of lattice cubes $Q$ of side length $R^{1/n}$ that are contained in $(M/2,M]$ many $T \in \T_{\epsilon}(B)$, for dyadic $M \in 2^{\mathbb{N}}$. Since only $O(\log R)$ of the $Y_M$ are non-empty, we can pigeonhole again and obtain
		\begin{equation} \label{eq:estY0complement}
			\int_{\R^n \setminus Y_0} \Big| \sum_{T \in \T_{\epsilon'}(B)} f_T \Big|^2 w \lesssim (\log R) \int_{Y_M} \Big| \sum_{T \in \T_{\epsilon'}(B)} f_T \Big|^2 w
		\end{equation}
		for some $M \in 2^{\mathbb{N}}$. Now we bound
		\[
		\int_{Y_M} \Big| \sum_{T \in \T_{\epsilon'}(B)} f_T \Big|^2 w 
		\leq  \Big(w^r(Y_M) \Big)^{1/r} \Big( \int_{\R^n} \Big| \sum_{T \in \T_{\epsilon'}(B)} f_T \Big|^p \Big)^{2/p}.
		\]
		Applying Theorem~\ref{thm:refined_dec_wp} and the bound 
		\begin{align*}
			M w^{r}(Y_M) 
			&=\sum_{Q \subset Y_M} M w^{r}(Q)\\
			&\leq 2 \sum_{Q  \subset Y_M}\sum_{T \in \T_{\epsilon'}(B) \, :\, Q \subset R^{\epsilon'} T} w^{r}(Q)\\
			&\leq 2 \sum_{T \in \T_{\epsilon'}(B)} \sum_{Q \subset R^{\epsilon'} T} w^{r}(Q) \\
			&\leq 2 \sum_{T \in \T_{\epsilon'}(B)}  w^{r}(R^{\epsilon'} T),
		\end{align*}
		we obtain
		\begin{equation} \label{eq:estYM1}
			\int_{Y_M} \Big| \sum_{T \in \T_{\epsilon'}(B)} f_T \Big|^2 w 
			\lesssim_{\epsilon} R^{\epsilon/2}  \Big( \sum_{T \in \T_{\epsilon'}(B)}  w^{r}(R^{\epsilon'} T) \Big)^{1/r} \Big( \sum_{T \in \T_{\epsilon'}(B)} \|f_T\|_{L^p(\R^n)}^p \Big)^{2/p}.
		\end{equation}
		Since $$\|f_T\|_{L^p(\R^n)}^p \sim \Big( \frac{\|f_T\|_{L^2(\R^n)}^2}{|T|} \Big)^{\frac{p}{2}-1} \|f_T\|_{L^2(\R^n)}^2$$ and all $\|f_T\|_{L^2(\R^n)}$'s are comparable if $T \in \T_{\epsilon'}(B)$, we obtain 
		\begin{equation} \label{eq:estYM2}
			\begin{split}
				&\Big( \sum_{T \in \T_{\epsilon'}(B)}  w^{r}(R^{\epsilon'} T) \Big)^{1/r} \Big( \sum_{T \in \T_{\epsilon'}(B)} \|f_T\|_{L^p(\R^n)}^p \Big)^{2/p} \\
				&\sim \Big( \sum_{T \in \T_{\epsilon'}(B)} \|f_T\|_{L^2(\R^n)}^2  \frac{w^{r}(R^{\epsilon'} T)}{|T|} \Big)^{1/r} \Big( \sum_{T \in \T_{\epsilon'}(B)} \|f_T\|_{L^{2}(\R^n)}^2 \Big)^{2/p}.
			\end{split}
		\end{equation}
		From \eqref{eq:estY0}, \eqref{eq:estY0complement}, \eqref{eq:estYM1} and \eqref{eq:estYM2}, we have
		\[
		\int_{\R^n} \Big| \sum_{T \in \T_{\epsilon'}(B)} f_T \Big|^2 w 
		\lesssim_{\epsilon} (\log R) R^{\epsilon/2} \Big( \sum_{T \in \T(B)} \|f_T\|_{L^2(\R^n)}^2 \frac{w^r(R^{\epsilon'}T)}{|T|}  \Big)^{1/r} \Big( \sum_{T \in \T(B)} \|f_T\|_{L^2(\R^n)}^2 \Big)^{2/p},
		\]
		so from \eqref{eq:estepsilon'2},
		\[
		\int_{\R^n} \Big| \sum_{T \in \T_{\epsilon',2}(B)} f_T \Big|^2 w \lesssim_{\epsilon} (\log R)^3 R^{\epsilon/2} \Big( \sum_{T \in \T(B)} \|f_T\|_{L^2(\R^n)}^2 \frac{w^r(R^{\epsilon'}T)}{|T|}  \Big)^{1/r} \Big( \sum_{T \in \T(B)} \|f_T\|_{L^2(\R^n)}^2 \Big)^{2/p}.
		\]
		Combining this with \eqref{eq:sparsepigeon} and \eqref{eq:estepsilon'1}, we obtain
		\[
		\int_{\R^n} \Big| \sum_{T \in \T(B)} f_T \Big|^2 w \lesssim_{\epsilon} (\log R)^3 R^{\epsilon/2 + 2 n \epsilon'} \Big( \sum_{T \in \T(B)} \|f_T\|_{L^2(\R^n)}^2 \frac{w^r(R^{\epsilon'}T)}{|T|}  \Big)^{1/r} \Big( \sum_{T \in \T(B)} \|f_T\|_{L^2(\R^n)}^2 \Big)^{2/p}.
		\]
		We now sum the above over $B$ using H\"{o}lder's inequality. Then from \eqref{eq:splitball}, using the bound $(\log R)^3 R^{\epsilon/2 + 3 n \epsilon'}  \lesssim R^{\epsilon}$, we have
		\[
		\int_{\R^n} |f|^2 w \lesssim_{\epsilon} R^{\epsilon} \Big( \sum_{T \in \T} \|f_T\|_{L^2(\R^n)}^2 \frac{w^r(R^{\epsilon'}T)}{|T|}  \Big)^{1/r} \Big( \sum_{T \in \T} \|f_T\|_{L^2(\R^n)}^2 \Big)^{2/p}.
		\]
		We may of course enlarge $\epsilon'$ to $\epsilon$ on the right hand side of the inequality, and use \eqref{eq:fTortho}. This proves \eqref{eq:weightedwp1}, and \eqref{eq:weightedwp2} follows since $|T| \sim R^{\frac{n+1}{2}}$ and we can invoke \eqref{eq:fTortho} again.
	\end{proof}
	
	We remark that since we have, (by local constancy or the uncertainty principle)
	\[
	\|f_T\|_{L^p(\R^n)}^2 \sim |T|^{-1/r} \|f_T\|_{L^2(\R^n)}^2,
	\]
	if we also have that all $\|f_T\|_{L^2(\R^n)}$ are comparable for $T \in 
	\T_{\epsilon'}(B)$, then the right-hand side in this result is essentially 
	\begin{align*}
		R^{\epsilon} M^{\frac{1}{2}-\frac{1}{p}} \Big( \sum_{T \in \T_{\epsilon'}(B)} \|f_T\|_{L^p(\R^n)}^p \Big)^{1/p} & = R^{\epsilon} M^{\frac{1}{2}-\frac{1}{p}} |T|^{-\frac{1}{2r}}\Big( \sum_{T \in \T_{\epsilon'}(B)} \|f_T\|_{L^2(\R^n)}^p \Big)^{1/p}\\
		&= R^{\epsilon} \Big(\frac{M}{\# \T_{\epsilon'}(B)}\Big)^{\frac{1}{2}-\frac{1}{p}} |T|^{-\frac{1}{2r}}\Big( \sum_{T \in \T_{\epsilon'}(B)} \|f_T\|_{L^2(\R^n)}^2 \Big)^{1/2},
	\end{align*}
	where $|T| \sim R^{\frac{n+1}{2}}$ is the common value of the volume of the wave packet tubes in $\T_{\epsilon'}(B)$.

	Turning now to Theorem~\ref{main_corrected_reprise}, we need to remove the {\em a priori} assumption on the weight in Theorem~\ref{thm:weightedwp}. We first observe that due to the compact support of $\widehat{f}$, the expression $\int_{\R^n} |f|^2 w$ is left essentially unchanged by replacing $w$ by $\Phi \ast w$ for a suitable normalised bump function $\Phi$, as in Remark~\ref{rem:pla}. Thus we may assume that $w$ is constant at unit scale. For such a $w$ we may then simply dominate it by its convolution with $(1+|x|)^{-N}$ for very large $N$, and apply Theorem~\ref{thm:weightedwp} as stated with weight $w \ast (1+|x|)^{-N}$. This will obviously introduce error terms on the right hand side involving sums of terms involving large dilates of tubes $T$. However, these terms will have coefficients which are rapidly decaying in the dilation parameter, as well as in $R$, and these considerations, the details of which now follow, lead to Theorem~\ref{main_corrected_reprise} as stated.
	
	{\em Proof of Theorem~\ref{main_corrected_reprise}.}
	Since $\Gamma$ has compact support, we may and shall assume (as in Remark~\ref{rem:pla}) that $w$ is constant on unit scale. Let $H(x) = (1+|x|)^{-N}$ for $N \gg n$ so that $w(x) \lesssim_N  \, w \ast H(x)$, and apply Theorem~\ref{thm:weightedwp} with the weight  $w \ast H$ in place of $w$. We need to understand the term $(w \ast H)^r(R^\epsilon T)$ which now arises on the right hand side. Now
	$$H(x) \lesssim \sum_{k \geq 0} 2^{k(n-N)} \frac{1_{\{|x| \leq 2^k\}}}{2^{kn}},$$
	and by H\"older's inequality
	$$ w \ast \frac{1_{\{|\cdot| \leq 2^k\}}}{2^{kn}}(x) \lesssim \left(w^r \ast \frac{1_{\{|\cdot| \leq 2^k\}}}{2^{kn}} \right)^{1/r}(x), $$
	so that 
	$$ \left(\int_{R^\epsilon T} \left(w \ast \frac{1_{\{|\cdot| \leq 2^k\}}}{2^{kn}}\right)^r\right)^{1/r} \lesssim \left(\int_{R^\epsilon T} w^r \ast \frac{1_{\{|\cdot| \leq 2^k\}}}{2^{kn}}\right)^{1/r} \lsm \left(\int_{N_{2^k} (R^\epsilon T)} w^r\right)^{1/r}$$
	where $N_{2^k} (R^\epsilon T)$ denotes the $2^k$-neighbourhood of $R^\epsilon T$. Hence, by Minkowski's inequality,
	$$\left((w \ast H)^r(R^\epsilon T) \right)^{1/r}\lesssim_N \sum_{k \geq 0} 2^{k(n-N)} \left(\int_{N_{2^k} (R^\epsilon T)} w^r\right)^{1/r}.$$
	Next we want to understand ${N_{2^k} (R^\epsilon T)}$ for $k \geq 0$. For $2^k \leq R^{1/n + \epsilon}$ we have $N_{2^k} (R^\epsilon T) \subset 2R^\epsilon T$, while for $2^k \geq R^{1/n + \epsilon}$ we have ${N_{2^{k}} (R^\epsilon T)} \subset (2^{k+1}/R^{1/n}) T$.
	Therefore
	$$((w \ast H)^r(R^\epsilon T))^{1/r}
	\lesssim_{N} 
	\sum_{2^k \leq R^{1/n + \epsilon}} 2^{k(n-N)} \left(\int_{2R^\epsilon T} w^r\right)^{1/r} + \sum_{2^k \geq R^{1/n + \epsilon}} 2^{k(n-N)} \left(\int_{(2^{k+1} R^{-1/n})T} w^r\right)^{1/r}$$
	$$ \lesssim_{N} 
	\left(\int_{2R^\epsilon T} w^r\right)^{1/r} + \sum_{m: 2^m \geq R^{\epsilon}} \left(2^m R^{1/n}\right)^{(n-N)} \left(\int_{2^m T} w^r\right)^{1/r}.$$
	Thus, by Theorem~\ref{thm:weightedwp}, 
	$$
	\int_{\R^n} |f(x)|^2 w(x) {\rm d} x$$ 
	$$\leq C_{\epsilon,N} R^{\epsilon} \left( \sum_{T \in \T} \frac{\|f_T\|_{L^2(\R^n)}^2} {|T|}
	\Big(w^r(2R^\epsilon T)^{1/r} + R^{\frac{(n-N)}{n}}\sum_{m: 2^m \geq R^{\epsilon}} 2^{m(n -N)}{w^r(2^{m} T)^{1/r}}\Big)^{r} \right)^{1/r} \|f\|_{L^2(\R^n)}^{4/p}
	$$
	which is the $l^r$ norm in $T$ of a sum in $m$ and so, by Minkowski's inequality again, is dominated by the sum in $m$ of the $l^r$ norms in $T$, which is 
	$$ 
	C_{\epsilon,N} R^{\epsilon} \Big( \sum_{T \in \T} \|f_T\|_{L^2(\R^n)}^2
	\frac{w^r(2R^\epsilon T)}{|T|}\Big)^{1/r} \|f\|_{L^2(\R^n)}^{4/p}
	$$ 
	$$ +  C_{\epsilon,N} R^{\epsilon} R^{\frac{(n-N)}{n}} \sum_{m: 2^m \geq R^{\epsilon}} 2^{m(n -N)} 2^{mn/r} \Big( \sum_{T \in \T} \|f_T\|_{L^2(\R^n)}^2
	\frac{w^r(2^{m} T)}{|2^mT|}\Big)^{1/r} \|f\|_{L^2(\R^n)}^{4/p}.
	$$
	The first term is as expected, and we write the second term as
	$$C_{\epsilon,N} R^{-K}  R^{\frac{(n-N)}{n} + K + \epsilon} \sum_{m: 2^m \geq R^{\epsilon}} 2^{m[(n -N) +n/r +K]} 2^{-mK} \Big(\sum_{T \in \T} \|f_T\|_{L^2(\R^n)}^2 
	\frac{w^r(2^{m} T)}{|2^mT|}\Big)^{1/r} \|f\|_{L^2(\R^n)}^{4/p}
	$$
	where $K$ is as in the statement of the theorem. Finally, choose $N$ sufficiently large so that this last expression is dominated by 
	$$R^{-K} \sup_{m: 2^m \geq R^{\epsilon}} 2^{-mK}  \Big( \sum_{T \in \T} \|f_T\|_{L^2(\R^n)}^2 
	\frac{w^r(2^{m} T)}{|2^mT|}\Big)^{1/r} \|f\|_{L^2(\R^n)}^{4/p}.
	$$
	\qed}

\subsection{Proof of Theorem~\ref{thm:refined_dec_theta}}\label{sec:pf_thm_6.2}
Fix %
$c_0 > 1$ and $2 \leq p \leq n(n+1)$. Let $\epsilon > 0$ and $N > n$ be given. 
We set 
\[
C_{\epsilon,N} := \max\{(2 R(\epsilon,N)^{\frac{1}{n}})^{1-\frac{1}{p}}, 2 C_{n,N}, 1\}
\] 
where $C_{n,N}$ is the product of the implicit constants in \eqref{eq:C_nN1} and \eqref{eq:C_nN2} (so that \eqref{eq:insignificant} holds for all $\theta$ insignificant to $Q$), and $$R(\epsilon,N) \geq \max\{ 2^{2n^2}, (2\delta_0(n,c_0)^{-1})^{n^2}\}$$ is a threshold depending only on $\epsilon, N, c_0, n$ and $p$ that will be determined in \eqref{eq:RepN_choice1} and \eqref{eq:Repsilon1} below. We can then write down a proposition $P(R)$ which says: \\

\emph{For every $\Gamma \in \mathcal{C}(c_0)$ and every $f=\sum_{\theta\in\Theta_{\Gamma}(R^{-1/n})} f_{\theta}$ as in the statement of Theorem~\ref{thm:BDG}, for every $M \geq 1$, if $Y_M$ is the union of $X$ disjoint cubes $Q$ of side length $R^{1/n}$, such that for each $Q \subset Y_M$, there are at most $M$ parallepipeds $\theta \in \Theta_{\Gamma}(R^{-1/n})$ for which \eqref{eq:significant_thetas} hold, then \eqref{eq:refined_dec_theta} holds.} \\

We shall prove, via induction on the scale $R$, that $P(R)$ holds for all $R \geq 1$. In fact, we shall define a sequence of scales by $R_1 := R(\epsilon,N)$ and $R_k := (2^{-n} R_{k-1})^{\frac{n}{n-1}}$ for $k \geq 2$; note that the hypothesis that $R(\epsilon,N) \geq 2^{2n^2}$ guarantees that $R_k \to \infty$ as $k \to \infty$ (it guarantees that $R_k \geq 2^{2n^2}$ for all $k \geq 1$ and hence $$R_{k+1} \geq 2^{-\frac{n^2}{n-1}} (2^{2n^2})^{\frac{1}{n-1}} R_k = 2^{\frac{n^2}{n-1}} R_k$$ for all $k \geq 1$). Hence we only need to prove by induction on $k$ that $P(R)$ holds for all $1 \leq R \leq R_k$.

The base case $k = 1$ holds true by our definition of $C_{\epsilon,N}$: in fact, 
\eqref{eq:refined_dec_theta} holds by H\"{o}lder's inequality for all $R \leq R(\epsilon,N)$ since the number of $\theta$'s is at most $2 R(\epsilon, N)^{1/n}$ and $(2 R(\epsilon,N)^{\frac{1}{n}})^{1-\frac{1}{p}} \leq C_{\epsilon,N}$. 

Suppose now proposition $P(R)$ had been proved for all $1 \leq R \leq R_{k-1}$, with $k \geq 2$, and let $R \in (R_{k-1}, R_k]$. Also let $f=\sum_{\theta\in\Theta_{\Gamma}(R^{-1/n})} f_{\theta}$, $M \geq 1$, and $Y_M$ is as in proposition $P(R)$. Our goal is to establish inequality \eqref{eq:refined_dec_theta}. To do so, define 
\[
r := \lceil \log_2 R^{1/n} \rceil, \quad r'' := \lceil \log_2 R^{1/n^2} \rceil, \quad r':=r-r''
\]
and let $R' := 2^{r'n}$ so that 
\begin{equation} \label{eq:R'est}
	R' = (2 \cdot 2^{r-1} 2^{-r''})^n \leq (2 R^{1/n} R^{-1/n^2})^n = 2^{n} R^{\frac{n-1}{n}} \leq 2^{n} R_k^{\frac{n-1}{n}} = R_{k-1}.
\end{equation}
Every $\beta \in \Theta_{\Gamma}(R^{-1/n^2})$ corresponds to an interval of length $2^{-r''}$, which can be partitioned into $2^{-r''}/2^{-r}  = 2^{r'} = (R')^{1/n}$ many intervals of length $2^{-r}$. Each subinterval of length $2^{-r}$ corresponds to some $\theta \in \Theta_{\Gamma}(R^{-1/n})$. In that case we say $\theta$ is a child of $\beta$, and write $\theta \subset \beta$. %
Given $\beta \in \Theta_{\Gamma}(R^{-1/n^2})$ and dyadic number $M' \in 2^{\mathbb{N}} \cap [2,2(R')^{1/n}]$, let $Y_{\beta,M'}$ be the union of all $Q \subset Y_M$ such that  the number of $\theta \subset \beta$ satisfying \eqref{eq:significant_thetas} lies in $[M'/2, M')$. 
We also let $Y_{\beta,0}$ be the union of all $Q \subset Y_M$ so that no $\theta \subset \beta$ satisfies \eqref{eq:significant_thetas}. For each $\beta$, 
\[
Y_M = Y_{\beta,0} \sqcup \bigsqcup_{M'} Y_{\beta,M'}
\] 
(henceforth any union or sum over $M'$ is over all dyadic numbers $M' \in 2^{\mathbb{N}} \cap [2,2(R')^{1/n}]$), so
\[
f 1_{Y_M} = \sum_{\beta} \sum_{\theta \subset \beta} f_{\theta} 1_{Y_M} = \sum_{\beta} \sum_{\theta \subset \beta} f_{\theta} \Big( 1_{Y_{\beta,0}} + \sum_{M'}  1_{Y_{\beta,M'}} \Big).
\]

We have a trivial bound
\begin{align*}
	\Big\| \sum_{\beta} \sum_{\theta \subset \beta} f_{\theta} 1_{Y_{\beta,0}} \Big\|_{L^p} 
	&\leq \sum_{\beta} \sum_{\theta \subset \beta} \sum_{Q \subset  Y_{\beta,0}} \|f_{\theta}\|_{L^p(Q)} \\
	&\leq \sum_{\beta} \sum_{\theta \subset \beta} \sum_{Q \subset  Y_{\beta,0}} |Q|^{1/p} \|f_{\theta}\|_{L^{\infty}(Q)} \\
	&\leq C_{n,N} \sum_{\beta} \sum_{\theta \subset \beta} \#( Q \subset  Y_{\beta,0}) R^{1/p} R^{-1} X^{-1} \|f_{\theta}\|_{L^p(\R^n)}.
\end{align*}
In the last inequality we used that $|Q| = R$, and that $Q \subset Y_{\beta,0}$  implies that none of the $\theta$'s that are contained in $\beta$ satisfies the lower bound on $\sup_{x \in Q} |f_{\theta}|*w_{\theta,N}(x)$ in \eqref{eq:significant_thetas}, so that the upper bound \eqref{eq:insignificant} for $\|f_{\theta}\|_{L^{\infty}(Q)}$ is valid for all the $\theta$'s that are contained in $\beta$. We now observe that the number of $Q \subset Y_{\beta,0}$ is at most $X$, and combine $\sum_{\beta} \sum_{\theta \subset \beta}$ into $\sum_{\theta}$ before applying H\"{o}lder's inequality for the resulting sum that has $2^r \leq 2 R^{1/n} \leq 2 R$ terms. We obtain
\begin{align*}
	\Big\| \sum_{\beta} \sum_{\theta \subset \beta} f_{\theta} 1_{Y_{\beta,0}} \Big\|_{L^p} 
	&\leq C_{n,N} R^{1/p} R^{-1} (2R)^{1-\frac{1}{p}} \Big( \sum_{\theta} \|f_{\theta}\|_{L^p(\R^n)}^p \Big)^{1/p} \\
	&\leq 2 C_{n,N} \Big( \sum_{\theta} \|f_{\theta}\|_{L^p(\R^n)}^p \Big)^{1/p}.
\end{align*}
(We just trivially bound $2^{1-\frac{1}{p}}$ by $2$ in the last inequality.)

Hence 
\[
\|f\|_{L^p(Y_M)} \leq 2 C_{n,N} \Big( \sum_{\theta} \|f_{\theta}\|_{L^p(\R^n)}^p \Big)^{1/p} + \sum_{M'} \Big\| \sum_{\beta} \sum_{\theta \subset \beta} f_{\theta} 1_{Y_{\beta,M'}} \Big\|_{L^p}
\]
Now temporarily fix $M'$ and $\beta$. For each $Q \subset Y_{\beta,M'}$, let
\[
\mathfrak{N}(Q,M') = \#\{\beta' \in \Theta_{\Gamma}(R^{-1/n^2}) \colon Q \subset Y_{\beta',M'}\}.
\]
Note then $\mathfrak{N}(Q,M') \in [1, 2^{r''}] \subset [1,2 R^{1/n^2}]$.
For $M'' \in 2^{\mathbb{N}} \cap [2,4R^{1/n^2}]$ let $Y_{\beta,M',M''}$ be the union of all $Q \subset Y_{\beta,M'}$ so that $\mathfrak{N}(Q,M') \in [M''/2, M'')$.
Then
\[
Y_{\beta,M'} = \bigsqcup_{M''} Y_{\beta,M',M''},
\]
so
\[
\sum_{\beta} \sum_{\theta \subset \beta} f_{\theta} 1_{Y_{\beta,M'}} = \sum_{\beta} \sum_{\theta \subset \beta} f_{\theta} \sum_{M''}  1_{Y_{\beta,M',M''}}, 
\]
which yields
\[
\Big\| \sum_{\beta} \sum_{\theta \subset \beta} f_{\theta} 1_{Y_{\beta,M'}} \Big\|_{L^p}
\leq \sum_{M''} \| f_{M',M''} \|_{L^p}
\]
where
\[
f_{M',M''}:=\sum_{\beta} \sum_{\theta \subset \beta} f_{\theta} 1_{Y_{\beta,M',M''}}.
\]
This gives
\[ %
\|f\|_{L^p(Y_M)} \leq 2 C_{n,N} \Big( \sum_{\theta} \|f_{\theta}\|_{L^p(\R^n)}^p \Big)^{1/p} + \sum_{M',M''} \| f_{M',M''} \|_{L^p}.
\] %
If all $f_{M', M''} \equiv 0$ in this sum, then since we chose $C_{\epsilon,N} \geq 2 C_{n,N}$, we would have established goal \eqref{eq:refined_dec_theta}. Thus we assume that some of the $f_{M',M''} \not\equiv 0$.

The number of choices for each of  $M'$ and $M''$ is $\lesssim \log R$. Thus we can pigeonhole (or just maximise over $M'$ and $M''$) and obtain some dyadic numbers $M'$ and $M''$ such that
\begin{equation}  \label{eq:pruned}
	\|f\|_{L^p(Y_M)} \lesssim_{n,N} \Big( \sum_{\theta} \|f_{\theta}\|_{L^p(\R^n)}^p \Big)^{1/p} + (\log R)^2 \| f_{M',M''} \|_{L^p}.
\end{equation}
We fix such $M', M''$ from now on.
Since $f_{M',M''} \not\equiv 0$, we observe that
\begin{equation} \label{eq:M'M''}
	M' M'' \leq 4 M:
\end{equation}
indeed, then there exists some $\beta_0$ so that $Y_{\beta_0,M',M''} \ne \emptyset$. Choose an arbitrary $Q_0 \subset Y_{\beta_0,M',M''}$. We infer that there are at least $M''/2$  $\beta$'s such that $Q_0 \subset Y_{\beta,M'}$, and each such $\beta$ contains $\geq M'/2$ many significant $\theta$'s to $Q_0$. So $\frac{M'}{2} \frac{M''}{2} \leq M$, the total number of significant $\theta$'s for this $Q_0$, as desired. 

We now estimate $\|f_{M',M''}\|_{L^p}$.
The support of $f_{M',M''}$ is covered by cubes $Q$ of side length $R^{1/n}$ that make up $Y_M$.
For each such $Q$,
\[
f_{M',M''} 1_Q = 1_Q \sum_{\beta \colon Q \subset Y_{\beta,M',M''}} \sum_{\theta \subset \beta} f_{\theta}.
\]
Since $\Gamma \in \mathcal{C}(c_0)$, we apply local decoupling \eqref{eq:BDGlocal} on $Q$ (note $Q$ has side length $R^{1/n}$), and obtain
\begin{equation} \label{eq:f_M'M''Lp2}
	\begin{split}
		&\|f_{M',M''}\|_{L^p(Q)} \\
		&\leq  D_{\epsilon/2} R^{\frac{\epsilon}{2n}} \Big( \sum_{\beta \in \Theta_{\Gamma}(R^{-1/n^2}) \colon Q \subset Y_{\beta,M',M''}} \|\sum_{\theta \subset \beta} f_{\theta}\|_{L^p(\R^n, A_{N'}(1+R^{-\frac{1}{n}}|x-c_Q|)^{-N'} {\rm d} x)}^2 \Big)^{1/2}.
	\end{split}
\end{equation}
Here $D_{\epsilon/2}$ can also depend on $c_0, n$ and $p$, and $N' > n$ is a large parameter to be chosen depending only on $\epsilon$, $N$, $n$ and $p$ (thus the same is true for $A_{N'}$).
We need to bound
\begin{equation} \label{eq:f_M'M''Lp1}
	\|f_{M',M''}\|_{L^p}^p = \sum_{Q \subset Y_M} \|f_{M',M''}\|_{L^p(Q)}^p.
\end{equation}
Let $\epsilon' > 0$ be given by 
\begin{equation} \label{eq:epsilon'_def}
	\epsilon'= \min \Big\{ \frac{p\epsilon}{4n^2}, \frac{1}{2n(n+N)} \Big\}.
\end{equation}
Note that $\epsilon'$ depends only on $\epsilon, N, n$ and $p$.
Inequality \eqref{eq:f_M'M''Lp2} implies
\begin{equation} \label{eq:f_M'M''Lp3}
	\|f_{M',M''}\|_{L^p(Q)} 
	\leq  D_{\epsilon/2} R^{\frac{\epsilon}{2n}} \Big( I(\R^n \setminus R^{\epsilon'} Q) + I(R^{\epsilon'} Q) \Big)
\end{equation}
where for any subset $Z$ of $\R^n$,
\[
I(Z) := \Big( \sum_{\beta \in \Theta_{\Gamma}(R^{-1/n^2}) \colon Q \subset Y_{\beta,M',M''}} \|\sum_{\theta \subset \beta} f_{\theta}\|_{L^p(Z, A_{N'}(1+R^{-\frac{1}{n}}|x-c_Q|)^{-N'} {\rm d} x)}^2 \Big)^{1/2}.
\]
We bound the contributions of $I(\R^n \setminus R^{\epsilon'} Q)$ to \eqref{eq:f_M'M''Lp1} as follows:
\begin{align*}
	\sum_{Q} I(\R^n \setminus R^{\epsilon'} Q)^p
	\leq & \, 2^{r''(\frac{1}{2}-\frac{1}{p})p}  \sum_{Q \subset \R^n} \sum_{\beta} \|\sum_{\theta \subset \beta} f_{\theta}\|_{L^p(\R^n \setminus R^{\epsilon'}Q, A_{N'}(1+R^{-\frac{1}{n}}|x-c_Q|)^{-N'} {\rm d} x)}^p \\
	\leq & \, 2^{r''(\frac{p}{2}-1)} 2^{r'(p-1)} \sum_{Q} \sum_{\theta} \|f_{\theta}\|_{L^p(\R^n \setminus R^{\epsilon'}Q, A_{N'}(1+R^{-\frac{1}{n}}|x-c_Q|)^{-N'} {\rm d} x)}^p \\
	\leq & \, (2R^{1/n})^p \sum_{\theta} \int_{\R^n} |f_{\theta}(x)|^p \sum_{Q \colon x \notin R^{\epsilon'} Q} A_{N'} (1+R^{-\frac{1}{n}} |x-c_Q|)^{-N'} {\rm d} x 
\end{align*}
where we used $2^{r''(\frac{p}{2}-1)} 2^{r'(p-1)} \leq 2^{r''p} 2^{r'p} = 2^{rp} \leq (2 R^{1/n})^p$.
Thus using 
\[
\sum_{k \in \mathbb{Z}^n \colon |k| \geq R^{\epsilon'}} |k|^{-N'} \lesssim_n R^{-\epsilon'(N'-n)} \quad \text{if $N' > n$},
\]
we obtain
\[
\sum_{Q} I(\R^n \setminus R^{\epsilon'} Q)^p 
\lesssim_n 2^p A_{N'} R^{\frac{p}{n}} R^{-\epsilon'(N'-n)} \sum_{\theta} \int_{\R^n} |f_{\theta}(x)|^p {\rm d} x.
\]
We now choose $N'$ large enough so that $\frac{p}{n} - \epsilon'(N'-n) \leq -\frac{\epsilon}{2n}p$. We may do so with $N'$ depending only on $\epsilon, N, n$ and $p$ since $\epsilon' > 0$ also depends only on $\epsilon, N, n$ and $p$. Then the above display is $\lesssim_n 2^p A_{N'} (R^{\frac{\epsilon}{2n}})^{-p}  \sum_{\theta} \int_{\R^n} |f_{\theta}(x)|^p {\rm d} x$, and since $A_{N'}$ depends only on $\epsilon, N, n$ and $p$, we obtain, from \eqref{eq:f_M'M''Lp1} and \eqref{eq:f_M'M''Lp3}, that
\[
\begin{split}
	\|f_{M',M''}\|_{L^p} & \leq D_{\epsilon/2} 
	R^{\frac{\epsilon}{2n}}  \Big[ \Big( \sum_{Q \subset Y_M} I(\R^n \setminus R^{\epsilon'}Q)^p \Big)^{1/p} + \Big( \sum_{Q \subset Y_M} I(R^{\epsilon'}Q)^p \Big)^{1/p} \Big]\\
	& \lesssim_{\epsilon,N,c_0,n,p} \Big( \sum_{\theta} \|f_{\theta}\|_{L^p(\R^n)}^p \Big)^{1/p} + R^{\frac{\epsilon}{2n}}  \Big( \sum_{Q \subset Y_M} \Big( \sum_{\beta \colon Q \subset Y_{\beta,M',M''}} \|\sum_{\theta \subset \beta} f_{\theta}\|_{L^p(R^{\epsilon'} Q)}^2 \Big)^{p/2} \Big)^{1/p}.
\end{split}
\]
On the other hand, since the number of $\beta$'s for which $Q \subset Y_{\beta,M',M''}$ is at most $M''$,
\[
\Big( \sum_{\beta \colon Q \subset Y_{\beta,M',M''}} \|\sum_{\theta \subset \beta} f_{\theta}\|_{L^p(R^{\epsilon'} Q)}^2 \Big)^{1/2} \leq (M'')^{\frac{1}{2}-\frac{1}{p}} \Big( \sum_{\beta \colon Q \subset Y_{\beta,M',M''}} \|\sum_{\theta \subset \beta} f_{\theta}\|_{L^p(R^{\epsilon'} Q)}^p \Big)^{1/p}.
\]
Thus
\begin{equation} \label{eq:f_M'M''3}
	\|f_{M',M''}\|_{L^p} \lesssim_{\epsilon,N,c_0,n,p} \Big( \sum_{\theta} \|f_{\theta}\|_{L^p(\R^n)}^p \Big)^{1/p} + R^{\frac{\epsilon}{2n}} (M'')^{\frac{1}{2}-\frac{1}{p}} \Big( \sum_{\beta} \sum_{Q \subset Y_{\beta,M',M''}} \|\sum_{\theta \subset \beta} f_{\theta}\|_{L^p(R^{\epsilon'} Q)}^p \Big)^{1/p}.
\end{equation}
Since $R^{\epsilon'}Q$ are $R^{n\epsilon'}$ overlapping as $Q$ varies over $Y_{\beta,M',M''}$, the second term above is bounded by
\[
R^{\frac{\epsilon}{2n}} (M'')^{\frac{1}{2}-\frac{1}{p}} \Big( \sum_{\beta} R^{n\epsilon'} \|\sum_{\theta \subset \beta} f_{\theta}\|_{L^p(R^{\epsilon'} Y_{\beta,M',M''})}^p \Big)^{1/p} 
\]
where we abused notation and wrote
\[
R^{\epsilon'} Y_{\beta,M',M''} := \bigcup_{ Q \subset Y_{\beta,M',M''} } R^{\epsilon'} Q.
\]
Suppose $R(\epsilon,N)$ is chosen sufficiently large, so that 
	\begin{equation} \label{eq:RepN_choice1}
		R(\epsilon,N) \geq (2^n C''_{n,N})^{2n}
	\end{equation}
	where $C''_{n,N}$ is a constant depending only on $n$ and $N$, given by the product of the implicit constants in \eqref{eq:X'}, \eqref{eq:Cn'2} and \eqref{eq:Cn'1} below.\\
{\bf Claim:} 
Our induction hypothesis at scale $R'$ would guarantee that
\begin{equation} \label{eq:induct_hypo}
	\|\sum_{\theta \subset \beta} f_{\theta}\|_{L^p(R^{\epsilon'} Y_{\beta,M',M''})} \leq C_{\epsilon,N} (2^n R^{\frac{n-1}{n}})^{\epsilon} (M')^{\frac{1}{2}-\frac{1}{p}} \Big( \sum_{\theta \subset \beta} \|f_{\theta}\|_{L^p(\R^n)}^p \Big)^{1/p}.
\end{equation}
In that case, remembering that $\frac{n\epsilon'}{p}\leq\frac{\epsilon}{4n}$ (choice of $\epsilon'$ from \eqref{eq:epsilon'_def}) and $M'M'' \leq 4M$ (from \eqref{eq:M'M''}), the considerations starting from \eqref{eq:f_M'M''3} imply
\[
\|f_{M',M''}\|_{L^p} \lesssim_{\epsilon,N,c_0,n,p} \Big( \sum_{\theta} \|f_{\theta}\|_{L^p(\R^n)}^p \Big)^{1/p} +  C_{\epsilon,N}  R^{(1-\frac{1}{4n})\epsilon} M^{\frac{1}{2}-\frac{1}{p}}  \Big( \sum_{\theta} \|f_{\theta}\|_{L^p(\R^n)}^p \Big)^{1/p}.
\]
(The power of $R$ comes from $\frac{\epsilon}{2n} + \frac{\epsilon}{4n} + \frac{n-1}{n} \epsilon = (1-\frac{1}{4n})\epsilon$, and the constant 
$2^{n\epsilon}$ in \eqref{eq:induct_hypo} had been absorbed into the implicit constant in the above display.) The first term of the above display is actually smaller than the second term, since we chose $C_{\epsilon,N} \geq 1$ and we have $R \geq 1$, $M \geq 1$.
From \eqref{eq:pruned}, we then have
\begin{equation} \label{eq:Aepsnp}
	\|f\|_{L^p(Y_M)} \leq A_{\epsilon,N,c_0,n,p} (\log R)^2 C_{\epsilon,N}  R^{(1-\frac{1}{4n}) \epsilon} M^{\frac{1}{2}-\frac{1}{p}} \Big( \sum_{\theta} \|f_{\theta}\|_{L^p(\R^n)}^p \Big)^{1/p}
\end{equation}
for some constant $A_{\epsilon,N,c_0,n,p}$ that can be computed once we know $\epsilon, N, c_0, n$ and $p$. 

It remains to choose $R(\epsilon,N)$ sufficiently large again, depending only on $\epsilon, N, c_0, n$ and $p$, so that
\begin{equation} \label{eq:Repsilon1}
	(\log R)^2 \leq R^{\frac{\epsilon}{8n}} \quad \text{and} \quad A_{\epsilon,N,c_0,n,p} \leq R^{\frac{\epsilon}{8n}}
\end{equation}
for all $R \geq R(\epsilon,N)$, with $A_{\epsilon,N,c_0,n,p}$ given in \eqref{eq:Aepsnp}. Then we establish goal \eqref{eq:refined_dec_theta} and close the induction, modulo the proof of {\bf Claim}.

To prove {\bf Claim}, we cover $R^{\epsilon'} Y_{\beta,M',M''} = \bigcup_{ Q \subset Y_{\beta,M',M''} } R^{\epsilon'} Q$ by $T_{\beta}^{-1} Q'$ where $Q'$ are disjoint cubes with side length $(R')^{1/n} = 2^{r'}$. 
We only consider those $Q'$ for which $T_{\beta}^{-1} Q'$ intersects $R^{\epsilon'} Q$ for some $Q \subset Y_{\beta,M',M''}$. Let $X'$ be the number of such $Q'$. We have (a very crude estimate) 
\begin{equation} \label{eq:X'}
	X' \lesssim_n R^{n\epsilon'} X.
\end{equation} 
In fact, if $Q'$ is a cube of side length $2^{r'}$, then $T_{\beta}^{-1}Q'$ are parallelepipeds with eccentricity $\sim_n 1$ of dimensions $2^{r'+r''} \times 2^{r'+2r''} \times \dots \times 2^{r'+nr''}$, and in particular if $T_{\beta}^{-1}Q'$ covers a point $x \in \R^n$, then $2T_{\beta}^{-1}Q'$ covers a ball of radius $\sim 2^{r'+r''} = 2^r \sim R^{1/n}$ centred at $x$. Thus each of the $\leq X R^{n\epsilon'}$ cubes of side length $R^{1/n}$ that makes up $R^{\epsilon'} Y_{\beta,M',M''}$ intersects $\lesssim_n 1$  of the $T_{\beta}^{-1} Q'$. We have
\[
\|\sum_{\theta \subset \beta} f_{\theta}\|_{L^p(R^{\epsilon} Y_{\beta,M',M''})} \leq \Big( \sum_{Q'} \|\sum_{\theta \subset \beta}  f_{\theta}\|_{L^p(T_{\beta}^{-1}Q')}^p \Big)^{1/p} 
\]
where the sum over $Q'$ is always over the $X'$ different cubes we used in the above covering. Now since $\beta \in \Theta_{\Gamma}(R^{-1/n^2}) = \Theta_{\Gamma}(2^{-r''})$ and $\theta \in \Theta_{\Gamma}(R^{-1/n}) = \Theta_{\Gamma}(2^{-r})$, those $\theta \subset \beta$ are in one-to-one correspondence with $\theta' \in  \Theta_{\tilde{\Gamma}}(2^{-r}/2^{-r''}) = \Theta_{\tilde{\Gamma}}(2^{-r'}) = \Theta_{\tilde{\Gamma}}((R')^{-1/n})$ via $\theta \mapsto \theta' := A_{\beta}^{-1} \theta$, where $\tilde{\Gamma} \in \mathcal{C}(c_0)$ as per the discussion preceding \eqref{eq:delta0_def}, since $2^{-r''} \leq 2R^{-1/n^2} \leq 2 R_1^{-1/n^2} = 2 R(\epsilon,N)^{-1/n^2} \leq \delta_0(n,c_0)$ by our choice of $R(\epsilon,N)$. If $\theta \subset \beta$ and $\theta' := A_{\beta}^{-1} \theta$, define
\[
F_{\theta'} := e^{-2\pi i \Gamma(\beta) \cdot T_{\beta}^{-1} x} f_{\theta} \circ T_{\beta}^{-1}(x).
\]
Then the Fourier transform of $F_{\theta'}$ is equal to $$|\det T_{\beta}| \widehat{f_{\theta}} (\Gamma(\beta) + T_{\beta}^t \xi) = |\det T_{\beta}| \widehat{f_{\theta}} \circ A_{\beta}(\xi),$$ which is supported on $A_{\beta}^{-1} \theta = \theta'$, and
\[
\Big( \sum_{Q'} \|\sum_{\theta \subset \beta}  f_{\theta}\|_{L^p(T_{\beta}^{-1}Q')}^p \Big)^{1/p} =|\det T_{\beta}^{-1}|^{1/p} \Big( \sum_{Q'} \|\sum_{\theta' \in \Theta_{\tilde{\Gamma}}((R')^{-1/n})} F_{\theta'} \|_{L^p(Q')}^p \Big)^{1/p}.
\]
We will show that 
at most $M'$ of the $\theta'$ satisfy 
\begin{equation} \label{eq:induct_geom}
	\sup_{x' \in Q'} |F_{\theta'}|*w_{\theta',N}(x') > (R')^{-1} (X')^{-1}  \|F_{\theta'}\|_{L^{\infty}(\R^n)}.
\end{equation}
Then by induction hypothesis (recall we assumed proposition $P(R')$ holds for $R' \leq R_{k-1}$),
\begin{align*}
	&\quad |\det T_{\beta}^{-1}|^{1/p} \Big( \sum_{Q'} \|\sum_{\theta' \in \Theta_{\tilde{\Gamma}}((R')^{-1/n})} F_{\theta'} \|_{L^p(Q')}^p \Big)^{1/p}  \\
	&\leq C_{\epsilon,N} (R')^{\epsilon} (M')^{\frac{1}{2}-\frac{1}{p}} |\det T_{\beta}^{-1}|^{1/p} \Big( \sum_{\theta' \in \Theta_{\tilde{\Gamma}}((R')^{-1/n})} \|F_{\theta'}\|_{L^p(\R^n)}^p \Big)^{1/p}\\
	&\leq C_{\epsilon,N} (2^n R^{\frac{n-1}{n}})^{\epsilon} (M')^{\frac{1}{2}-\frac{1}{p}} \Big( \sum_{\theta \subset \beta} \|f_{\theta}\|_{L^p(\R^n)}^p \Big)^{1/p}
\end{align*}
where we used $R' \leq 2^n R^{\frac{n-1}{n}}$ which we proved in \eqref{eq:R'est}. This establishes \eqref{eq:induct_hypo}.

We turn to the proof that at most $M'$ of the $\theta'$ satisfy \eqref{eq:induct_geom}. Recall $Q'$ are chosen such that $T_{\beta}^{-1} Q'$ intersects $R^{\epsilon'} Q$ for some $Q \subset Y_{\beta,M',M''}$. We fix this $Q$ for the moment.
In fact, recall $R(\epsilon,N)$ was chosen as in \eqref{eq:RepN_choice1}, and $R \geq R(\epsilon,N)$. The main thrust is to prove:

\noindent{\bf Claim$^\prime$:} 
 If \eqref{eq:significant_thetas} fails for $Q$, or equivalently
\begin{equation} \label{eq:not_significant_hypo}
	\sup_{x \in Q} |f_{\theta}|*w_{\theta,N}(x) \leq  R^{-1} X^{-1} \|f_{\theta}\|_{L^{\infty}(\R^n)},
\end{equation}
then %
 \eqref{eq:induct_geom} is not satisfied for $\theta' := A_{\beta}^{-1} \theta$.

Since for fixed $Q \subset Y_{\beta,M',M''}$, at most $M'$ of the $\theta \subset \beta$ satisfy \eqref{eq:significant_thetas}, we also know for any of the $Q'$ we considered, at most $M'$ of the $\theta'$ satisfy \eqref{eq:induct_geom}, as desired.

To prove {\bf Claim$^\prime$}, first note that by \eqref{eq:w_compare},
\begin{equation}
	\sup_{x' \in Q'} |F_{\theta'}|*w_{\theta',N}(x') 
	\lesssim_{n,N} \inf_{x' \in Q'} |F_{\theta'}|*w_{\theta',N}(x'). \label{eq:Cn'2}    
\end{equation}
Since $T_{\beta}^{-1}Q'$ intersects $R^{\epsilon'}Q$, and since $|F_{\theta'}|*w_{\theta',N}(x') = (|f_{\theta}|*w_{\theta,N})(T_{\beta}^{-1} x')$, we have
\begin{align}
	\inf_{x' \in Q'} |F_{\theta'}|*w_{\theta',N}(x') \leq \sup_{x \in R^{\epsilon'} Q} |f_{\theta}|*w_{\theta,N}(x). \notag
\end{align}
Now if $|y| \lesssim_n R^{\frac{1}{n}+\epsilon'}$, then
\[
w_{\theta,N}(x+y) \lesssim_{n,N} R^{N\epsilon'} w_{\theta,N}(x) \quad \text{for all $x \in \R^n$},
\]
which implies
\[
|f_{\theta}|*w_{\theta,N}(x+y) \lesssim_{n,N} R^{N\epsilon'} |f_{\theta}|*w_{\theta,N}(x) \quad \text{for all $x \in \R^n$}.
\]
Hence
\begin{align}
	\sup_{x \in R^{\epsilon'} Q} |f_{\theta}|*w_{\theta,N}(x)
	& \lesssim_{n,N}  R^{N\epsilon'} \sup_{x \in Q} |f_{\theta}|*w_{\theta,N}(x) \label{eq:Cn'1} \\
	& \leq  R^{N\epsilon'} R^{-1} X^{-1} \|f_{\theta}\|_{L^{\infty}(\R^n)}, \notag
\end{align}
where we used \eqref{eq:not_significant_hypo} in the last inequality. 
Now $\|f_{\theta}\|_{L^{\infty}(\R^n)} = \|F_{\theta'}\|_{L^{\infty}(\R^n)}$.
In addition, \eqref{eq:X'} says $X^{-1} \lesssim_n R^{n \epsilon'} (X')^{-1}$. Thus
\begin{align*}
	\sup_{x' \in Q'} |F_{\theta'}|*w_{\theta',N}(x') 
	&\leq C''_{n,N} R^{(n+N)\epsilon'} R^{-1} (X')^{-1}  \|F_{\theta'}\|_{L^{\infty}(\R^n)} %
\end{align*}
where $C''_{n,N}$ is the product of the implicit constants in \eqref{eq:X'}, \eqref{eq:Cn'2} and \eqref{eq:Cn'1}.
Recalling $(n+N)\epsilon' \leq  \frac{1}{2n} $ from \eqref{eq:epsilon'_def}, we have, using \eqref{eq:R'est}, that
\[
C''_{n,N} R^{(n+N)\epsilon'} R^{-1} \leq C''_{n,N} R^{-\frac{1}{2n}}  R^{-\frac{n-1}{n}} \leq C''_{n,N} 2^n R^{-\frac{1}{2n}} (R')^{-1}.
\]
Hence remembering \eqref{eq:RepN_choice1} and $R \geq R(\epsilon,N)$, we have
\[
\sup_{x' \in Q'} |F_{\theta'}|*w_{\theta',N}(x') \leq  (R')^{\-1} (X')^{-1}  \|F_{\theta'}\|_{L^{\infty}(\R^n)},
\]
and \eqref{eq:induct_geom} is not satisfied. This finishes the proof of {\bf Claim$^\prime$}, and with that the proof of Theorem~\ref{thm:refined_dec_theta}.

\subsection{Equivalence between Theorems \ref{thm:refined_dec_wp_prior} and \ref{main_corrected_reprise}}\label{sec:duality}

Finally, we show that  Theorem~\ref{main_corrected_reprise} implies Theorem~\ref{thm:refined_dec_wp_prior} by simply testing on a suitable weight $w$. 

Assume the hypotheses of Theorem~\ref{thm:refined_dec_wp_prior} hold. First of all we may dyadically pigeonhole to reduce to the the case in which 
$\|f_T\|_{L^p(\R^n)}^2 \sim |T|^{-1/r} \|f_T\|_{L^2(\R^n)}^2$ is roughly constant over all wave packet tubes $T \in \mathbb{T}(B)$, and we shall now assume we are in that case. Secondly, since $\|f_T\|_{L^p(\R^n)}^2 \sim |T|^{-1/r} \|f_T\|_{L^2(\R^n)}^2$, it suffices to prove that
\begin{equation} \label{eq:refined_dec_wp_prior_post}
\Big\| \sum_{T\in \T(B)} f_T \Big\|_{L^p(Y_M)} \lesssim_{\epsilon} R^{\epsilon} M^{\frac{1}{2r}} |T|^{-\frac{1}{2r}} \Big( \sum_{T \in \T(B)} \|f_T\|_{L^2(\R^n)}^p \Big)^{1/p}.
\end{equation}

Indeed, under the hypotheses of Theorem~\ref{thm:refined_dec_wp_prior}, we apply Theorem~\ref{main_corrected_reprise} with 
$f = \sum_{T \in \mathbb{T}(B)}  f_T$ 
and $w \in L^{r}(Y_M)$ of unit norm, and we have
\begin{equation*}
\begin{split}\int_{\R^n} |f(x)|^2 & w(x) {\rm d} x \leq C_{\epsilon,K} R^{\epsilon} \Big( \sum_{T \in \mathbb{T}(B)} \|f_T\|_{L^2(\R^n)}^2  \frac{w^r(2 R^{\epsilon} T)}{|T|} \Big)^{1/r} \|f\|_{L^2(\R^n)}^{4/p} \\
+ & C_{\epsilon,K} R^{-K}\sup_{m\geq 1} 2^{-mK} \Big(\sum_{T \in \mathbb{T}(B)} \|f_T\|_{L^2(\R^n)}^2 \frac{w^r(2^m T)}{|2^mT|} \Big)^{1/r} \|f\|_{L^2(\R^n)}^{4/p}.
\end{split}
\end{equation*}
We will deal with the main term here; the error term is handled similarly, with plenty of room to spare, due to the exponential decay. Letting $\lambda$ denote the common value of $\|f_T\|_2^2$, and letting $\mathcal{B}$ being a disjoint covering of $Y_M$ by $R^{1/n}$-cubes,
the main term on the right hand side here is essentially 
    $$ C_\epsilon R^\epsilon \Big(\frac{\lambda}{|T|}\Big)^{\frac{1}{r}} \left( \sum_{T \in \mathbb{T}(B)} \sum_{Q\in\mathcal{B}:\;Q\cap 2R^\epsilon T\neq\emptyset}w^{r}(Q) \right)^{\frac{1}{r}} \big(\sum_{T \in \mathbb{T}(B)} \|f_T\|_2^2\big)^{\frac{2}{p}}$$
$$ \sim C_\epsilon R^\epsilon |T|^{-\frac{1}{r}} \left( \sum_{T \in \mathbb{T}(B)} \sum_{Q\in\mathcal{B}:\;Q\cap 2R^\epsilon T\neq\emptyset}w^{r}(Q) \right)^{\frac{1}{r}} \big(\sum_{T \in \mathbb{T}(B)} \|f_T\|_2^{\frac{p}{r}}\|f_T\|_2^2\big)^{\frac{2}{p}}$$
$$=  C_\epsilon R^\epsilon |T|^{-\frac{1}{r}} \left( \sum_{Q\in\mathcal{B}} w^{r}(Q) \; \#\{T \in \mathbb{T}(B)\, : \, Q\cap 2R^\epsilon T\neq\emptyset\}  \right)^{\frac{1}{r}} \big(\sum_{T \in \mathbb{T}(B)} \|f_T\|_2^p\big)^{\frac{2}{p}}$$
$$\lesssim  C_\epsilon R^\epsilon |T|^{-\frac{1}{r}} M^{\frac{1}{r}}\left( \sum_{Q\in\mathcal{B}} w^{r}(Q) \; 
\right)^{\frac{1}{r}} \big(\sum_{T \in \mathbb{T}(B)} \|f_T\|_2^p\big)^{\frac{2}{p}},$$
as needed to establish \eqref{eq:refined_dec_wp_prior_post}, and thus verify Theorem~\ref{thm:refined_dec_wp_prior}.

\end{document}